\documentclass[10pt,a4paper,final]{amsart}

\usepackage{amsmath,amsthm,amssymb}
\usepackage{mathtools}
\allowdisplaybreaks

\usepackage[T1]{fontenc}
\usepackage[utf8]{inputenc}
\usepackage{microtype}

\usepackage[paper=a4paper,margin=.7in]{geometry}
\usepackage{setspace}
\onehalfspace
\usepackage{fancyhdr}
\fancypagestyle{plain}{%
  \fancyhead{}
  \fancyfoot[CE,CO]{\thepage}
  \fancyfoot[LE,RO]{}
  \fancyfoot[RE,LO]{}
  
}

\usepackage{subcaption}

\usepackage{tikz}
\usetikzlibrary{shapes,fit,positioning,calc,matrix,cd,decorations.pathmorphing}
\tikzset{
  PRTnode/.style={draw=none,circle,fill,line width=0,inner sep=-1pt,minimum width=.5ex},
  PRTmark/.style={draw,ellipse,fill=none,thick,inner sep=1ex},
  PRT/.style={grow'=up,thick,level distance=3ex,sibling distance=1em,every node/.style={PRTnode}},
}

\usepackage[backend=biber,style=numeric-comp,sorting=nyt]{biblatex}
\usepackage[inline]{enumitem}

\usepackage[final,bookmarksopen,bookmarksdepth=2,pdfusetitle]{hyperref}

\usepackage[noabbrev]{cleveref}

\usepackage[all]{cmath}

\setlist[enumerate,1]{label=(\emph{\roman*})}
\newlist{conditions}{enumerate}{1}
\setlist[conditions]{label=(\emph{\roman*})}
\crefname{conditionsi}{condition}{conditions}

\theoremstyle{plain}
\newtheorem*{pp:KosCx:acyclic}{Proposition \ref{pp:KosCx:acyclic}}
\newtheorem*{pp:ELie3:HTT}{Proposition \ref{pp:ELie3:HTT}}
\newtheorem*{lm:SS:Phi}{Lemma \ref{lm:SS:Phi}}

\makeatletter
\def\@endtheorem{\endtrivlist}%
\makeatother

\makeatletter
\newcommand{\addresseshere}{%
  \enddoc@text\let\enddoc@text\relax
}
\makeatother

\newcommand{\diff}{\mathrm{d}}
\newcommand{\dEnd}{\partial}
\newcommand{\rSh}{\overline{\mathrm{Sh}}}
\newcommand{\uSh}{\Sh^{-1}}
\newcommand{\ruSh}{\rSh^{-1}}
\newcommand{\VF}{\mathfrak{X}}
\DeclareMathOperator{\linspan}{span}
\newcommand{\nKd}{\diamond}
\newcommand{\LieD}[1][]{{\Lie^\nKd_{#1}}}
\newcommand{\qiso}{\sim}
\newcommand{\NN}{\mathbb{N}}
\newcommand{\ZZ}{\mathbb{Z}}
\newcommand{\opdO}{\mathcal{O}}
\newcommand{\opdP}{\mathcal{P}}
\newcommand{\opdQ}{\mathcal{Q}}
\newcommand{\coopdC}{\mathcal{C}}
\newcommand{\smodI}{\mathrm{I}}
\newcommand{\SSy}[1]{\widebar{#1}}
\newcommand{\dprime}{\prime\prime}

\makeatletter
\let\save@mathaccent\mathaccent
\newcommand*\if@single[3]{%
  \setbox0\hbox{${\mathaccent"0362{#1}}^H$}%
  \setbox2\hbox{${\mathaccent"0362{\kern0pt#1}}^H$}%
  \ifdim\ht0=\ht2 #3\else #2\fi
  }
\newcommand*\rel@kern[1]{\kern#1\dimexpr\macc@kerna}
\newcommand*\widebar[1]{\@ifnextchar^{{\wide@bar{#1}{0}}}{\wide@bar{#1}{1}}}
\newcommand*\wide@bar[2]{\if@single{#1}{\wide@bar@{#1}{#2}{1}}{\wide@bar@{#1}{#2}{2}}}
\newcommand*\wide@bar@[3]{%
  \begingroup
  \def\mathaccent##1##2{%
    \let\mathaccent\save@mathaccent
    \if#32 \let\macc@nucleus\first@char \fi
    \setbox\z@\hbox{$\macc@style{\macc@nucleus}_{}$}%
    \setbox\tw@\hbox{$\macc@style{\macc@nucleus}{}_{}$}%
    \dimen@\wd\tw@
    \advance\dimen@-\wd\z@
    \divide\dimen@ 3
    \@tempdima\wd\tw@
    \advance\@tempdima-\scriptspace
    \divide\@tempdima 10
    \advance\dimen@-\@tempdima
    \ifdim\dimen@>\z@ \dimen@0pt\fi
    \rel@kern{0.6}\kern-\dimen@
    \if#31
      \overline{\rel@kern{-0.6}\kern\dimen@\macc@nucleus\rel@kern{0.4}\kern\dimen@}%
      \advance\dimen@0.4\dimexpr\macc@kerna
      \let\final@kern#2%
      \ifdim\dimen@<\z@ \let\final@kern1\fi
      \if\final@kern1 \kern-\dimen@\fi
    \else
      \overline{\rel@kern{-0.6}\kern\dimen@#1}%
    \fi
  }%
  \macc@depth\@ne
  \let\math@bgroup\@empty \let\math@egroup\macc@set@skewchar
  \mathsurround\z@ \frozen@everymath{\mathgroup\macc@group\relax}%
  \macc@set@skewchar\relax
  \let\mathaccentV\macc@nested@a
  \if#31
    \macc@nested@a\relax111{#1}%
  \else
    \def\gobble@till@marker##1\endmarker{}%
    \futurelet\first@char\gobble@till@marker#1\endmarker
    \ifcat\noexpand\first@char A\else
      \def\first@char{}%
    \fi
    \macc@nested@a\relax111{\first@char}%
  \fi
  \endgroup
}
\makeatother

\makeatletter
\newcommand{\bcirc}{\mathrel{\vphantom{|}\mathpalette\do@bcirc\relax}}
\newcommand{\do@bcirc}[2]{%
  \ooalign{%
    $#1\m@th\circ$\cr
    \hidewidth$#1\m@th|$\hidewidth\cr
  }%
}
\makeatother

\newcommand{\pcirc}{{\circ_{(1)}}}

\newcommand{\QAlg}{\opdQ\cat{-Alg}}
\DeclareMathOperator{\Ho}{Ho}
\DeclareMathOperator{\coker}{coker}
\newcommand{\HoQAlg}{\Ho\opdQ\cat{-Alg}}
\newcommand{\hty}{\sim_h}

\makeatletter
\newcommand{\LieKgen}[1]{%
  \def\LieKgen@args{{#1}}%
  \LieKgen@
}
\newcommand{\LieKgen@}[1][]{%
  \expandafter\LieKgenerator\LieKgen@args{#1}%
}
\newcommand{\LieDgen}[1]{%
  \def\LieDgen@args{{#1}}%
  \LieDgen@
}
\newcommand{\LieDgen@}[1][]{%
  \expandafter\LieDgenerator\LieDgen@args{#1}%
}
\newcommand{\LeibKstr}[1]{%
  \def\LeibKstr@args{{#1}}%
  \LeibKstr@
}
\newcommand{\LeibKstr@}[1][]{%
  \expandafter\LeibKstructure\LeibKstr@args{#1}%
}
\newcommand{\LieKstr}[1]{%
  \def\LieKstr@args{{#1}}%
  \LieKstr@
}
\newcommand{\LieKstr@}[1][]{%
  \expandafter\LieKstructure\LieKstr@args{#1}%
}
\newcommand{\LieDstr}[1]{%
  \def\LieDstr@args{{#1}}%
  \LieDstr@
}
\newcommand{\LieDstr@}[1][]{%
  \expandafter\LieDstructure\LieDstr@args{#1}%
}
\makeatother

\newcommand{\LeibKtw}{\lambda}
\newcommand{\LieKtw}{\SSy{\lambda}}
\newcommand{\LieDtw}{\lambda}
\newcommand{\LeibKgen}[1]{\LieDgen{#1}}
\newcommand{\LieKgenerator}[2]{\bar{l}_{#1}^{#2}}
\newcommand{\LieDgenerator}[2]{l_{#1}^{#2}}
\newcommand{\LeibKstructure}[2]{\LeibKtw_{#1}^{#2}}
\newcommand{\LieKstructure}[2]{\LieKtw_{#1}^{#2}}
\newcommand{\LieDstructure}[2]{\LieDtw_{#1}^{#2}}

\DeclareRobustCommand{\SkipTocEntry}[9]{}

\begin{document}

  \title{On weak Lie 3-algebras}
  \author{Malte Dehling}

  \address{Mathematisches Institut,
  Georg-August-Universit\"at G\"ottingen,
  Bunsenstra\ss{}e 3-5,
  37073 G\"ottingen, Germany}
  \email{malte.dehling@uni-goettingen.de}
  \urladdr{http://www.uni-math.gwdg.de/mdehling/}

  \thanks{This work was supported by the ANR grant SAT}


  \begin{abstract}
    In this article, we introduce a category of \emph{weak Lie 3-algebras} with suitable \emph{weak morphisms}.  The
    definition is based on the construction of a partial resolution over $\ZZ$ of the Koszul dual cooperad $\LieK$ of
    the $\Lie$ operad, with free symmetric group action.  Weak Lie 3-algebras and their morphisms are then defined via
    the usual operadic approach---as solutions to Maurer--Cartan equations.  As 2-term truncations we recover
    Roytenberg's category of \emph{weak Lie 2-algebras}.  We prove a version of the \emph{homotopy transfer theorem}
    for weak Lie 3-algebras.  A right homotopy inverse to the resolution is constructed and leads to a
    \emph{skew-symmetrization construction} from weak Lie 3-algebras to 3-term $\Linf$-algebras.  Finally, we give two
    applications: the first is an extension of a result of Rogers comparing algebraic structures related to
    $n$-plectic manifolds; the second is the construction of a weak Lie 3-algebra associated to an CLWX 2-algebroid
    leading to a new proof of a result of Liu--Sheng.
  \end{abstract}

  \maketitle

  \thispagestyle{plain}

  \setcounter{tocdepth}{1}
  \tableofcontents


  \section*{Introduction}
    The study of algebraic structures up to homotopy combines the fields of algebra and homotopy theory.  The objects
    of study are types of algebras and their invariance properties with respect to certain homotopy operations on
    their underlying spaces.  In our setting, the underlying base category of spaces is a symmetric monoidal model
    category $\cat{C}$, and the algebraic structures considered are algebras over an operad $\opdP$ in $\cat{C}$.
    Operads model many input, single output operations and their composition and are therefore suitable to describe
    many of the classical types of algebras, e.g.\@ associative, commutative and Lie algebras (see
    \cite{loday2012algebraic,MSS}.)

    In general, algebras over an operad are rigid structures, meaning they do not play nice with homotopy operations
    on the underlying space.  However, for some operads $\opdQ$ their algebras do have good homotopy properties.  This
    is the case in particular for those operads $\opdQ$ that are cofibrant in the model structure on operads in
    $\cat{C}$ (see \cite{Hinich97,Hinich03,BergerMoerdijk03}.)  This model structure exists under some assumptions on
    the underlying model category $\cat{C}$ and some restrictions on the operads, see op.\@ cit.\@ for details.  For
    such a cofibrant operad $\opdQ$, we can also equip its category $\QAlg$ of $\opdQ$-algebras with a model structure
    and in this category a version of the Boardman--Vogt homotopy invariance property holds: given a homotopy
    equivalence of cofibrant-fibrant spaces $X$, $Y$ in $\cat{C}$, a structure of $\opdQ$-algebra on either induces a
    homotopy equivalent $\opdQ$-algebra structure on the other \cite[Theorem~3.5]{BergerMoerdijk03}.

    We will often be interested in the homotopy category of $\opdQ$-algebras, which is defined as the localization
    $\HoQAlg = \QAlg[W^{-1}]$ with respect to the class $W$ of weak equivalences.  An isomorphism $A \to A'$ in the
    homotopy category is a zigzag of weak equivalences in $\QAlg$,
    \begin{equation*}
      \begin{tikzcd}
        A \ar[r,<-,"\qiso"] & \bullet \ar[r,"\qiso"] & \dotsb \ar[r,<-,"\qiso"] & \bullet \ar[r,"\qiso"] & A' .
      \end{tikzcd}
    \end{equation*}
    Given cofibrant resp.\@ fibrant replacement functors $Q$ resp.\@ $R$ on $\QAlg$, it is a consequence of a general
    result on model categories (see e.g.\@ \cite[Theorem~1.2.10]{Hovey99}) that
    \begin{equation*}
      \Hom_{\HoQAlg}(A,A') \iso \Hom_{\QAlg}(QA,RA')/\hty ,
    \end{equation*}
    where the relation $\hty$ is homotopy of morphisms.  We call a morphism $QA \to RA'$ of $\opdQ$-algebras a
    homotopy morphism from $A$ to $A'$ and denote it by $A \hto A'$.

    In this paper, we work in the differential graded framework.  The category $\cat{C}$ will be the category of
    unbounded chain complexes equipped with the standard projective model structure, i.e.\@ weak equivalences are
    quasi-isomorphisms, fibrations are degree-wise epimorphisms, and cofibrations are determined by the left lifting
    property w.r.t.\@ acyclic fibrations.  This model structure can be transferred to give model structures on
    $\Sy$-modules, operads, and algebras over an operad, by defining the weak equivalences resp.\@ fibrations to be
    those maps that are weak equivalences resp.\@ fibrations on all underlying chain complexes.  The cofibrations are
    then again determined by their lifting property.  Note that the lifting property defining cofibrations depends on
    the structures, not just the underlying chain complexes.  In particular, while it is clear from the above
    definition that all operads in $\cat{C}$ are fibrant, cofibrancy is an entirely different question.

    On the category of operads in $\cat{C}$, we have functorial cofibrant resolutions given by the counit
    $\COBAR\BAR\opdP \qito \opdP$ of the cobar-bar adjunction for any operad $\opdP$ that is already $\Sy$-cofibrant,
    i.e.\@ cofibrant in the underlying category of $\Sy$-modules.  When working with chain complexes over a field
    $\Fk$ of characteristic $0$, this $\Sy$-cofibrancy condition is always satisfied.  When $\Fk$ is an arbitrary
    unital commutative ring, one way to obtain an $\Sy$-cofibrant resolution for any operad $\opdP$ is to take the
    arity-wise tensor product with the algebraic Barratt--Eccles operad $\BE$ (see e.g.\@ \cite{BergerFresse04}.) In
    this way, one obtains a cofibrant resolution $\COBAR\BAR(\opdP\tensor\BE) \qito \opdP\tensor\BE \qito \opdP$ of
    any operad.  We give a more conceptual understanding of this resolution in joint work with Bruno Vallette
    \cite{dehling-vallette_symmetric_2015} where we model non-unital operads as algebras over a particular colored
    operad $\opdO$.  A choice of Koszul presentation for this operad $\opdO$, leads to a new type of cobar-bar
    resolution resolving simultaneously operadic composition as well as the symmetric group actions.  We denote the
    new cobar and bar constructions by $\widetilde{\COBAR}$ resp.\@ $\widetilde{\BAR}$.  It is then shown in loc.\@
    cit.\@ that there is an isomorphism of operads
    $\smodI \oplus \widetilde{\COBAR}\widetilde{\BAR}\red{\opdP} \iso \COBAR\BAR(\opdP\tensor\BE)$, where
    $\red{\opdP}$ denotes the operadic augmentation ideal, $\opdP = \red{\opdP} \dsum \smodI$.

    For the classical types of algebras $\opd{Ass}$, $\opd{Com}$, and $\opd{Lie}$, there are well known homotopy
    invariant analogues.  For associative algebras, these are the $\Ainf$-algebras introduced by Stasheff
    \cite{Stasheff63}.  For commutative algebras, there exist the notions of $\Cinf$-algebras, introduced by
    Kadeishvili \cite{Kadeishvili88}, and of $\Einf$-algebras, going back to May \cite{May72} and Boardman--Vogt
    \cite{BoardmanVogt73}.  For Lie algebras, the homotopy invariant notion of $\Linf$-algebras was introduced by
    Hinich--Schechtman \cite{HinichSchechtman93}, see also \cite{drinfeld1988letter}.  The operad $\Ainf$ is cofibrant
    over any unital commutative ring $\Fk$ since the operad $\opd{Ass}$ is already $\Sy$-cofibrant.  The operads
    $\Cinf$ and $\Linf$ are only cofibrant over fields $\Fk$ of characteristic $0$.  The notion of $\Einf$-algebras
    describes algebras over any $\Sy$-cofibrant resolution of $\opd{Com}$, e.g.\@ over the Barratt--Eccles operad
    $\BE$.  A cofibrant $\Einf$ operad is then given by the cobar-bar resolution of the Barratt--Eccles operad,
    $\COBAR\BAR\BE \iso \smodI \oplus \widetilde{\COBAR}\widetilde{\BAR}\red{\opd{Com}}$.

    In the case of the $\Ainf$, $\Linf$ and $\Cinf$ operads, these can be obtained as resolutions
    $\opdP_{\infty} = \COBAR\opdP^\ash \qito \opdP$ for some cooperad $\opdP^\ash$ weakly equivalent to $\BAR\opdP$ in
    a Hinich-type model structure on dg cooperads (see \cite{LeGrignou16}.) The cooperad $\opdP^\ash$ is given by a
    presentation dual to a choice of presentation for $\opdP$.  The presentation is called \emph{Koszul}, if in fact
    $\opdP^\ash \into \BAR\opdP$ is a weak equivalence and therefore
    $\COBAR\opdP^\ash \qito \COBAR\BAR\opdP \qito \opdP$ is a quasi-isomorphism.  By this \emph{Koszul duality}
    approach, it is possible to obtain much smaller resolutions for many operads.
    The restriction that $\opdP$ needs to be $\Sy$-cofibrant, however, still holds.

    Our interest in this article lies in finding a small cofibrant replacement $\ELinf \qito \Lie$ for the $\Lie$
    operad over any unital commutative ring $\Fk$.  Since the operad $\Lie$ is not $\Sy$-cofibrant, we cannot use the
    classical Koszul duality methods.  Within the context of our new cobar-bar adjunction
    $\widetilde{\COBAR}\dashv\widetilde{\BAR}$, however, a Koszul duality approach is not yet available.  We will
    therefore resort to a more ad hoc approach to build our resolution: we start from the usual Koszul dual cooperad
    $\LieK$ and build an $\Sy$-free resolution step by step.  Assume for a moment that we had completed this process,
    i.e.\@ we have an $\Sy$-free resolution $\psi\colon \LieD \qito \LieK$ of dg cooperads.  Since $\LieD$ is
    $\Sy$-free, $\COBAR\LieD$ is then a cofibrant operad.  To show that the composition
    $g_\kappa\circ\COBAR\psi\colon \COBAR\LieD \to \COBAR\LieK \qito \Lie$ forms a cofibrant resolution, we need to
    verify that $\COBAR\psi$ is a quasi-isomorphism or, equivalently, that the twisted composite product
    $\LieD\circ_{\kappa\circ\psi}\Lie$ is acyclic.  In this case, we can apply the standard machinery of algebraic
    operads to obtain a category of $\ELinf$-algebras with homotopy morphisms satisfying a version of the homotopy
    transfer theorem.  We recall the relevant background material in \Cref{S:Prelim}.

    While we do not have a complete $\Sy$-free resolution of dg cooperads as described above, in \Cref{S:LieD} we do
    construct such a resolution in low degrees and show that at least truncated versions of the relevant statements
    hold.  As a first step, we introduce an explicit $\Sy$-free resolution $\psi\colon \LieD[3] \to \LieK$ of dg
    $\Sy$-modules in low degrees, i.e.\@ such that $\Hm_r(\psi)$ are isomorphisms for $r \leq 3$.  In the second step,
    we equip $\LieD$ with a decomposition map turning it into a dg cooperad.  Since Leibniz algebras are essentially
    non-symmetric Lie algebras, it makes sense to use their $\Sy$-free Koszul dual cooperad $\LeibK$ as a starting
    point for both steps.  The higher degrees of $\LieD$ can be viewed as a coherent system of higher homotopies for
    the (missing) skew-symmetry.  We extend the decomposition map of $\LeibK$ to the higher degrees in a way that is
    compatible with the differential by solving systems of linear diophantine equations and verify that it is actually
    coassociative.  Finally we prove the following result, which---while of no immediate consequences---is a necessary
    condition if we intend to extend our low degree resolution to a full cofibrant resolution.
    \begin{pp:KosCx:acyclic}
      The twisted composite product $\LieD[3]\circ_{\kappa\circ\psi}\Lie$ satisfies
      \begin{equation*}
        \Hm_r\big((\LieD[3]\circ_{\kappa\circ\psi}\Lie)(n)\big) = 0 ,
      \end{equation*}
      for all $r \leq 3$ in all arities $n$.
    \end{pp:KosCx:acyclic}

    In \Cref{S:ELie3}, we use the resolution $\LieD[3]$ to define \emph{weak Lie 3-algebras} as
    $\COBAR\LieD[3]$-algebras on a 3-term complex and introduce the corresponding notion of \emph{weak morphisms}.
    Since we used $\LeibK$ as a starting point for our resolution, weak Lie 3-algebras and their morphisms consist of
    extra structure on top of Leibniz 3-algebras and morphisms of such.  We make the definitions explicit in terms of
    structure maps and equations.  We proceed to show explicitly that, given a deformation retract of 3-term chain
    complexes, i.e.\@ chain maps $p$ and $i$ and a chain homotopy $h$ as in
    \begin{align}
      \begin{tikzcd}[ampersand replacement=\&,/tikz/baseline=-0.5ex]
        (L,\diff)
          \arrow[loop left]{l}{h}
          \arrow[transform canvas={yshift=0.5ex}]{r}{p}
        \& (L',\diff')
          \arrow[transform canvas={yshift=-0.5ex}]{l}{i}
      \end{tikzcd} ,
      &&\textrm{such that }
      \begin{cases}
        \id_L - i\circ p = [\diff,h] ,  \\
        \id_{L'} - p\circ i = 0  ,
      \end{cases}
    \end{align}
    the following homotopy transfer property holds.
    \begin{pp:ELie3:HTT}
      Let $(L,\diff,\LieDtw)$ be a weak Lie 3-algebra and let $(L',\diff')$ be a deformation retract of $(L,\diff)$.
      Then $(L',\diff')$ can be equipped with a transferred weak Lie 3-algebra structure in such a way, that the map
      $i$ admits an extension to a weak morphism of weak Lie 3-algebras.
    \end{pp:ELie3:HTT}

    In \cite{baez_higher-dimensional_2004} the notion of a \emph{Lie 2-algebra} is introduced using a very different
    approach known as categorification.  It is then shown that the definition is equivalent to that of a 2-term
    $\Linf$-algebra, i.e.\@ a 2-term chain complex with a binary graded skew-symmetric bracket satisfying the Jacobi
    identity up to homotopy.  In \cite{roytenberg_weak_2007}, the definition of a \emph{weak Lie 2-algebra} is
    introduced, again as a categorification of Lie algebras, this time with the skew-symmetry of the Lie bracket
    relaxed up to homotopy in addition to the Jacobi identity.  Truncating the complexes underlying our weak Lie
    3-algebras to 2-term complexes, we recover Roytenberg's definitions of weak Lie 2-algebras and their weak
    morphisms.  Similarly, we recover his homotopy transfer theorem \cite[Theorem~4.1]{roytenberg_weak_2007} for weak
    Lie 2-algebras as a truncation of \Cref{pp:ELie3:HTT}.

    By construction of $\LieD[3]$, we have morphisms of dg cooperads $\LeibK[3] \into \LieD[3] \onto \LieK[3]$ and
    therefore functors
    \begin{equation*}
      \begin{tikzcd}
        \cat{Leibniz 3-algebras} 
        & \ar[l] \cat{Weak Lie 3-algebras}
        & \ar[l] \cat{Lie 3-algebras}
      \end{tikzcd} .
    \end{equation*}
    While (homotopy) Lie algebras are precisely (homotopy) Leibniz algebras with skew-symmetric structure maps, in
    general skew-symmetrizing the bracket(s) of a (homotopy) Leibniz algebra does not give a (homotopy) Lie algebra.
    Operadically speaking, this means that there is no simple non-trivial morphism of cooperads $\LieK \to \LeibK$.
    Using the higher degree terms in $\LieD[3]$ we can, however, construct a morphism
    $\LieD[3]\rightsquigarrow\LeibK[3]$ of dg cooperads up to homotopy.  Precisely, we will prove the following result
    in \Cref{S:SS}.
    \begin{lm:SS:Phi}
      The morphism $\COBAR\psi$ admits a right inverse, i.e.\@ a morphism $\Phi$ of dg operads
      \begin{align*}
        \begin{tikzcd}[ampersand replacement=\&]
          \COBAR\psi : \COBAR\LieD[3]
            \arrow[transform canvas={yshift=0.5ex}]{r}{}
          \& \COBAR\LieK[3] : \Phi
            \arrow[transform canvas={yshift=-0.5ex}]{l}{}
        \end{tikzcd} ,
        &&\textrm{such that $\COBAR\psi\circ\Phi = \id$.}
      \end{align*}
    \end{lm:SS:Phi}
    This leads to a \emph{skew-symmetrization construction} producing for each weak Lie 3-algebra $(L,\diff,\LieDtw)$
    a (semi-strict) Lie 3-algebra $(L,\diff,\LieKtw)$.  We then introduce an ad hoc definition of skew-symmetrization
    for morphisms of weak Lie 3-algebras and proceed to show that this construction---while \emph{not strictly}
    functorial---is in some sense functorial up to homotopy.

    We end this article with a discussion of applications of our results in higher differential geometry in
    \Cref{S:Appl}.  In \cite{rogers_thesis_2011,baez_hoffnung_rogers_2008} the concept of $n$-plectic manifolds is
    introduced as a higher analogue for symplectic manifolds.  To any $n$-plectic manifold $(M,\omega)$, two algebraic
    structures on an $n$-term truncation of the de Rham complex are associated: an $\Linf$-algebra $\Linf(M,\omega)$
    and a dg Leibniz algebra $\Leib(M,\omega)$ with a certain hidden skew-symmetry.  In the case of 3-plectic
    manifolds, both are examples of weak Lie 3-algebras.  We show that $\Linf(M,\omega)$ and $\Leib(M,\omega)$ are
    isomorphic as such.  The analogous result for 2-plectic manifolds is due to Rogers
    \cite[Appendix~A]{rogers_thesis_2011}.

    In recent work of Liu--Sheng \cite{liu2016qp}, the notion of an \emph{CLWX 2-algebroid} is introduced as a higher
    Courant algebroid, and it is shown that a Lie 3-algebra can be assigned to any CLWX 2-algebroid.  We refine the
    construction to give a weak Lie 3-algebra and show that the Lie 3-algebra constructed in op.\@ cit.\@ is in fact
    the skew-symmetrization of our weak Lie 3-algebra, thereby giving a new proof for \cite[Theorem~3.10]{liu2016qp}.

    \subsection*{Acknowledgements}
      The author would like to thank Chris Rogers and Dmitry Roytenberg for their insightful remarks and Yunhe Sheng
      for his inspiration.  The author would in particular like to thank Chenchang Zhu for many helpful discussions
      and suggestions and Bruno Vallette for his detailed criticism of this article and his hospitality on numerous
      occasions.

  \section{Preliminaries} \label{S:Prelim}
    In this section, we introduce some conventions, fix notation, and recall basic definitions and results of the
    theory of algebraic operads.  For a detailed introduction to the theory of algebraic operads we refer the reader
    to \cite{loday2012algebraic}, upon which this section is heavily based.  This section is organized as follows.  In
    \Cref{S:Prelim:CN}, we introduce some conventions and notation we follow in the remainder of this article. In
    \Cref{S:Prelim:SMod}, we introduce $\Sy$-modules and equip them with the structure of a monoidal category.  In
    \Cref{S:Prelim:Op}, we give the basic definitions of operads and cooperads.  In \Cref{S:Prelim:CoBar}, we define
    the operadic bar and cobar constructions.  In \Cref{S:Prelim:Tw}, we recall the twisting morphism bifunctor.
    \Cref{S:Prelim:Koszul} serves as a quick reminder of Koszul duality for operads.  Finally, \Cref{S:Prelim:htyAlg}
    deals with homotopy algebras and morphisms in general.

    \subsection{Conventions and notation} \label{S:Prelim:CN}
      We denote by $\Fk$ an arbitrary unital commutative ring.  For any computations we will work over the integers
      $\Fk = \ZZ$.  Since $\ZZ$ is the initial object in the category of unital commutative rings, this ensures that
      our results hold over any such ring $\Fk$.  Our chain complexes are $\Z$-graded complexes of $\Fk$-modules.  We
      follow the Koszul sign rule, i.e.\@ whenever symbols $x$, $y$ of homological degree $|x|$ resp.\@ $|y|$ change
      their relative order, a factor $(-1)^{|x||y|}$ is introduced.

      \subsubsection{Suspension}  We denote by $s\Fk$ the chain complex that is $\Fk$ in degree $1$ and zero in all
      other degrees.  For any chain complex $(V,\diff^V)$, we define its \emph{suspension} to be
      $sV\defeq s\Fk\tensor V$ with $\diff^{sV} = 1\tensor\diff^V$.  We denote by $s\colon V \to sV$,
      $v \mapsto sv = s\tensor v$ the \emph{suspension isomorphism}.  \emph{Desuspension} $s^{-1}$ is defined
      similarly.

      \subsubsection{Symmetric group}  We denote by $\Sy_n$ the \emph{symmetric group} on $n$ elements, i.e.\@ the
      group of bijections of the set $\underline{n}=\{1,\dots,n\}$.  We use the notation $\Fk[\Sy_n]$ for the group
      algebra and the (right) \emph{regular representation} of $\Sy_n$.  By $\Fk\cdot\sgn_n$ we denote the
      one-dimensional \emph{signature represenation} of $\Sy_n$, i.e.\@ its underlying module is $\Fk$ and the
      adjacent transpositions $\sigma_i = (i\ i+1)$ act by multiplication with $-1$.  We implicitely extend the
      group representations to representations of the group algebra and write e.g.\@
      $x^{-\sigma+\tau} = -x^\sigma + x^\tau$.

      \subsubsection{Shuffle permutations}  Let $n_1,\dotsc,n_m$ be natural numbers,
      s.t.\@ $n=n_1+\dotsb+n_m$.  We call $\sigma\in\Sy_n$ an \emph{$(n_1,\dotsc,n_m)$-shuffle}, if
      $\sigma(i)<\sigma(i+1)$ for all $1\leq i < n$ except when $i = n_1 + \dotsb + n_j$ for some $1 \leq j < m$.  We
      denote by $\Sh(n_1,\dotsc,n_m) \subset \Sy_n$ the subset of these $(n_1,\dotsc,n_m)$-shuffles.  The shuffles
      $\Sh(n_1,\dotsc,n_m)$ form a set of representatives for the cosets
      $\Sy_n/(\Sy_{n_1}\times\dotsb\times\Sy_{n_m})$.

      We call an $(n_1,\dotsc,n_m)$-shuffle $\sigma$ \emph{reduced}, if $\sigma(n_j) < \sigma(n_{j+1})$ for all
      $1 \leq j < m$.  The set of these reduced shuffles is denoted as $\rSh(n_1,\dotsc,n_m)$.  The inverse of a
      shuffle is called an \emph{unshuffle} and the set of these is denoted by $\Sh^{-1}(n_1,\dotsc,n_m)$ resp.\@
      $\rSh^{-1}(n_1,\dotsc,n_m)$.

    \subsection{\texorpdfstring{$\Sy$-Modules}{S-Modules}} \label{S:Prelim:SMod}
      A \emph{dg $\Sy$-module $M$} consists of a (right) dg $\Fk[\Sy_n]$-module $M(n)$ for each \emph{arity}
      $n \in \NN$.  We sometimes write a dg $\Sy$-module $M$ as a sequence $(M(0), M(1), \dotsc)$.  A \emph{morphism
      of dg $\Sy$-modules $f\colon M \to N$} consists of a morphism $f(n)\colon M(n) \to N(n)$ of dg
      $\Fk[\Sy_n]$-modules for each \emph{arity} $n \in \NN$.  We denote the category of dg $\Sy$-modules by
      $\dgSMod$.  When $M(0) = 0$, we call $M$ \emph{reduced}.

      \subsubsection{Monoidal structure}
      For $(M,\diff^M), (N,\diff^N) \in \dgSMod$, we define their \emph{composite product}
      $(M\circ N,\diff^{M\circ N})$ by
      \begin{equation*}
        (M \circ N)(n) \defeq \Dsum_{m\in\NN} \Dsum_{n=n_1+\dotsb+n_m}
          M(m) \tensor_{\Sy_m} \Ind^{\Sy_n}_{\Sy_{n_1}\times\dotsb\times\Sy_{n_m}}
          \Big(N(n_1) \tensor \dotsb \tensor N(n_m)\Big) ,
      \end{equation*}
      with the differential given by
      \begin{equation*}
        \diff^{M\circ N} \defeq \Dsum_{m\in\NN} \Dsum_{n=n_1+\dotsb+n_m} \diff^M\tensor_{\Sy_m} 1^{\tensor m}
          + 1\tensor_{\Sy_m}\left( \sum_{i=1}^{m} 1^{\tensor (i-1)}\tensor\diff^N\tensor 1^{\tensor(m-i)} \right) .
      \end{equation*}
      The \emph{composite product of morphisms $f\colon M \to M'$, $g\colon N \to N'$} is defined to be the morphism
      $f\circ g\colon M\circ N \to M'\circ N'$ given by
      \begin{equation*}
        f\circ g \defeq \Dsum_{m\in\NN} \Dsum_{n=n_1+\dotsb+n_m} f\tensor_{\Sy_m} g^{\tensor m} .
      \end{equation*}

      Note that, since the shuffles $\Sh(n_1,\dotsc,n_m)$ form a set of representatives for the cosets
      $\Sy_n/(\Sy_{n_1}\times\dotsb\times\Sy_{n_m})$, this composite product admits an expansion
      \begin{align*}
        (M \circ N)(n) &= \Dsum_{m\in\NN} \Dsum_{n=n_1+\dotsb+n_m}
          M(m) \tensor_{\Sy_m} \Big(N(n_1) \tensor \dotsb \tensor N(n_m) \tensor \Fk[\Sh(n_1,\dotsc,n_m)]\Big) ,
      \intertext{and when the $\Sy$-module $N$ is reduced, the $\Sy_m$-action on the right is free and therefore we
      obtain}
        &= \Dsum_{m\in\NN} \Dsum_{n=n_1+\dotsb+n_m}
          M(m) \tensor N(n_1) \tensor \dotsb \tensor N(n_m) \tensor \Fk\Big[\RSh(n_1,\dotsc,n_m)\Big] .
      \end{align*}
      We denote an element $\mu \tensor_{\Sy_m} \nu_1 \tensor \dotsb \tensor \nu_m \tensor \sigma^{-1}$ of $M\circ N$
      by $\mu\circ(\nu_1,\dotsc,\nu_m)^\sigma$.

      The $\Sy$-module $\smodI = (0, \Fk, 0, 0, \dotsc)$ acts as a (two-sided) \emph{unit} w.r.t.\@ the composite
      product.  This structure turns the category of dg $\Sy$-modules into a monoidal category
      $(\dgSMod,\circ,\smodI)$.

      \subsubsection{The linearized composite product}  Consider the composite product $M\circ(N_1 \oplus N_2)$.  We
      denote by $M \circ (N_1; N_2)$ the sub dg $\Sy$-module that is linear in $N_2$, i.e.\@ that is spanned by
      elements $\mu\circ(\nu_1,\dotsc,\nu_m)^\sigma$ where $\nu_i \in N_2$ for exactly one of $\nu_1,\dotsc,\nu_m$ and
      $\nu_j \in N_1$ for $j\neq i$.  We use the notation $M\pcirc N$ instead of $M\circ(\smodI;N)$ since it appears
      so frequently.  We write $\mu\tensor_i\nu$ for $\mu\circ(1,\dotsc,1,\nu,1,\dotsc,1)$ with $\nu$ in $i$-th
      place.

      We introduce two types of linearized composite products of morphisms.  Given $f\colon M \to M'$,
      $g_1\colon N_1 \to N_1'$, and $g_2\colon N_2 \to N_2'$, we define
      \begin{equation*}
        f\circ(g_1;g_2) \defeq \left(
            \begin{tikzcd}
              M\circ(N_1;N_2) \ar[r,hook]
              & M\circ(N_1\dsum N_2) \ar[r,"f\circ(g_1\dsum g_2)"]
              &[1em] M'\circ(N_1'\dsum N_2') \ar[r,two heads]
              & M'\circ(N_1';N_2')
            \end{tikzcd}
          \right) .
      \end{equation*}
      Consider now $f\colon M \to M'$ and $g\colon N \to N'$.  We denote by $f\circ_{(1)}g$ the morphism
      $f\circ(1;g)$, i.e.\@ $\big(f\circ_{(1)}g\big)(\mu\tensor_i\nu) = (-1)^{|g||\mu|} f(\mu)\tensor_i g(\nu)$.
      Given the same data, we can also define a morphism
      \begin{align*}
        &f\circ'g\colon M\circ N \to M'\circ(N;N') , &
        f\circ'g &\defeq \Dsum_{m\in\NN} \Dsum_{n=n_1+\dotsb+n_m}
          f \tensor_{\Sy_m} \left( \sum_{i=1}^m 1^{\tensor(i-1)}\tensor g\tensor 1^{\tensor(m-i)}\right) .
      \end{align*}
      When $N' = N$ we implicitly postcompose with $M\circ(N;N) \onto M\circ N$ s.t.\@
      $f\circ'g\colon M\circ N \to M'\circ N$.  With this notation, we can write the differential of the full
      composite product as $\diff^{M\circ N} = \diff^M\circ 1 + 1\circ'\diff^N$.

    \subsection{Operads and cooperads} \label{S:Prelim:Op}
      A \emph{dg operad} is a monoid $(\opdP,\gamma,\eta)$ in $\dgSMod$, i.e.\@ an $\Sy$-module $\opdP$ with
      \emph{composition} map $\gamma\colon \opdP\circ\opdP \to \opdP$ and \emph{unit} $\eta\colon \smodI \to \opdP$
      satisfying associativity and left and right unit axioms.  A \emph{morphism of dg operads
      $f\colon \opdP \to \opdP'$} is a monoid morphism, i.e.\@ a morphism of the underlying dg $\Sy$-modules commuting
      with the structure maps.  An \emph{augmentation for $\opdP$} is a morphism of dg operads
      $\varepsilon\colon \opdP \to \smodI$ s.t.\@ $\varepsilon\eta = \id$.  We denote by $\red{\opdP}$ the
      \emph{augmentation ideal} $\red{\opdP} \defeq \ker\varepsilon$ and the restriction of $\gamma$ to it by
      $\red{\gamma}\colon \red{\opdP}\circ\red{\opdP} \to \red{\opdP}$.  We call $\opdP$ \emph{reduced}, if its
      underlying $\Sy$-module is reduced.

      Dually, a \emph{dg cooperad} is a comonoid $(\coopdC,\Delta,\varepsilon)$ in $\dgSMod$, i.e.\@ an $\Sy$-module
      $\coopdC$ with \emph{decomposition} map $\decomp\colon \coopdC \to \coopdC\circ\coopdC$ and \emph{counit}
      $\varepsilon\colon \coopdC \to \smodI$ satisfying coassociativity and left and right counit axioms.  A
      \emph{morphism of dg cooperads $f\colon \coopdC \to \coopdC'$} is a comonoid morphism, i.e.\@ a morphism of the
      underlying dg $\Sy$-modules commuting with the structure maps.  A \emph{coaugmentation for $\coopdC$} is a
      morphism of dg cooperads $\eta\colon \smodI \to \coopdC$ s.t.\@ $\varepsilon\eta = \id$.  We denote by
      $\red{\coopdC}$ the \emph{coaugmentation coideal} $\red{\coopdC} \defeq \coker\eta$ and by
      $\rdecomp\colon \red{\coopdC} \to \red{\coopdC}\circ\red{\coopdC}$ the corestriction of $\decomp$, i.e.\@
      $\decomp(\mu) = 1\circ\mu + \rdecomp(\mu) + \mu\circ(1,\dotsc,1)$.  A coaugmented dg cooperad is called
      \emph{conilpotent} if for any element its successive decompositions stabilize, see
      \cite[\S5.8.6]{loday2012algebraic} for the technical details.  We call $\coopdC$ \emph{reduced}, if its
      underlying $\Sy$-module is reduced.

      \subsubsection{Infinitesimal (de)composition}
      By restriction, we obtain for an operad $(\opdP,\gamma,\eta)$ the \emph{infinitesimal} or \emph{partial
      composition} map $\gamma_{(1)}\colon \opdP\circ_{(1)}\opdP \to \opdP$.  Similarly, we obtain the
      \emph{infinitesimal} of \emph{partial decomposition} map
      $\Delta_{(1)}\colon \coopdC \to \coopdC\circ_{(1)}\coopdC$ for a cooperad $(\coopdC,\Delta,\varepsilon)$ by
      corestriction.  Note that the (co)operad (de)composition maps are completely determined by their partial
      (de)composition counterparts.
      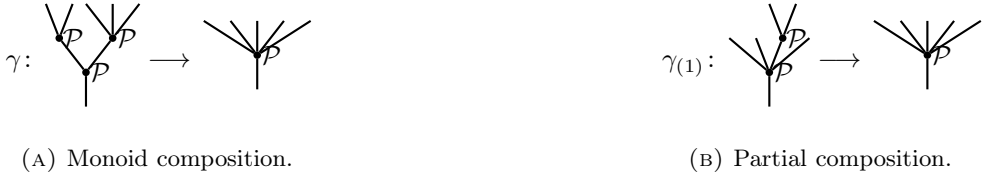
\begin{figure}[ht]
        \begin{subfigure}[b]{.5\linewidth}
          \begin{align*}
            \gamma\colon &
            \begin{aligned}\begin{tikzpicture}[PRT]
              \coordinate(root) child { node[label=right:$\opdP$] {}
                child { node[label=right:$\opdP$] {} child child }
                child[missing]
                child { node[label=right:$\opdP$] {} child child child }
              };
            \end{tikzpicture}\end{aligned}
            \longrightarrow
            \begin{aligned}\begin{tikzpicture}[PRT]
              \coordinate(root) child { node[label=right:$\opdP$] {}
                child child child child child
              };
            \end{tikzpicture}\end{aligned}
          \end{align*}
          \caption{Monoid composition.}
        \end{subfigure}%
        \begin{subfigure}[b]{.5\linewidth}
          \begin{align*}
            \gamma_{(1)}\colon &
            \begin{aligned}\begin{tikzpicture}[PRT]
              \coordinate(root) child { node[label=right:$\opdP$] {}
                child
                child
                child { node[label=right:$\opdP$] {} child child }
                child
              };
            \end{tikzpicture}\end{aligned}
            \longrightarrow
            \begin{aligned}\begin{tikzpicture}[PRT]
              \coordinate(root) child { node[label=right:$\opdP$] {}
                child child child child child
              };
            \end{tikzpicture}\end{aligned}
          \end{align*}
          \caption{Partial composition.}
        \end{subfigure}
        \caption{Comparison of operadic composition maps.}
      \end{figure}
      We use the classical notation $\mu \circ_i \nu \defeq \gamma_{(1)}(\mu \tensor_i \nu)$.

      \subsubsection{The Hadamard tensor product} Given operads $\opdP$, $\opdP'$, we can equip the arity-wise tensor
      product of their underlying $\Sy$-modules $(\opdP\tensor\opdP')(n) \defeq \opdP(n)\tensor\opdP'(n)$ with an
      operad structure as follows.  For elements $\mu\tensor\mu' \in (\opdP\tensor\opdP')(m)$ and
      $\nu\tensor\nu' \in (\opdP\tensor\opdP')(n)$, define their $i$-th partial composition by
      \begin{equation*}
        (\mu\tensor\mu') \circ_i (\nu\tensor\nu') \defeq (\mu\circ_i\nu) \tensor (\mu'\circ_i\nu') .
      \end{equation*}
      We call the resulting operad $\opdP\tensor\opdP'$ the \emph{Hadamard product} of $\opdP$ and $\opdP'$.

      \subsubsection{The endomorphism operad} Given a chain complex $(V,\diff)$, we define its \emph{endomorphism
      operad} $\End_V$ as follows.  Let $\End_V(n) = \hom(V^{\tensor n}, V)$ be the space of $n$-multilinear maps on
      $V$.  Note that $\hom$ denotes internal homomorphisms, i.e.\@ elements of $\End_V(n)$ need not be chain maps.
      With the symmetric group action given by permutation of inputs
      \begin{equation*} 
        \mu^\sigma(v_1,\dotsc,v_n) \defeq \varepsilon(\sigma;v_1,\dotsc,v_n) \cdot
          \mu(v_{\sigma^{-1}(1)},\dotsc,v_{\sigma^{-1}(n)}) ,
      \end{equation*}
      where $\varepsilon$ denotes the Koszul sign, and the usual differential
      $\partial(\mu) = \diff^V\circ\mu - (-1)^{|\mu|} \mu\circ\diff^{V^{\tensor n}}$, $\End_V$ becomes a dg
      $\Sy$-module.  For $\mu\in\End_V(n)$ and $\nu\in\End_V(k)$, we define their partial composition by
      \begin{equation*} 
        (\mu\circ_i\nu)(v_1,\dotsc,v_{n+k-1}) \defeq (-1)^{|\nu|(|v_1|+\dotsb+|v_{i-1}|)} \cdot
          \mu(v_1,\dotsc,v_{i-1},\nu(v_i,\dotsc,v_{i+k-1}),v_{i+k},\dotsc,v_{n+k-1}) .
      \end{equation*}

      We introduce the shorthand notation $\opdS \defeq \End_{s\Fk}$ and call this operad the \emph{suspension
      operad}.  Similarly, we introduce the \emph{desuspension operad} $\opdS^{-1} \defeq \End_{s^{-1}\Fk}$.  The
      operad $\opd{uCom} \defeq \End_\Fk$ governs \emph{unital commutative algebras}.  The operad $\Com$ differs from
      $\opd{uCom}$ only in arity $0$ where $\opd{uCom}(0) = \hom(\Fk,\Fk) \iso \Fk$ while $\Com(0) = 0$.

      \subsubsection{The (co)free (co)operad}  For any $M \in \dgSMod$, we denote by $\FOpd(M)$ the \emph{free
      augmented dg operad on $M$} and by $\cFCoopd(M)$ the \emph{cofree conilpotent dg cooperad on $M$}.  The
      underlying dg $\Sy$-module is the same in both cases.  It consists of planar trees with vertices labeled by
      elements of $M$ of arity corresponding to the number of incoming edges at the vertex and with a total ordering
      on the set of leaves.  Free composition is then given by grafting of such trees and reordering the leaves.  The
      (co)free (co)operad carries an extra grading we call weight-grading.  The component $\FOpd(M)^{(w)}$ resp.\@
      $\cFCoopd(M)^{(w)}$ of weight $w$ is spanned by the trees with precisely $w$ vertices.  We refer the reader to
      \cite{loday2012algebraic} for more details.

    \subsection{The category of \texorpdfstring{$\opdP$-algebras}{P-algebras}} 
      By definition of operads as monoids in dg $\Sy$-modules, they come with the usual notions of modules.  Of
      particular interest is the notion of a \emph{$\opdP$-algebra}, i.e.\@ a left $\opdP$-module structure on a chain
      complex viewed as an $\Sy$-module concentrated in arity 0.  Such a module structure is given by a morphism
      $\gamma_A\colon \opdP \circ A \to A$.  Under the currying isomorphism, this map corresponds to a collection of
      morphisms $\opdP(n) \to \hom(A^{\tensor n}, A)$.  In fact, a \emph{$\opdP$-algebra $(A,\diff,\gamma_A)$} is
      equivalently a triple $(A,\diff,g)$ consisting of a chain complex $(A,\diff)$ equipped with a morphism of
      operads $g\colon \opdP \to \End_A$.

      \subsubsection{Morphisms of $\opdP$-algebras} \label{S:Prelim:PAlg:mor} A \emph{morphism
      $f\colon (A,\diff,\gamma_A) \to (A',\diff',\gamma_A')$ of $\opdP$-algebras} is a morphism of the underlying
      chain complexes $f\colon (A,\diff) \to (A',\diff')$ commuting with the structure maps, i.e.\@ such that the
      following diagram commutes:
      \begin{equation}
        \begin{tikzcd}
          \opdP\circ A \ar[r,"\gamma_A"] \ar[d,"1\circ f"'] & A \ar[d,"f"]  \\
          \opdP\circ A' \ar[r,"\gamma_A'"] & A' .
        \end{tikzcd}
        \label{eq:opdAlg:dfmod}
      \end{equation}
      Note that, under the currying isomorphism, the diagonal $\opdP\circ A \to A'$ corresponds to a collection of
      morphisms $\opdP(n) \to \hom(A^{\tensor n},A')$.

      Given chain complexes $(V,\diff)$ and $(V',\diff')$, one may consider the sub dg $\Sy$-module
      $\End^V_{V'} \subset \End_{V\dsum V'}$ given by $\End^V_{V'}(n) = \hom(V^{\tensor n},V')$.  While $\End^V_{V'}$
      is not a sub dg operad, it does have an obvious $(\End_{V'},\End_V)$-bimodule structure given by restriction of
      the composition map in $\End_{V\dsum V'}$.  Using this notation, commutativity of the diagram in
      \cref{eq:opdAlg:dfmod} can equivalently be expressed as commutativity of the diagram
      \begin{equation*}
        \begin{tikzcd}
          \opdP \ar[r,"g"] \ar[d,"g'"'] & \End_A \ar[d,"f_*"]  \\
          \End_{A'} \ar[r,"f^*"] & \End^A_{A'} ,
        \end{tikzcd}
      \end{equation*}
      where $f_*$ resp.\@ $f^*$ denote post- resp.\@ pre-composition with $f$.

    \subsection{The cobar-bar resolution} \label{S:Prelim:CoBar}
      The category of reduced augmented dg operads over a commutative ring $\Fk$ admits a model category structure
      \cite{Hinich97,Hinich03,BergerMoerdijk03}, where the weak equivalences are the arity-wise quasi-isomorphisms
      and the fibrations are the arity-wise degree-wise epimorphisms.  In this context, we will be interested in
      cofibrant replacements $\opdQ \xto{\qiso} \opdP$ for a given operad $\opdP$.  Algebras over such a cofibrant
      replacement satisfy a certain homotopy invariance property \cite[Theorem 3.5]{BergerMoerdijk03}.  One way to
      obtain such a cofibrant resolution for an operad is by means of the cobar-bar resolution
      \cite{GetzlerJones94,BergerMoerdijk06}, i.e.\@ the counit of the adjunction
      \begin{equation*}
        \begin{tikzcd}[/tikz/baseline=-0.5ex]
          \COBAR : \cat{conil dg Coop}
            \arrow[transform canvas={yshift=0.5ex}]{r}{}
          & \cat{aug dg Op} : \BAR
            \arrow[transform canvas={yshift=-0.5ex}]{l}{}
        \end{tikzcd} ,
      \end{equation*}
      which we recall here.

      \subsubsection{The bar construction} The \emph{bar construction} of an augmented dg operad $\opdP$ is the cofree
      conilpotent dg cooperad on the arity-wise suspension $s\red{\opdP}$ with an additional term in the
      codifferential, $\BAR\opdP \defeq \left(\cFCoopd(s\red{\opdP}), \diff = \diff_1 - \diff_2\right)$.  The
      component $\diff_1$ of the codifferential $\diff$ is the coextension of the internal differential of $\opdP$
      with appropriate shifting, i.e.\@ $\diff_1$ coextends
      \begin{equation*}
        \begin{tikzcd}[/tikz/baseline=-0.5ex]
          \cFCoopd(s\red{\opdP})  \arrow[r, two heads]
          & \cFCoopd(s\red{\opdP})^{(1)} \iso s\red{\opdP}  \arrow[r,"\diff^{s\red{\opdP}}"]
          & s\red{\opdP}
        \end{tikzcd} .
      \end{equation*}
      Similarly, $\diff_2$ coextends the multiplication of $\opdP$ with appropriate shifting, i.e.\@ the following
      composition,
      \begin{equation*}
        \begin{tikzcd}[/tikz/baseline=-0.5ex]
          \cFCoopd(s\red{\opdP})  \arrow[r, two heads]
          & \cFCoopd(s\red{\opdP})^{(2)}
            \iso s\red{\opdP} \circ_{(1)} s\red{\opdP}  \arrow[rr,"s^{-1}\circ_{(1)}s^{-1}"]
          && \red{\opdP} \circ_{(1)} \red{\opdP}  \arrow[r,"\gamma_{(1)}"]
          & \red{\opdP}  \arrow[r,"s"]
          & s\red{\opdP}
        \end{tikzcd} .
      \end{equation*}
      In plain english, the codifferential amounts to the sum of applying $\diff^{\red{\opdP}}$ at each vertex and
      contracting each internal edge while composing its source and target vertex by the partial composition map
      $\gamma_{(1)}$ (with appropriate signs), as shown in \Cref{fig:BarCodiff}.
      \begin{figure}[h]
        \begin{equation*}
          \diff^{\BAR\opdP}\left(
            \begin{aligned}\begin{tikzpicture}[PRT]
              \coordinate(root) child { node[label=-2:$\,\mu_1$] (a) {}
                child { node[label=right:$\mu_2$] (b) {}
                  child
                  child { node[label=right:$\mu_3$] {} child child }
                }
                child[missing]
                child[missing]
                child { node[label=right:$\,\mu_4$] {} child child child }
              };
            \end{tikzpicture}\end{aligned}
          \right) = \sum_\textrm{vertices} \pm
            \begin{aligned}\begin{tikzpicture}[PRT]
              \coordinate(root) child { node[label=-2:$\,\mu_1$] (a) {}
                child { node[label=right:$\mu_2$] (b) {}
                  child
                  child { node[label=right:$\mu_3$] {} child child }
                }
                child[missing]
                child[missing]
                child { node[label=right:$\,\mu_4$] {} child child child }
              };
              \node[PRTmark,circle,dashed,fit=(a),xshift=.9ex,yshift=.3ex,inner sep=1.2ex,
                  label={-45:$\diff^{\opdP}$}] {};
            \end{tikzpicture}\end{aligned}
          + \sum_\textrm{inner edges} \pm
            \begin{aligned}\begin{tikzpicture}[PRT]
              \coordinate(root) child { node[label=-2:$\,\mu_1$] (a) {}
                child { node[label=right:$\mu_2$] (b) {}
                  child
                  child { node[label=right:$\mu_3$] {} child child }
                }
                child[missing]
                child[missing]
                child { node[label=right:$\,\mu_4$] {} child child child }
              };
              \node[PRTmark,dashed,fit=(a) (b),rotate=-32,xshift=.7ex,yshift=.5ex,yscale=0.5,inner sep=1.2ex,
                  label={below:$\gamma_{(1)}$}] {};
            \end{tikzpicture}\end{aligned}
          .
        \end{equation*}
        \caption{Codifferential of the bar construction.}
        \label{fig:BarCodiff}
      \end{figure}

      \subsubsection{The cobar construction} The cobar construction of a conilpotent dg cooperad $\coopdC$ is defined
      as the dg operad $\COBAR\coopdC\defeq \left(\FOpd(s^{-1}\red{\coopdC}), \diff = \diff_1 - \diff_2\right)$.  The
      differential again has two terms.  The $\diff_1$ term is the extension of
      \begin{equation*}
        \begin{tikzcd}[/tikz/baseline=-0.5ex]
          s^{-1}\red{\coopdC}  \arrow[r,"\diff^{s^{-1}\red{\coopdC}}"]
          & s^{-1}\red{\coopdC} \iso \FOpd(s^{-1}\red{\coopdC})^{(1)}  \arrow[r, hook]
          & \FOpd(s^{-1}\red{\coopdC})
        \end{tikzcd}
      \end{equation*}
      and $\diff_2$ extends
      \begin{equation*}
        \begin{tikzcd}[/tikz/baseline=-0.5ex]
          s^{-1}\red{\coopdC}  \arrow[r,"s"]
          & \red{\coopdC}  \arrow[r,"\Delta_{(1)}"]
          & \red{\coopdC} \circ_{(1)} \red{\coopdC}  \arrow[rr,"s^{-1}\circ_{(1)}s^{-1}"]
          && s^{-1}\red{\coopdC} \circ_{(1)} s^{-1}\red{\coopdC}
            \iso \FOpd(s^{-1}\red{\coopdC})^{(2)}  \arrow[r, hook]
          & \FOpd(s^{-1}\red{\coopdC})
        \end{tikzcd} .
      \end{equation*}

      \subsubsection{The adjunction} The fact that the cobar and bar constructions form an adjunction
      $\COBAR\dashv\BAR$ is most easily seen from \cref{eq:Tw:Iso} below after introducing the twisting morphism
      bifunctor.

      For the counit of the cobar-bar adjunction to give a cofibrant resolution of an operad $\opdP$, one needs the
      assumption that the operad $\opdP$ is $\Sy$-cofibrant.  This condition means that the underlying $\Sy$-module is
      in fact cofibrant, i.e.\@ the $\Fk[\Sy_n]$-modules $\opdP(n)$ are projective in each arity $n$.  This is always
      true when $\Fk$ is a field of characteristic $0$ and for this special case the above result is contained in
      \cite{GetzlerJones94,GinzburgKapranov94}.

      \begin{thm}[{\cite[\S8.5]{BergerMoerdijk06}}]
        The counit of the cobar-bar adjunction gives a cofibrant resolution $\COBAR\BAR\opdP \xto{\qiso} \opdP$,
        provided the operad $\opdP$ is $\Sy$-cofibrant.
      \end{thm}

    \subsection{Twisting morphisms} \label{S:Prelim:Tw}
      Let $\opdP$ be an augmented dg operad and $\coopdC$ a conilpotent dg cooperad.  Since the operad underlying the
      cobar construction is the free operad $\FOpd(s^{-1}\red{\coopdC})$, any morphism of augmented operads
      $\COBAR(\coopdC) \to \opdP$ is uniquely determined by its value on generators $s^{-1}\red{\coopdC}$.  This means
      morphisms of augmented operads $\FOpd(s^{-1}\red{\coopdC}) \to \opdP$ are in one to one correspondence with
      morphisms of $\Sy$-modules $s^{-1}\red{\coopdC} \to \red{\opdP}$ or, equivalently, degree $-1$ morphisms of
      $\Sy$-modules $\red{\coopdC} \to \red{\opdP}$.  A similar argument holds for the bar construction and we thus
      obtain
      \begin{align*}
        \Hom_\cat{aug Op}(\FOpd(s^{-1}\red{\coopdC}), \opdP)
          \iso \hom_{\SMod}(\red{\coopdC},\red{\opdP})_{-1}
          \iso \Hom_\cat{conil Coop}(\coopdC, \cFCoopd(s\red{\opdP}))  .
      \end{align*}

      We introduce the subset of twisting morphisms
      $\Tw(\coopdC,\opdP) \subset \hom_{\SMod}(\red{\coopdC},\red{\opdP})_{-1}$ that correspond (under the above
      isomorphisms) to morphisms of dg (co)operads.  Given $f,g \in \hom_{\SMod}(\red{\coopdC},\red{\opdP})$, we
      define their pre-Lie \emph{convolution product} to be the composition
      \begin{equation*}
        f \star g = \left(
            \begin{tikzcd}
              \coopdC \ar[r, "\Delta_{(1)}"]
              & \coopdC \circ_{(1)} \coopdC  \ar[r, "f\circ_{(1)}g"]
              & \opdP \circ_{(1)} \opdP  \ar[r, "\gamma_{(1)}"]
              & \opdP
            \end{tikzcd}
          \right) .
      \end{equation*}
      The differential for $f \in \hom_{\SMod}(\red{\coopdC},\red{\opdP})$ is defined as usual by
      $\partial(f) = \diff^{\opdP} \circ f - (-1)^{|f|} f \circ \diff^{\coopdC}$.  We consider the subset of degree
      $-1$ elements satisfying the Maurer--Cartan equation,
      \begin{equation}
        \Tw(\coopdC, \opdP)
          \defeq \{\, \alpha \in \hom_{\SMod}(\red{\coopdC}, \red{\opdP})_{-1}
            \mid  \partial\alpha + \alpha\star\alpha = 0 \,\} .
        \label{eq:Tw}
      \end{equation}
      These elements are called \emph{twisting morphisms}.  A simple computation shows that the twisting morphism
      bifunctor is represented by the cobar and bar functors in the following sense:
      \begin{equation}
        \begin{tikzcd}[row sep=small]
          \Hom_\cat{dg aug Op}(\COBAR\coopdC, \opdP) \ar[r,<->,"\iso"] \ar[d,phantom,"\in" rotate=90]
          & \Tw(\coopdC, \opdP) \ar[r,<->,"\iso"] \ar[d,phantom,"\in" rotate=90]
          & \Hom_\cat{dg conil Coop}(\coopdC, \BAR\opdP) \ar[d,phantom,"\in" rotate=90]  \\
          g_\alpha \ar[r,<->]
          & \alpha \ar[r,<->]
          & f_\alpha  .
        \end{tikzcd}
        \label{eq:Tw:Iso}
      \end{equation}

      \subsubsection{Koszul twisting morphisms}  The differential of the composite product of $\Sy$-module
      $\coopdC\circ\opdP$ is defined as $\diff^{\coopdC\circ\opdP} = \diff^\coopdC \circ 1 + 1 \circ' \diff^\opdP$.
      Given a twisting morphism $\alpha\colon \coopdC \to \opdP$, we may use it to modify this differential by adding
      an additional term,
      \begin{equation}
        \diff_\alpha = \left(
            \begin{tikzcd}
              \coopdC\circ\opdP  \ar[r, "\Delta_{(1)}\circ 1"]
              & (\coopdC\circ_{(1)}\coopdC)\circ\opdP  \ar[r, "(1\circ_{(1)}\alpha)\circ 1"]
              &[1em] (\coopdC\circ_{(1)}\opdP)\circ\opdP
              \iso \coopdC\circ(\opdP;\opdP\circ\opdP)  \ar[r, "1\circ(1;\gamma)"]
              & \coopdC\circ(\opdP;\opdP)  \ar[r, two heads]
              & \coopdC\circ\opdP
            \end{tikzcd}
          \right) .
        \label{eq:TwDiff}
      \end{equation}
      We denote by $\coopdC\circ_\alpha\opdP$ the composite product $\coopdC\circ\opdP$ with the twisted differential
      $\diff^{\coopdC\circ_\alpha\opdP} = \diff^{\coopdC\circ\opdP} + \diff_\alpha$ and call it the \emph{twisted
      composite product}.

      We call a twisting morphism $\alpha$ \emph{Koszul} if the twisted composite product $\coopdC\circ_\alpha\opdP$
      is acyclic.  The set of Koszul twisting morphisms is denoted by $\Kos(\coopdC,\opdP)$.

      \begin{pp}
        Assume that $\coopdC(n)$ and $\opdP(n)$ are projective $\Fk$-modules for all $n$.
        Under the isomorphisms of \cref{eq:Tw:Iso}, the Koszul twisting morphisms correspond to the quasi-isomorphisms,
        \begin{equation*}
          \qIso_\cat{dg aug Op}(\COBAR\coopdC, \opdP)
            \iso \Kos(\coopdC, \opdP)
            \iso \qIso_\cat{dg conil Coop}(\coopdC, \BAR\opdP)  .
        \end{equation*}
      \end{pp}

      \begin{proof}
        See \cite[Theorem~2.1.15]{Fresse04} or \cite[Theorem~6.6.1]{loday2012algebraic}.
      \end{proof}

    \subsection{Koszul duality} \label{S:Prelim:Koszul}
      In \Cref{S:Prelim:CoBar}, we have introduced the (functorial) cobar-bar resolution for any operad $\opdP$.  In
      this section, we consider a type of operads for which a much smaller resolution can be produced.

      Assume that we have a \emph{quadratic presentation $(E,R)$}, i.e.\@ an $\Sy$-module of \emph{generators} $E$,
      and a sub $\Sy$-module $R \subset \FOpd(E)^{(2)}$ of \emph{relations}.  We denote by $(R)$ the \emph{operadic
      ideal} generated by $R$, i.e.\@ the smallest sub $\Sy$-module $R \subset (R) \subset \FOpd(E)$ s.t.\@
      $\FOpd(E)/{(R)}$ with the induced structure is an operad.  This operad is denoted $\opdP = \opdP(E,R)$.  The
      same quadratic presentation also determines a cooperad $\coopdC(E,R)$; it is defined as the largest subcooperad
      of $\cFCoopd(E)$ for which the composition
      \begin{equation*}
        \begin{tikzcd}
          \coopdC(E,R) \ar[r, hook] & \cFCoopd(E) \ar[r,two heads] & \cFCoopd(E)/R
        \end{tikzcd}
      \end{equation*}
      vanishes, i.e.\@ any other such subcooperad inclusion $\coopdC \into \cFCoopd(E)$ factors through
      $\coopdC(E,R)$.

      We define the \emph{Koszul dual cooperad of $\opdP(E,R)$} to be $\opdP^\ash \defeq \coopdC(sE,s^2R)$.  It comes
      with a canonical twisting morphism
      \begin{equation*}
        \kappa = \left(
            \begin{tikzcd}
              \opdP^\ash = \coopdC(sE,s^2R) \ar[r, hook]
              & \cFCoopd(sE) \ar[r, two heads]
              & sE \ar[r, "s^{-1}"] \ar[r]
              & E \ar[r, hook]
              & \opdP(E,R)
            \end{tikzcd}
          \right) .
      \end{equation*}
      An operad $\opdP(E,R)$ is called \emph{Koszul} if the twisted composite $\opdP^\ash\circ_\kappa\opdP$ is acyclic
      or, equivalently, $g_\kappa\colon \COBAR\opdP^\ash \to \opdP$ is a quasi-isomorphism.  The classical operads
      $\Ass$ of associative algebras, $\Com$ of commutative associative algebras, and $\Lie$ of Lie algebras are
      examples of Koszul operads.

      Given a finite-dimensional presentation $(E,R)$, we can also produce a dual presentation.  We restrict ourselves
      to the case of a \emph{binary} quadratic presentation, i.e.\@ when $E=E(2)$ is concentrated in arity $2$.  The
      general case is essentially the same but requires some degree shifts.  Consider the $\Sy$-module
      $E^\vee \defeq E \tensor \Fk\cdot\sgn_2$ with the orthogonal space of relations
      $R^\perp \subset \FOpd(E^\vee)^{(2)}$.  This gives a new binary quadratic presentation $(E^\vee,R^\perp)$ and an
      associated operad $\opdP^! = \opdP(E^\vee,R^\perp)$ which we call the \emph{Koszul dual operad}.  One can show
      that the Koszul dual operad and cooperad are related by the equation $\opdP^\ash = (\opdS\tensor\opdP^!)^*$.

      For the classical operads we have $\Ass^! = \Ass$, $\Lie^! = \Com$, and $\Com^! = \Lie$.  Our interest lies in
      particular in the case of the $\Lie$ operad.  Here we obtain
      $\LieK = (\opdS\tensor\Lie^!)^* = (\opdS\tensor\Com)^* = \opdS^*$ and therefore $\LieK(0) = 0$ and
      $\LieK(n) = \big(\LieKgen{n}\cdot\Fk\cdot\sgn_n\big)[n-1]$ for $n\geq 1$ as an $\Sy$-module.  We will make its
      cooperad structure explicit in \Cref{S:LieD:LieKLeibK}.

    \subsection{Homotopy algebras and morphisms} \label{S:Prelim:htyAlg}
      Given any cooperad $\coopdC$, we can consider the category of \emph{homotopy algebras over $\COBAR\coopdC$}.
      The objects in this category are the usual $\COBAR\coopdC$-algebras, however we introduce a different notion of
      morphisms.  Recall that algebras over any operad $\opdP$ come with a natural notion of morphism.  Such a
      morphism of $\opdP$-algebras is a morphisms of the underlying chain complexes commuting with all structure maps,
      see \Cref{S:Prelim:PAlg:mor}.  In the context of homotopy algebras, this notion of morphism is too strict.  A
      better behaved type of morphism is introduced below.

    \subsubsection{Weak morphisms of homotopy algebras} \label{S:Prelim:wMor} Consider two $\COBAR\coopdC$-algebras
      $V$ and $V'$, i.e.\@ two chain complexes $(V,\diff)$, $(V',\diff')$ equipped with structure maps given by
      twisting morphisms
      \begin{align*}
        &  \alpha \in \Tw(\coopdC,\End_V) ,
        && \alpha' \in \Tw(\coopdC,\End_{V'}) .
      \end{align*}
      We apply the general theory as described in \cite[\S10.2.4]{loday2012algebraic}.  Even though we are not working
      with Koszul operads, the relevant proofs still hold for general cooperads $\coopdC$.  In particular, a weak
      morphism as defined below corresponds to a morphism of quasi-cofree codifferential $\coopdC$-coalgebras
      $\coopdC(V) \to \coopdC(V')$.

      A \emph{homotopy morphism} or \emph{weak morphism} of $\COBAR\coopdC$-algebras is a degree 0 solution to the
      following Maurer--Cartan equation:
      \begin{align}
        f &\colon \coopdC \to \End^V_{V'} ,
        &&\partial(f) - f \ast \alpha + \alpha' \circledast f = 0 ,
        \label{eq:htymor:MCE}
      \end{align}
      where
      \begin{align}
        f\ast\alpha &= \left(
            \begin{tikzcd}[ampersand replacement=\&]
              \coopdC \ar[r,"\Delta_{(1)}"]
              \& \coopdC\circ_{(1)}\coopdC \ar[r,"f\circ_{(1)}\alpha"]
              \& \End^V_{V'}\circ_{(1)}\End_V \ar[r]
              \& \End^V_{V'}
            \end{tikzcd}
          \right) ,  \\
        \alpha'\circledast f &= \left(
            \begin{tikzcd}[ampersand replacement=\&]
              \coopdC \ar[r,"\Delta"]
              \& \coopdC\circ\coopdC \ar[r,"\alpha'\circ f"]
              \& \End_{V'}\circ\End^V_{V'} \ar[r]
              \& \End^V_{V'}
            \end{tikzcd}
          \right) .
      \end{align}

      \subsubsection{Composition of weak morphisms} 
      Given weak morphisms $V \xto{f} V' \xto{f'} V''$, their composition $f'\circ f$ is defined to be
      \begin{equation*}
        f'\circ f = \left(
            \begin{tikzcd}
              \coopdC \ar[r,"\Delta"]
              & \coopdC\circ\coopdC \ar[r,"f'\circ f"]
              & \End^{V'}_{V''}\circ\End^V_{V'} \ar[r]
              & \End^V_{V''}
            \end{tikzcd}
          \right) .
      \end{equation*}
      Note the double use of the notation $f'\circ f$; the meaning should be clear from the context.

      It is simple to verify that $f'\circ f$ satisfies \cref{eq:htymor:MCE}, i.e.\@ $f'\circ f$ is indeed a weak
      morphism.  This definition for the composition is equivalent to the composition of the corresponding
      quasi-cofree codifferential $\coopdC$-coalgebras $\coopdC(V) \to \coopdC(V') \to \coopdC(V'')$.

  \section{An \texorpdfstring{$\Sy$-free}{S-free} resolution of the Koszul dual cooperad of the Lie operad}
    \label{S:LieD}
    In this section, we attempt to construct an $\Sy$-free resolution $\psi\colon \LieD \xto{\qiso} \LieK$ of dg
    cooperads over $\ZZ$ for the Koszul dual cooperad of the Lie operad.  Assume for a moment that we have such a
    resolution.  If we can show that $\psi$ is in fact a weak equivalence, i.e.\@ $\COBAR\psi$ is a quasi-isomorphism,
    then this gives us a small cofibrant resolution
    \begin{equation*}
      \begin{tikzcd}[column sep=huge]
        \ELinf\defeq\COBAR\LieD \ar[dr, start anchor=east, "g_{\kappa\circ\psi}", "\qiso"']
          \ar[d, "\COBAR\psi"', "\qiso"]  \\
        \Linf\defeq\COBAR\LieK \ar[r, "g_\kappa", "\qiso"']
          & \Lie
      \end{tikzcd}
    \end{equation*}
    of the Lie operad over $\ZZ$, and therefore over any unital commutative ring $\Fk$.  Unfortunately, we do not yet
    have a general method to obtain such an $\Sy$-free resolution.  Instead, we proceed degree-wise to construct a dg
    cooperad $\LieD[3]$ that satisfies the desired conditions in low degrees (see the following sections for the
    precise meaning of ``in low degrees.'')

    \begin{rk}
      While our notation below may suggest that we are working with truncations of a dg cooperad $\LieD$, this is not
      a proven result.  We do not have a complete resolution $\LieD \xto{\qiso} \LieK$, nor do we have a proof that
      our $\LieD[3]$ extends to such a resolution.
    \end{rk}

    The remainder of this section is organized as follows.  In \Cref{S:LieD:LieKLeibK}, we recall the Koszul dual
    cooperads $\LieK$ and $\LeibK$ of the Lie and Leibniz operads, respectively.  In \Cref{S:LieD:SMod}, we construct
    an explicit $\Sy$-free resolution $\psi\colon \LieD[3] \to \LieK$ as a dg $\Sy$-module in low degrees.  In
    \Cref{S:LieD:Coop}, we equip $\LieD[3]$ with a decomposition map and show that our definition turns it into a dg
    cooperad.  Finally, in \Cref{S:LieD:KosCx}, we prove that the homology of the twisted composite product
    $\LieD[3] \circ_{(\kappa\circ\psi)} \Lie$ vanishes in low degrees.

    \subsection{The Koszul dual cooperads of the Lie and Leibniz operads} \label{S:LieD:LieKLeibK}
      In \Cref{S:Prelim:Koszul}, we saw that the Koszul dual cooperad for $\Lie$ is $\LieK = \opdS^*$.  Explicitly,
      this means the $\Sy$-module underlying $\LieK$ is given by $\LieK(0) = 0$ and
      $\LieK(n) = \big(\LieKgen{n}\cdot\Fk\cdot\sgn_n\big)[n-1]$ for $n\geq 1$, and its decomposition map is given by
      \begin{equation}
        \begin{split}
          \decomp(\LieKgen{n}) =
            \sum_{\substack{1\leq j\leq n\\i_1+\dotsb+i_j=n}}
              (-1)^{(j-1)(n-j)} \cdot (-1)^{\sum_{p=1}^j (p-1)(i_p-1)}
            \sum_{\sigma\in\ruSh(i_1,\dotsc,i_j)}
              (-1)^{|\sigma|} \cdot \LieKgen{j}\circ(\LieKgen{i_1},\dotsc,\LieKgen{i_j})^\sigma  .
        \end{split}
        \label{eq:LeibK:decomp}
      \end{equation}
      From this we obtain, by projection to $\LieK\pcirc\LieK$, the partial decomposition map
      \begin{equation}
        \Delta_{(1)}(\LieKgen{n}) =
          \sum_{i+j=n+1} (-1)^{(j-1)(i-1)}
          \sum_{p=1}^j (-1)^{(p-1)(i-1)} \sum_{\sigma\in\ruSh(p-1,i)}
            (-1)^{|\sigma|} \cdot (\LieKgen{j} \tensor_p \LieKgen{i})^\sigma .
        \label{eq:LeibK:pdecomp}
      \end{equation}
      Using the skew-symmetry of the $\LieKgen{n}$, it can be rewritten as
      \begin{equation}
        \Delta_{(1)}(\LieKgen{n}) = \sum_{i+j=n+1} (-1)^{(j-1)(i-1)}
          \sum_{\sigma\in\uSh(i,j-1)} (-1)^{|\sigma|} \cdot (\LieKgen{j} \tensor_1 \LieKgen{i})^\sigma  .
        \label{eq:LieK:pdecomp}
      \end{equation}

      The Koszul dual operad $\Leib^! = \Zinb$ for $\Leib$ was first introduced by Loday in \cite{Loday95}.  Algebras
      over $\Leib^!$ were originally referred to as \emph{dual Leibniz algebras}, but are now more commonly known as
      \emph{Zinbiel algebras}.  As before, the Koszul dual cooperad $\LeibK$ can be computed as
      $\LeibK = (\opdS \tensor \Zinb)^*$.  Explicitly, it is given by $\LeibK(0) = 0$ and
      $\LeibK(n) = \big(\LeibKgen{n}\cdot\Fk[\Sy_n]\big)[n-1]$ for $n\geq 1$, and with (partial) decomposition map
      given as in \cref{eq:LeibK:decomp,eq:LeibK:pdecomp}, substituting $\LeibKgen{n}$ for $\LieKgen{n}$.  Note
      however, that \cref{eq:LeibK:pdecomp,eq:LieK:pdecomp} are not equivalent in this case, since the $\LeibKgen{n}$
      are not skew-symmetric.

      \subsubsection{Homotopy Lie algebras are homotopy Leibniz algebras} \label{S:LieD:LieKLeibK:psi}
      There is an obvious morphism of dg cooperads,
      \begin{align}
        \psi&\colon \LeibK \to \LieK ,  && \psi(\LeibKgen{n}[\sigma]) = (-1)^\sigma \cdot \LieKgen{n} .
        \label{eq:LeibKtoLieK}
      \end{align}
      Since $\LeibK$ is $\Sy$-free and this morphism is surjective, this provides us with a good starting point for
      our resolution of $\LieK$.

    \subsection{The resolution as an \texorpdfstring{$\Sy$-free dg $\Sy$-module}{S-free dg S-module}}
      \label{S:LieD:SMod}
      In this section, we fix an $\Sy$-free resolution of $\LieK$ as an $\Sy$-module in low degrees.  We describe, in
      general, a way to obtain an $\Sy$-free $\Sy$-module $\LieD[k]$ with a morphism $\psi\colon \LieD[k] \to \LieK$
      satisfying $\Hm_r(\psi) = 0$ for $r\leq k$.  We make such an $\Sy$-module explicit for $k=3$.

      We construct, for each arity $n \geq 1$, an exact augmented complex
      $0 \leftarrow \LieK(n)_{n-1} \leftarrow \LieD[k](n)_\bullet$ in $\Fk[\Sy_n]$-modules up to degree $k+1$.  As
      indicated in \Cref{S:LieD:LieKLeibK:psi}, we may choose $\LieD[k](n)_{n-1} \defeq \LeibK(n)_{n-1}$ and the
      augmentation map to be $\psi$ as defined by \cref{eq:LeibKtoLieK}, i.e.\@ our complexes are of the following
      shape:
      \begin{equation*}
        \begin{tikzcd}[column sep=large]
          0 & \ar[l] \LieK(n)_{n-1}
          & \ar[l, "\psi = \diff_{n-1}"'] \LeibK(n)_{n-1}
          & \ar[l, "\diff_n"'] \LieD[k](n)_{n}
          & \ar[l, "\diff_{n+1}"'] \dots
          & \ar[l, "\diff_{k+1}"'] \LieD[k](n)_{k+1} .
        \end{tikzcd}
      \end{equation*}

      Our general approach to determining the higher degrees of $\LieD[k](n)$ is as follows.  For $r = n-1,\dotsc,k$,
      successively, we extend the complex using the following steps:
      \begin{enumerate}
        \item compute $\ker(\diff_r)$,
        \item choose generators $\{x_i\}_{i\in I}$ for $\ker(\diff_r)$ as a $\Fk[\Sy_n]$-module, and
        \item define $\LieD[k](n)_{r+1} \defeq \langle\hat{x}_i\rangle_{i\in I}$ to be the free $\Fk[\Sy_n]$-module
          generated by symbols $\{\hat{x}_i\}_{i\in I}$ and the differential by $\diff_{r+1}(\hat{x}_i) \defeq x_i$.
      \end{enumerate}
      Obviously, any complex $\LieD[k](n)$ constructed in this way will be exact.
      We describe the beginning of these computations explicitly for general $n$.
      \begin{description}[leftmargin=0em]
        \item[Degree \boldmath$r=n-1$] Since
          \begin{equation*}
            \diff_{n-1}\big(\LieDgen{n} - (-1)^{|\sigma|}\cdot\LieDgen{n}[\sigma]\big)
              = \psi\big(\LieDgen{n} - (-1)^{|\sigma|}\cdot\LieDgen{n}[\sigma]\big)
              = \LieKgen{n} - (-1)^{|\sigma|}\cdot(-1)^{|\sigma|}\cdot\LieKgen{n}[\sigma]
              = 0
          \end{equation*}
          and $\dim_\Fk\big(\im(\diff_{n-1})\big) = 1$, we obtain
          $\ker(\diff_{n-1}) = \linspan_\Fk\left\{ \LieDgen{n} - (-1)^{|\sigma|}\cdot\LieDgen{n}[\sigma] \mid
            \sigma\in\Sy_n\setminus\id \right\}$.
          As a $\Fk[\Sy_n]$-module, the kernel is generated by the set
          $\{\LieDgen{n}+\LieDgen{n}[\sigma_i]\}_{i=1}^{n-1}$ for the adjacent transpositions $\sigma_i = (i\ i+1)$.
          We define
          \begin{align*}
            \LieD[k](n)_n &\defeq \langle \LieDgen{n;i} \mid 1\leq i < n \rangle ,
            & \diff_n(\LieDgen{n;i}) &= - \LieDgen{n} - \LieDgen{n}[\sigma_i]  .
          \end{align*}
        \item[Degree \boldmath$r=n$] Clearly, the following hold
          \begin{align*}
            &\diff_n\big( \LieDgen{n;i} - \LieDgen{n;i}[\sigma_i] \big)
              = \big( - \LieDgen{n} - \LieDgen{n}[\sigma_i] \big)
                - \big( - \LieDgen{n} - \LieDgen{n}[\sigma_i] \big)^{\sigma_i} = 0 ,  \\
            \begin{split}
              &\diff_n\big( \LieDgen{n;i} - \LieDgen{n;i+1}[\sigma_i] + \LieDgen{n;i}[\sigma_{i+1}\sigma_i]
                  - \LieDgen{n;i+1} + \LieDgen{n;i}[\sigma_{i+1}] + \LieDgen{n;i+1}[\sigma_i\sigma_{i+1}] \big)  \\
              &\quad= \big( - \LieDgen{n} - \LieDgen{n}[\sigma_i] \big)
                - \big( - \LieDgen{n} - \LieDgen{n}[\sigma_{i+1}] \big)^{\sigma_i}
                + \big( - \LieDgen{n} - \LieDgen{n}[\sigma_i] \big)^{\sigma_{i+1}\sigma_i}  \\
              &\qquad- \big( - \LieDgen{n} - \LieDgen{n}[\sigma_{i+1}] \big)
                + \big( - \LieDgen{n} - \LieDgen{n}[\sigma_i] \big)^{\sigma_{i+1}}
                - \big( - \LieDgen{n} - \LieDgen{n}[\sigma_{i+1}] \big)^{\sigma_i\sigma_{i+1}}  = 0 ,
            \end{split}
          \end{align*}
          and, for $|i-j|>1$,
          \begin{equation*}
            \diff_n\big( \LieDgen{n;i} - \LieDgen{n;j}[\sigma_i] - \LieDgen{n;j} + \LieDgen{n;i}[\sigma_j] \big)
              = \big( - \LieDgen{n} - \LieDgen{n}[\sigma_i] \big)
                - \big( - \LieDgen{n} - \LieDgen{n}[\sigma_j] \big)^{\sigma_i}
                - \big( - \LieDgen{n} - \LieDgen{n}[\sigma_j] \big)
                + \big( - \LieDgen{n} - \LieDgen{n}[\sigma_i] \big)^{\sigma_j} = 0 .
          \end{equation*}
          In fact, these elements generate the kernel of $\diff_n$ under the $\Fk[\Sy_n]$-action.  We omit the general
          proof here, since we will only need this result in arity $3$ where it is a trivial computation of the rank
          of a $12 \times 18$ matrix over $\Fk$.

          We define
          \begin{align*}
            & \LieD[k](n)_{n+1} \defeq \langle \LieDgen{n;i,j} \mid 1\leq i\leq j < n \rangle ,
          \intertext{and}
            & \diff_{n+1}(\LieDgen{n;i,j}) =
            \begin{cases}
              - \LieDgen{n;i} + \LieDgen{n;i}[\sigma_i] ,
                & j = i ,  \\
              \LieDgen{n;i} - \LieDgen{n;i+1}[\sigma_i] + \LieDgen{n;i}[\sigma_{i+1}\sigma_i]
                  - \LieDgen{n;i+1} + \LieDgen{n;i}[\sigma_{i+1}] + \LieDgen{n;i+1}[\sigma_i\sigma_{i+1}] ,
                & j = i+1 ,  \\
              \LieDgen{n;i} - \LieDgen{n;j}[\sigma_i] - \LieDgen{n;j} + \LieDgen{n;i}[\sigma_j] ,
                & j > i+1 .
            \end{cases}
          \end{align*}
      \end{description}

      The above computations are already enough for our purpose.  In summary, we obtain the following explicit dg
      $\Sy$-module $\LieD[3]$.
      \begin{equation}
        \begin{tikzcd}[ampersand replacement=\&, row sep=small, column sep=huge]
          \langle \LieDgen{2} \rangle
            \&\ar[l, "\begin{bsmallmatrix}-1-(12)\end{bsmallmatrix}"'] \langle \LieDgen{2;1} \rangle
            \&[+3.8em]\ar[l, "\begin{bsmallmatrix}-1+(12)\end{bsmallmatrix}"'] \langle \LieDgen{2;1,1} \rangle
            \&[+9.8em]\ar[l, "\begin{bsmallmatrix}-1-(12)\end{bsmallmatrix}"'] \langle \LieDgen{2;1,1,1} \rangle  \\
          \& \langle \LieDgen{3} \rangle
            \&\ar[swap]{l}{\begin{bsmallmatrix}-1-(12)&-1-(23)\end{bsmallmatrix}}
              \langle \LieDgen{3;1},\LieDgen{3;2} \rangle
            \&\ar[swap]{l}{\begin{bsmallmatrix}-1+(12)&1+(23)+(132)&0\\0&-1-(12)-(123)&-1+(23)\end{bsmallmatrix}}
              \langle \LieDgen{3;1,1},\LieDgen{3;1,2},\LieDgen{3;2,2} \rangle  \\
          \&\& \langle \LieDgen{4} \rangle
            \&\ar[swap]{l}{\begin{bsmallmatrix}-1-(12)&-1-(23)&-1-(34)\end{bsmallmatrix}}
              \langle \LieDgen{4;1},\LieDgen{4;2},\LieDgen{4;3} \rangle  \\
          \&\&\& \langle \LieDgen{5} \rangle
        \end{tikzcd}
        \label{eq:LieD:cx}
      \end{equation}
      Note that we left out arity 1 so far.  We will need to define $\LieD[3](1) \defeq \LieK[3](1) = \Fk$ for the
      counit of the cooperad structure introduced in the following section.

    \subsection{The resolution as a dg cooperad} \label{S:LieD:Coop}
      In this section, we describe how to equip a dg $\Sy$-module $\LieD[k]$, as obtained in the previous section,
      with a decomposition map $\decomp$ in such a way that
      \begin{conditions}
        \item \label{LieD:coopd} $(\LieD[k],\diff,\decomp)$ becomes a dg cooperad, and
        \item \label{LieD:mor:coopd} $\psi\colon \LieD[k] \to \LieK$ becomes a morphism of dg cooperads.
      \end{conditions}
      As before, we begin by explaining the general approach and then proceed to make such a structure explicit for
      the case $k=3$, i.e.\@ on the dg $\Sy$-module $\LieD[3]$ of the previous section.

      Consider \cref{LieD:mor:coopd} first.  Since we started our resolution of dg $\Sy$-modules by
      $\psi\colon \LieD[k](n)_{n-1}=\LeibK(n) \to \LieK(n)$, and this is in fact a morphism of dg cooperads, we define
      the decomposition map $\decomp$ on $\LieD[k](n)_{n-1}$ in the same way as on $\LeibK(n)$, i.e.\@ by
      \cref{eq:LeibK:decomp}.  It remains to define the decomposition map for the higher degree terms of
      $\LieD[k](n)$.

      Next, consider \cref{LieD:coopd}.  Note that it implies, in particular, that the decomposition map be a
      map of dg $\Sy$-modules, i.e.\@ $\decomp$ has to commute with the differential as in the diagram
      \begin{equation}
        \begin{tikzcd}
          \LieD[k](n)_{r+1} \ar[d, "\diff"] \ar[r, dashed, "\decomp"]
          & \big(\LieD[k]\circ\LieD[k]\big)(n)_{r+1} \ar[d, "\diff"]  \\
          \LieD[k](n)_r \ar[r, "\decomp"]
          & \big(\LieD[k]\circ\LieD[k]\big)(n)_r .
        \end{tikzcd}
        \label{eq:decomp:dg}
      \end{equation}
      We use this condition to define $\decomp$ as follows:  In each arity $n$, proceed degree-wise for
      $r=n-1,\dots,k+1$.  Given a $\Fk[\Sy_n]$-basis $\{x_i\}_{i\in I}$ for $\LieD[k](n)_{r+1}$, solve the equations
      $\diff(y_i) = \decomp(\diff x_i)$ for $y_i$ and define $\decomp$ by $\decomp(x_i) \defeq y_i$ and
      $\Fk[\Sy_n]$-linearity.  Finally, it remains to check that our decomposition map $\decomp$ satisfies the
      coassociativity condition.

      The remainder of this section consists of the explicit computations for the case of $\LieD[3]$.  As mentioned in
      \Cref{S:Prelim:CN}, we work over the integers $\ZZ$.  Finding a decomposition map as described above amounts to
      solving systems of linear diophantine equations.  For $\LieD[3]$ these are still manageable by hand; for
      $\LieD[k]$ for $k\geq 4$, a computer can be used to solve them.  To save a bit of space, we work with the
      reduced decomposition map $\rdecomp$ instead of the full decomposition $\decomp$ below.

      \begin{description}[leftmargin=0em]
        \item[Arity \boldmath$n=2$]  The (reduced) decomposition $\rdecomp$ vanishes for degree reasons, i.e.\@ we
          have
          \begin{equation}
            \rdecomp(\LieDgen{2})
              \defeq \rdecomp(\LieDgen{2;1})
              \defeq \rdecomp(\LieDgen{2;1,1})
              \defeq \rdecomp(\LieDgen{2;1,1,1})
              \defeq 0 .
            \label{eq:LieD:l21x}
          \end{equation}
        \item[Arity \boldmath$n=3$]  For $\LieDgen{3}$, we have defined $\rdecomp$ by \cref{eq:LeibK:decomp}, i.e.\@
          \begin{equation}
            \rdecomp(\LieDgen{3}) =
              - \LieDgen{2}\circ(\LieDgen{2},1)
              + \LieDgen{2}\circ(1,\LieDgen{2})
              - \LieDgen{2}\circ(1,\LieDgen{2})^{(12)}  ,
            \label{eq:LieD:l3}
          \end{equation}
          as in $\LeibK(3)$.  To extend the decomposition map $\rdecomp$ to the next degree, we compute for
          $\LieDgen{3;1}$,
          \begin{align*}
            \rdecomp(\diff \LieDgen{3;1}) &= \rdecomp(- \LieDgen{3} - \LieDgen{3}[(12)])
              = - \rdecomp(\LieDgen{3}) - \rdecomp(\LieDgen{3})^{(12)}  \\
            &= - \left(- \LieDgen{2}\circ(\LieDgen{2},1)
                + \LieDgen{2}\circ(1,\LieDgen{2})
                - \LieDgen{2}\circ(1,\LieDgen{2})^{(12)} \right)
              - \left(- \LieDgen{2}\circ(\LieDgen{2},1)
                + \LieDgen{2}\circ(1,\LieDgen{2})
                - \LieDgen{2}\circ(1,\LieDgen{2})^{(12)} \right)^{(12)}  \\
            &= \LieDgen{2}\circ(\LieDgen{2},1) + \LieDgen{2}\circ(\LieDgen{2}[(12)],1)  ,
          \intertext{then solve for a preimage under $\diff$,}
            &= - \LieDgen{2}\circ(\diff \LieDgen{2;1},1) = \diff\big( \LieDgen{2}\circ(\LieDgen{2;1},1) \big)  .
          \end{align*}
          We proceed for $\LieDgen{3;2}$ in the same way,
          \begin{align*}
            \rdecomp(\diff \LieDgen{3;2}) &= \rdecomp(- \LieDgen{3} - \LieDgen{3}[(23)])
              = - \rdecomp(\LieDgen{3}) - \rdecomp(\LieDgen{3})^{(23)}  \\
            &= - \left(- \LieDgen{2}\circ(\LieDgen{2},1)
                + \LieDgen{2}\circ(1,\LieDgen{2})
                - \LieDgen{2}\circ(1,\LieDgen{2})^{(12)} \right)
              - \left(- \LieDgen{2}\circ(\LieDgen{2},1)
                + \LieDgen{2}\circ(1,\LieDgen{2})
                - \LieDgen{2}\circ(1,\LieDgen{2})^{(12)} \right)^{(23)}  \\
            &= \LieDgen{2}\circ(\LieDgen{2},1)
              - \LieDgen{2}\circ(1,\LieDgen{2})
              + \LieDgen{2}\circ(1,\LieDgen{2})^{(12)}
              + \LieDgen{2}[(12)]\circ(1,\LieDgen{2})^{(12)}
              - \LieDgen{2}\circ(1,\LieDgen{2}[(12)])
              + \LieDgen{2}[(12)]\circ(\LieDgen{2},1)  \\
            &= \diff\big(- \LieDgen{2;1}\circ(\LieDgen{2},1)
                - \LieDgen{2}\circ(1,\LieDgen{2;1})
                - \LieDgen{2;1}\circ(1,\LieDgen{2})^{(12)} \big)  .\\
          \end{align*}
          This gives us candidates for the definition of $\rdecomp$ for which the diagram (\ref{eq:decomp:dg})
          commutes.  We define
          \begin{align}
            \rdecomp(\LieDgen{3;1}) &\defeq \LieDgen{2}\circ(\LieDgen{2;1},1) ,
              \label{eq:LieD:l31}  \\
            \rdecomp(\LieDgen{3;2}) &\defeq
              - \LieDgen{2;1}\circ(\LieDgen{2},1)
              - \LieDgen{2}\circ(1,\LieDgen{2;1})
              - \LieDgen{2;1}\circ(1,\LieDgen{2})^{(12)}  .
              \label{eq:LieD:l32}
          \end{align}
          Using these definitions, we continue in the next degree and find
          \begin{align*}
            \rdecomp(\diff \LieDgen{3;1,1}) &= \rdecomp(- \LieDgen{3;1} + \LieDgen{3;1}[(12)])
              = - \rdecomp(\LieDgen{3;1}) + \rdecomp(\LieDgen{3;1})^{(12)}  \\
            &= - \LieDgen{2}\circ(\LieDgen{2;1},1) + \LieDgen{2}\circ(\LieDgen{2;1},1)^{(12)}
              = - \LieDgen{2}\circ(\LieDgen{2;1},1) + \LieDgen{2}\circ(\LieDgen{2;1}[(12)],1)  \\
            &= \LieDgen{2}\circ(\diff \LieDgen{2;1,1},1)
              = \diff\big( - \LieDgen{2}\circ(\LieDgen{2;1,1},1) \big)  ,\\
            \rdecomp(\diff \LieDgen{3;1,2}) &= \rdecomp(\LieDgen{3;1} + \LieDgen{3;1}[(23)] + \LieDgen{3;1}[(132)]
              - \LieDgen{3;2} - \LieDgen{3;2}[(12)] - \LieDgen{3;2}[(123)])  \\
            &= \LieDgen{2}\circ(\LieDgen{2;1},1)
              + \LieDgen{2}\circ(\LieDgen{2;1},1)^{(23)}
              + \LieDgen{2}\circ(\LieDgen{2;1},1)^{(132)}
              - \big(- \LieDgen{2;1}\circ(\LieDgen{2},1)
                - \LieDgen{2}\circ(1,\LieDgen{2;1})
                - \LieDgen{2;1}\circ(1,\LieDgen{2})^{(12)} \big)  \\
              &\quad- \big(- \LieDgen{2;1}\circ(\LieDgen{2},1)
                - \LieDgen{2}\circ(1,\LieDgen{2;1})
                - \LieDgen{2;1}\circ(1,\LieDgen{2})^{(12)} \big)^{(12)}  \\
              &\quad- \big(- \LieDgen{2;1}\circ(\LieDgen{2},1)
                - \LieDgen{2}\circ(1,\LieDgen{2;1})
                - \LieDgen{2;1}\circ(1,\LieDgen{2})^{(12)} \big)^{(123)}  \\
            &= \LieDgen{2}\circ(\LieDgen{2;1},1)
              + \LieDgen{2}[(12)]\circ(1,\LieDgen{2;1})^{(12)}
              + \LieDgen{2}[(12)]\circ(1,\LieDgen{2;1})
              + \LieDgen{2;1}\circ(\LieDgen{2},1)
              + \LieDgen{2}\circ(1,\LieDgen{2;1})
              + \LieDgen{2;1}\circ(1,\LieDgen{2})^{(12)}  \\
              &\quad+ \LieDgen{2;1}\circ(\LieDgen{2}[(12)],1)
              + \LieDgen{2}\circ(1,\LieDgen{2;1})^{(12)}
              + \LieDgen{2;1}\circ(1,\LieDgen{2})
              + \LieDgen{2;1}[(12)]\circ(1,\LieDgen{2}[(12)])^{(12)}
              + \LieDgen{2}[(12)]\circ(\LieDgen{2;1},1)
              + \LieDgen{2;1}\circ(1,\LieDgen{2}[(12)])  \\
            &= - (\diff \LieDgen{2;1})\circ(\LieDgen{2;1},1)
              - (\diff \LieDgen{2;1})\circ(1,\LieDgen{2;1})^{(12)}
              - (\diff \LieDgen{2;1})\circ(1,\LieDgen{2;1})
              - \LieDgen{2;1}\circ(\diff \LieDgen{2;1},1)
              - \LieDgen{2;1}\circ(1,\diff \LieDgen{2;1})^{(12)}  \\
              &\quad- \LieDgen{2;1}\circ(1,\diff \LieDgen{2;1})
              + (\diff \LieDgen{2;1,1})\circ(1,\LieDgen{2})^{(132)}  \\
            &= \diff\big(- \LieDgen{2;1}\circ(\LieDgen{2;1},1)
              - \LieDgen{2;1}\circ(1,\LieDgen{2;1})
              - \LieDgen{2;1}\circ(1,\LieDgen{2;1})^{(12)}
              + \LieDgen{2;1,1}\circ(1,\LieDgen{2})^{(132)} \big)  ,\\
            \rdecomp(\diff \LieDgen{3;2,2}) &= \rdecomp(- \LieDgen{3;2} + \LieDgen{3;2}[(23)])
              = - \rdecomp(\LieDgen{3;2}) + \rdecomp(\LieDgen{3;2})^{(23)}  \\
            &= - \left(- \LieDgen{2;1}\circ(\LieDgen{2},1)
              - \LieDgen{2}\circ(1,\LieDgen{2;1})
              - \LieDgen{2;1}\circ(1,\LieDgen{2})^{(12)} \right)  \\
              &\quad+ \left(- \LieDgen{2;1}\circ(\LieDgen{2},1)
                - \LieDgen{2}\circ(1,\LieDgen{2;1})
                - \LieDgen{2;1}\circ(1,\LieDgen{2})^{(12)} \right)^{(23)}  \\
            &= \LieDgen{2;1}\circ(\LieDgen{2},1)
              + \LieDgen{2}\circ(1,\LieDgen{2;1})
              + \LieDgen{2;1}\circ(1,\LieDgen{2})^{(12)}
              - \LieDgen{2;1}[(12)]\circ(1,\LieDgen{2})^{(12)}
              - \LieDgen{2}\circ(1,\LieDgen{2;1}[(12)])
              - \LieDgen{2;1}[(12)]\circ(\LieDgen{2},1)  \\
            &= - (\diff \LieDgen{2;1,1})\circ(\LieDgen{2},1)
              - \LieDgen{2}\circ(1,\diff \LieDgen{2;1,1})
              - (\diff \LieDgen{2;1,1})\circ(1,\LieDgen{2})^{(12)}  \\
            &= \diff\big(- \LieDgen{2;1,1}\circ(\LieDgen{2},1)
              + \LieDgen{2}\circ(1,\LieDgen{2;1,1})
              - \LieDgen{2;1,1}\circ(1,\LieDgen{2})^{(12)} \big)  .
          \end{align*}
          This again gives us candidates for the definition of $\rdecomp$ and we define
          \begin{align}
            \rdecomp(\LieDgen{3;1,1}) &\defeq - \LieDgen{2}\circ(\LieDgen{2;1,1},1) ,
              \label{eq:LieD:l311}  \\
            \rdecomp(\LieDgen{3;1,2}) &\defeq
              - \LieDgen{2;1}\circ(\LieDgen{2;1},1)
              - \LieDgen{2;1}\circ(1,\LieDgen{2;1})
              - \LieDgen{2;1}\circ(1,\LieDgen{2;1})^{(12)}
              + \LieDgen{2;1,1}\circ(1,\LieDgen{2})^{(132)} ,
              \label{eq:LieD:l312}  \\
            \rdecomp(\LieDgen{3;2,2}) &\defeq
              - \LieDgen{2;1,1}\circ(\LieDgen{2},1)
              + \LieDgen{2}\circ(1,\LieDgen{2;1,1})
              - \LieDgen{2;1,1}\circ(1,\LieDgen{2})^{(12)} .
              \label{eq:LieD:l322}
          \end{align}
        \item[Arity \boldmath$n=4$]  For $\LieDgen{4}$ we have
          \begin{equation}
            \begin{split}
            \rdecomp(\LieDgen{4}) &= \LieDgen{2}\circ(\LieDgen{3},1)
              + \LieDgen{2}\circ(1,\LieDgen{3})
              - \LieDgen{2}\circ(1,\LieDgen{3})^{(12)}
              + \LieDgen{2}\circ(1,\LieDgen{3})^{(123)}
              - \LieDgen{2}\circ(\LieDgen{2},\LieDgen{2})
              + \LieDgen{2}\circ(\LieDgen{2},\LieDgen{2})^{(23)}  \\
              &\quad- \LieDgen{2}\circ(\LieDgen{2},\LieDgen{2})^{(132)}
              + \LieDgen{3}\circ(\LieDgen{2},1,1)
              - \LieDgen{3}\circ(1,\LieDgen{2},1)
              + \LieDgen{3}\circ(1,\LieDgen{2},1)^{(12)}
              + \LieDgen{3}\circ(1,1,\LieDgen{2})
              - \LieDgen{3}\circ(1,1,\LieDgen{2})^{(23)}  \\
              &\quad+ \LieDgen{3}\circ(1,1,\LieDgen{2})^{(132)}  ,
            \end{split} \label{eq:LieD:l4}
          \end{equation}
          and analogous to the arity 3 computations, we obtain (indicating by the ellipses $\dotsb$ terms that cancel
          on the nose)
          \begin{align*}
            \rdecomp(\diff \LieDgen{4;1}) &= \rdecomp(- \LieDgen{4} - \LieDgen{4}[(12)])
              = - \rdecomp(\LieDgen{4}) - \rdecomp(\LieDgen{4})^{(12)}  \\
            &= - \left( \LieDgen{2}\circ(\LieDgen{3},1)
                + \LieDgen{2}\circ(1,\LieDgen{3})^{(123)}
                - \LieDgen{2}\circ(\LieDgen{2},\LieDgen{2})
                + \LieDgen{3}\circ(\LieDgen{2},1,1)
                + \LieDgen{3}\circ(1,1,\LieDgen{2})
                + \dotsb \right)  \\
              &\quad- \left( \LieDgen{2}\circ(\LieDgen{3}[(12)],1)
                + \LieDgen{2}\circ(1,\LieDgen{3}[(12)])^{(123)}
                - \LieDgen{2}\circ(\LieDgen{2}[(12)],\LieDgen{2})
                + \LieDgen{3}\circ(\LieDgen{2}[(12)],1,1)
                + \LieDgen{3}[(12)]\circ(1,1,\LieDgen{2})
                + \dotsb \right)  \\
            &= \LieDgen{2}\circ(\diff \LieDgen{3;1},1)
              + \LieDgen{2}\circ(1,\diff \LieDgen{3;1})^{(123)}
              - \LieDgen{2}\circ(\diff \LieDgen{2;1},\LieDgen{2})
              + \LieDgen{3}\circ(\diff \LieDgen{2;1},1,1)
              + (\diff \LieDgen{3;1})\circ(1,1,\LieDgen{2})  \\
            &= \diff\left(- \LieDgen{2}\circ(\LieDgen{3;1},1)
              - \LieDgen{2}\circ(1,\LieDgen{3;1})^{(123)}
              + \LieDgen{2}\circ(\LieDgen{2;1},\LieDgen{2})
              + \LieDgen{3}\circ(\LieDgen{2;1},1,1)
              + \LieDgen{3;1}\circ(1,1,\LieDgen{2}) \right)  .\\
            \rdecomp(\diff \LieDgen{4;2}) &= \rdecomp(- \LieDgen{4} - \LieDgen{4}[(23)])
              = - \rdecomp(\LieDgen{4}) - \rdecomp(\LieDgen{4})^{(23)}  \\
            &= - \left( \LieDgen{2}\circ(\LieDgen{3},1)
                + \LieDgen{2}\circ(1,\LieDgen{3})
                - \LieDgen{2}\circ(\LieDgen{2},\LieDgen{2})^{(132)}
                + \LieDgen{3}\circ(\LieDgen{2},1,1)
                - \LieDgen{3}\circ(1,\LieDgen{2},1)
                + \LieDgen{3}\circ(1,\LieDgen{2},1)^{(12)} \right.  \\
                &\quad\left.{}+ \LieDgen{3}\circ(1,1,\LieDgen{2})^{(132)}
                + \dotsb \right)
              - \left( \LieDgen{2}\circ(\LieDgen{3}[(23)],1)
                + \LieDgen{2}\circ(1,\LieDgen{3}[(12)])
                - \LieDgen{2}\circ(\LieDgen{2}[(12)],\LieDgen{2})^{(132)}
                + \LieDgen{3}[(12)]\circ(1,\LieDgen{2},1)^{(12)} \right.  \\
                &\quad\left.{}- \LieDgen{3}\circ(1,\LieDgen{2}[(12)],1)
                + \LieDgen{3}[(12)]\circ(\LieDgen{2},1,1)
                + \LieDgen{3}[(12)]\circ(1,1,\LieDgen{2})^{(132)}
                + \dotsb \right)  \\
            &= \LieDgen{2}\circ(\diff \LieDgen{3;2},1)
              + \LieDgen{2}\circ(1,\diff \LieDgen{3;1})
              - \LieDgen{2}\circ(\diff \LieDgen{2;1},\LieDgen{2})^{(132)}
              + (\diff \LieDgen{3;1})\circ(\LieDgen{2},1,1)
              + (\diff \LieDgen{3;1})\circ(1,\LieDgen{2},1)^{(12)}  \\
              &\quad- \LieDgen{3}\circ(1,\diff \LieDgen{2;1},1)
              + (\diff \LieDgen{3;1})\circ(1,1,\LieDgen{2})^{(132)}  \\
            &= \diff\left(- \LieDgen{2}\circ(\LieDgen{3;2},1)
              - \LieDgen{2}\circ(1,\LieDgen{3;1})
              + \LieDgen{2}\circ(\LieDgen{2;1},\LieDgen{2})^{(132)}
              + \LieDgen{3;1}\circ(\LieDgen{2},1,1)
              + \LieDgen{3;1}\circ(1,\LieDgen{2},1)^{(12)} \right.  \\
              &\quad\left.{}- \LieDgen{3}\circ(1,\LieDgen{2;1},1)
              + \LieDgen{3;1}\circ(1,1,\LieDgen{2})^{(132)} \right)  .\\
            \rdecomp(\diff \LieDgen{4;3}) &= \rdecomp(- \LieDgen{4} - \LieDgen{4}[(34)])
              = - \rdecomp(\LieDgen{4}) - \rdecomp(\LieDgen{4})^{(34)}  \\
            &= - \left( \LieDgen{2}\circ(\LieDgen{3},1)
                + \LieDgen{2}\circ(1,\LieDgen{3})
                - \LieDgen{2}\circ(1,\LieDgen{3})^{(12)}
                + \LieDgen{2}\circ(1,\LieDgen{3})^{(123)}
                - \LieDgen{2}\circ(\LieDgen{2},\LieDgen{2})
                + \LieDgen{2}\circ(\LieDgen{2},\LieDgen{2})^{(23)} \right.  \\
                &\quad\left.{}- \LieDgen{2}\circ(\LieDgen{2},\LieDgen{2})^{(132)}
                + \LieDgen{3}\circ(\LieDgen{2},1,1)
                - \LieDgen{3}\circ(1,\LieDgen{2},1)
                + \LieDgen{3}\circ(1,\LieDgen{2},1)^{(12)}
                + \LieDgen{3}\circ(1,1,\LieDgen{2})
                - \LieDgen{3}\circ(1,1,\LieDgen{2})^{(23)} \right. \\
                &\quad\left.{}+ \LieDgen{3}\circ(1,1,\LieDgen{2})^{(132)}
                + \dotsb \right)
              - \left( \LieDgen{2}[(12)]\circ(1,\LieDgen{3})^{(123)}
                + \LieDgen{2}\circ(1,\LieDgen{3}[(23)])
                - \LieDgen{2}\circ(1,\LieDgen{3}[(23)])^{(12)}
                + \LieDgen{2}[(12)]\circ(\LieDgen{3},1) \right.  \\
                &\quad\left.{}- \LieDgen{2}\circ(\LieDgen{2},\LieDgen{2}[(12)])
                - \LieDgen{2}[(12)]\circ(\LieDgen{2},\LieDgen{2})^{(132)}
                + \LieDgen{2}[(12)]\circ(\LieDgen{2},\LieDgen{2})^{(23)}
                + \LieDgen{3}[(23)]\circ(\LieDgen{2},1,1)
                - \LieDgen{3}[(23)]\circ(1,1,\LieDgen{2})^{(23)} \right.  \\
                &\quad\left.{}+ \LieDgen{3}[(23)]\circ(1,1,\LieDgen{2})^{(132)}
                + \LieDgen{3}\circ(1,1,\LieDgen{2}[(12)])
                - \LieDgen{3}[(23)]\circ(1,\LieDgen{2},1)
                + \LieDgen{3}[(23)]\circ(1,\LieDgen{2},1)^{(12)}
                + \dotsb \right)  \\
            &= (\diff \LieDgen{2;1})\circ(1,\LieDgen{3})^{(123)}
              + \LieDgen{2}\circ(1,\diff \LieDgen{3;2})
              - \LieDgen{2}\circ(1,\diff \LieDgen{3;2})^{(12)}
              + (\diff \LieDgen{2;1})\circ(\LieDgen{3},1)
              - \LieDgen{2}\circ(\LieDgen{2},\diff \LieDgen{2;1}) \\
              &\quad- (\diff \LieDgen{2;1})\circ(\LieDgen{2},\LieDgen{2})^{(132)}
              + (\diff \LieDgen{2;1})\circ(\LieDgen{2},\LieDgen{2})^{(23)}
              + (\diff \LieDgen{3;2})\circ(\LieDgen{2},1,1)
              - (\diff \LieDgen{3;2})\circ(1,1,\LieDgen{2})^{(23)} \\
              &\quad+ (\diff \LieDgen{3;2})\circ(1,1,\LieDgen{2})^{(132)}
              + \LieDgen{3}\circ(1,1,\diff \LieDgen{2;1})
              - (\diff \LieDgen{3;2})\circ(1,\LieDgen{2},1)
              + (\diff \LieDgen{3;2})\circ(1,\LieDgen{2},1)^{(12)}  \\
            &= \diff\left( \LieDgen{2;1}\circ(1,\LieDgen{3})^{(123)}
                - \LieDgen{2}\circ(1,\LieDgen{3;2})
                + \LieDgen{2}\circ(1,\LieDgen{3;2})^{(12)}
                + \LieDgen{2;1}\circ(\LieDgen{3},1)
                - \LieDgen{2}\circ(\LieDgen{2},\LieDgen{2;1})
                - \LieDgen{2;1}\circ(\LieDgen{2},\LieDgen{2})^{(132)} \right.  \\
                &\quad+ \LieDgen{2;1}\circ(\LieDgen{2},\LieDgen{2})^{(23)}
                + \LieDgen{3;2}\circ(\LieDgen{2},1,1)
                - \LieDgen{3;2}\circ(1,1,\LieDgen{2})^{(23)}
                + \LieDgen{3;2}\circ(1,1,\LieDgen{2})^{(132)}
                + \LieDgen{3}\circ(1,1,\LieDgen{2;1})  \\
                &\quad\left.{}- \LieDgen{3;2}\circ(1,\LieDgen{2},1)
                + \LieDgen{3;2}\circ(1,\LieDgen{2},1)^{(12)} \right)  .\\
          \end{align*}
          As before, this gives us candidates for the definition of $\rdecomp$.  We define
          \begin{align}
            \rdecomp(\LieDgen{4;1}) &\defeq - \LieDgen{2}\circ(\LieDgen{3;1},1)
              - \LieDgen{2}\circ(1,\LieDgen{3;1})^{(123)}
              + \LieDgen{2}\circ(\LieDgen{2;1},\LieDgen{2})
              + \LieDgen{3}\circ(\LieDgen{2;1},1,1)
              + \LieDgen{3;1}\circ(1,1,\LieDgen{2})
              \label{eq:LieD:l41}  \\
            \begin{split}
              \rdecomp(\LieDgen{4;2}) &\defeq
                - \LieDgen{2}\circ(\LieDgen{3;2},1)
                - \LieDgen{2}\circ(1,\LieDgen{3;1})
                + \LieDgen{2}\circ(\LieDgen{2;1},\LieDgen{2})^{(132)}
                + \LieDgen{3;1}\circ(\LieDgen{2},1,1)
                + \LieDgen{3;1}\circ(1,\LieDgen{2},1)^{(12)}  \\
                &\quad- \LieDgen{3}\circ(1,\LieDgen{2;1},1)
                + \LieDgen{3;1}\circ(1,1,\LieDgen{2})^{(132)}
            \end{split} \label{eq:LieD:l42}  \\
            \begin{split}
              \rdecomp(\LieDgen{4;3}) &\defeq \LieDgen{2;1}\circ(1,\LieDgen{3})^{(123)}
                - \LieDgen{2}\circ(1,\LieDgen{3;2})
                + \LieDgen{2}\circ(1,\LieDgen{3;2})^{(12)}
                + \LieDgen{2;1}\circ(\LieDgen{3},1)
                - \LieDgen{2}\circ(\LieDgen{2},\LieDgen{2;1})  \\
                &\quad- \LieDgen{2;1}\circ(\LieDgen{2},\LieDgen{2})^{(132)}
                + \LieDgen{2;1}\circ(\LieDgen{2},\LieDgen{2})^{(23)}
                + \LieDgen{3;2}\circ(\LieDgen{2},1,1)
                - \LieDgen{3;2}\circ(1,1,\LieDgen{2})^{(23)}
                + \LieDgen{3;2}\circ(1,1,\LieDgen{2})^{(132)}  \\
                &\quad+ \LieDgen{3}\circ(1,1,\LieDgen{2;1})
                - \LieDgen{3;2}\circ(1,\LieDgen{2},1)
                + \LieDgen{3;2}\circ(1,\LieDgen{2},1)^{(12)}  .
            \end{split} \label{eq:LieD:l43}
          \end{align}
        \item[Arity \boldmath$n=5$]  We have already defined the decomposition map for $\LieDgen{5}$.  Explicitly, it
          is given by
          \begin{equation}
            \begin{split}
              \rdecomp(\LieDgen{5}) &= - \LieDgen{2}\circ(\LieDgen{4},1)
                + \LieDgen{2}\circ(\LieDgen{3},\LieDgen{2})
                - \LieDgen{2}\circ(\LieDgen{3},\LieDgen{2})^{(34)}
                + \LieDgen{2}\circ(\LieDgen{3},\LieDgen{2})^{(243)}
                - \LieDgen{2}\circ(\LieDgen{3},\LieDgen{2})^{(1432)}
                - \LieDgen{2}\circ(\LieDgen{2},\LieDgen{3})  \\
                &\quad+ \LieDgen{2}\circ(\LieDgen{2},\LieDgen{3})^{(23)}
                - \LieDgen{2}\circ(\LieDgen{2},\LieDgen{3})^{(234)}
                - \LieDgen{2}\circ(\LieDgen{2},\LieDgen{3})^{(132)}
                + \LieDgen{2}\circ(\LieDgen{2},\LieDgen{3})^{(1342)}
                - \LieDgen{2}\circ(\LieDgen{2},\LieDgen{3})^{(13)(24)}  \\
                &\quad+ \LieDgen{2}\circ(1,\LieDgen{4})
                - \LieDgen{2}\circ(1,\LieDgen{4})^{(12)}
                + \LieDgen{2}\circ(1,\LieDgen{4})^{(123)}
                - \LieDgen{2}\circ(1,\LieDgen{4})^{(1234)}
                + \LieDgen{3}\circ(\LieDgen{3},1,1)
                - \LieDgen{3}\circ(\LieDgen{2},\LieDgen{2},1)  \\
                &\quad+ \LieDgen{3}\circ(\LieDgen{2},\LieDgen{2},1)^{(23)}
                - \LieDgen{3}\circ(\LieDgen{2},\LieDgen{2},1)^{(132)}
                + \LieDgen{3}\circ(\LieDgen{2},1,\LieDgen{2})
                + \LieDgen{3}\circ(\LieDgen{2},1,\LieDgen{2})^{(243)}
                - \LieDgen{3}\circ(\LieDgen{2},1,\LieDgen{2})^{(34)}  \\
                &\quad+ \LieDgen{3}\circ(1,\LieDgen{3},1)
                - \LieDgen{3}\circ(1,\LieDgen{3},1)^{(12)}
                + \LieDgen{3}\circ(1,\LieDgen{3},1)^{(123)}
                - \LieDgen{3}\circ(1,\LieDgen{2},\LieDgen{2})
                + \LieDgen{3}\circ(1,\LieDgen{2},\LieDgen{2})^{(34)}  \\
                &\quad- \LieDgen{3}\circ(1,\LieDgen{2},\LieDgen{2})^{(243)}
                + \LieDgen{3}\circ(1,\LieDgen{2},\LieDgen{2})^{(12)}
                - \LieDgen{3}\circ(1,\LieDgen{2},\LieDgen{2})^{(12)(34)}
                + \LieDgen{3}\circ(1,\LieDgen{2},\LieDgen{2})^{(1243)}
                - \LieDgen{3}\circ(1,\LieDgen{2},\LieDgen{2})^{(143)}  \\
                &\quad+ \LieDgen{3}\circ(1,\LieDgen{2},\LieDgen{2})^{(1432)}
                + \LieDgen{3}\circ(1,1,\LieDgen{3})
                - \LieDgen{3}\circ(1,1,\LieDgen{3})^{(23)}
                + \LieDgen{3}\circ(1,1,\LieDgen{3})^{(234)}
                + \LieDgen{3}\circ(1,1,\LieDgen{3})^{(132)}  \\
                &\quad- \LieDgen{3}\circ(1,1,\LieDgen{3})^{(1342)}
                + \LieDgen{3}\circ(1,1,\LieDgen{3})^{(13)(24)}
                - \LieDgen{4}\circ(\LieDgen{2},1,1,1)
                + \LieDgen{4}\circ(1,\LieDgen{2},1,1)
                - \LieDgen{4}\circ(1,\LieDgen{2},1,1)^{(12)}  \\
                &\quad- \LieDgen{4}\circ(1,1,\LieDgen{2},1)
                + \LieDgen{4}\circ(1,1,\LieDgen{2},1)^{(23)}
                - \LieDgen{4}\circ(1,1,\LieDgen{2},1)^{(132)}
                + \LieDgen{4}\circ(1,1,1,\LieDgen{2})
                - \LieDgen{4}\circ(1,1,1,\LieDgen{2})^{(34)}  \\
                &\quad+ \LieDgen{4}\circ(1,1,1,\LieDgen{2})^{(243)}
                - \LieDgen{4}\circ(1,1,1,\LieDgen{2})^{(1432)}  .
            \end{split}
            \label{eq:LieD:l5}
          \end{equation}
          Since there are no higher degree terms in $\LieD[3](5)$, we are done here.
      \end{description}

      \begin{lm} \label{lm:LieD:coopd}
        The triple $(\LieD[3],\diff,\decomp)$ consisting of the dg $\Sy$-module $(\LieD[3],\diff)$ defined in
        \cref{eq:LieD:cx} and the decomposition map as defined by
        \cref{eq:LieD:l21x,eq:LieD:l3,eq:LieD:l31,eq:LieD:l32,eq:LieD:l311,eq:LieD:l312,eq:LieD:l322,eq:LieD:l4,eq:LieD:l41,eq:LieD:l42,eq:LieD:l43,eq:LieD:l5}
        is a dg cooperad.
      \end{lm}

      The decomposition structure map we defined is compatible with the differential by its construction.  What is
      left to do, in order to show that it defines a dg cooperad structure on the dg $\Sy$-module $\LieD[3]$ defined
      in the previous section, is to check that it satisfies the coassociativity condition
      $(\decomp\circ\id)\decomp=(\id\circ\decomp)\decomp$.  Note that coassociativity is automatic in arities $n\leq
      3$ and we already know that it holds for $\LieDgen{n} \in \LieD[3](n)_{n-1} = \LeibK(n)$.  Thus, it is
      sufficient to check it for the elements $\LieDgen{4;1}$, $\LieDgen{4;2}$, and $\LieDgen{4;3}$.  See
      \Cref{A:comp:coass} for this long and tedious computation.

    \subsection{The twisted composite product} \label{S:LieD:KosCx}
      As explained in the introduction to this section, if we could show that for an $\Sy$-free resolution
      $\psi\colon \LieD \to \LieK$ of dg cooperads the map $\COBAR\psi$ is a quasi-isomorphism, this would give us a
      cofibrant resolution of the Lie operad over $\ZZ$.  One way to do this would be to show, that the twisted
      composite product $\LieD\circ_{\tilde\kappa}\Lie$ is acyclic for $\tilde\kappa = \kappa\circ\psi$.  Since we do
      not have the full dg cooperad $\LieD$, we will show the following truncated statement.  This is of course a
      neccessary condition for our dg cooperad $\LieD[3]$ to be a truncation of an $\Sy$-free resolution $\LieD$ with
      the desired properties.

      \begin{pp}\label{pp:KosCx:acyclic}
        The twisted composite product $\LieD[3]\circ_{\tilde\kappa}\Lie$ satisfies
        \begin{equation*}
          \Hm_r\big((\LieD[3]\circ_{\tilde\kappa}\Lie)(n)\big) = 0 ,
        \end{equation*}
        for all $r \leq 3$ in all arities $n$.
      \end{pp}

      \begin{proof}
        Since the result actually holds for any $\LieD[k]$ constructed as in the previous sections, we phrase the
        proof for arbitrary $k$ instead of just $k=3$.

        Consider the bigrading on the composite product $\LieD[k]\circ\Lie$ given by
        \begin{equation*}
          \big(\LieD[k]\circ\Lie\big)(n)_{p,q}
          = \Big( \LieD[k](p+1) \tensor_{\Sy_{p+1}} \big( \Lie^{\tensor p+1} \big)(n) \Big)_{p+q}  .
        \end{equation*}
        The differential of this composite product, $\diff^{\LieD[k]\circ\Lie} = \diff^{\LieD[k]} \circ 1$, is of
        bidegree $(0,-1)$.  On the twisted composite product $\LieD[k]\circ_{\tilde\kappa}\Lie$, the differential has
        another term, $\diff_{\tilde\kappa}$, defined as in \cref{eq:TwDiff}.  Note that $\diff_{\tilde\kappa}$ is of
        bidegree $(-1,0)$, since $\tilde\kappa \defeq \kappa\circ\psi$ vanishes everywhere except on $\LieD[k](2)_1$.
        Thus, with the above bigrading, the twisted composite product becomes a first quadrant bicomplex.  Below, we
        consider its spectral sequence $E(n)$ for each arity $n$.

        Since $\Lie(0) = 0$, the action of $\Sy_{p+1}$ on $\Lie^{\tensor p+1}(n)$ is free and hence $E(n)^0_{p,q}$
        admits the following expansion,
        \begin{align*}
          E(n)^0_{p,q} &= \big(\LieD[k]\circ\Lie\big)(n)_{p,q}  \\
          &= \Big( \LieD[k](p+1) \tensor_{\Sy_{p+1}} \big( \Lie^{\tensor p+1} \big)(n) \Big)_{p+q}  \\
          &= \left( \Dsum_{n=k_1+\dotsb+k_{p+1}} \LieD[k](p+1) \tensor_{\Sy_{p+1}} \left(
              \mathrm{Ind}_{\Sy_{k_1}\times\dotsb\times\Sy_{k_{p+1}}}^{\Sy_n}
                \Lie(k_1)\tensor\dotsb\tensor\Lie(k_{p+1})
            \right) \right)_{p+q}  \\
          &= \left( \Dsum_{n=k_1+\dotsb+k_m} \LieD[k](p+1) \tensor
            \Lie(k_1)\tensor\dotsb\tensor\Lie(k_{p+1}) \tensor \Fk[\Sh(k_1,\dotsc,k_{p+1})] \right)_{p+q}  .
        \end{align*}
        Since $\Lie$ is $\Fk$-projective, i.e.\@ $\Lie(n)$ are projective $\Fk$-modules for all arities $n$, this
        implies that the first page $E(n)^1_{p,q}$ is given by
        \begin{align*}
          E(n)^1_{p,q}
          &= \Hm_{p+q}\Big( \LieD[k](p+1) \tensor_{\Sy_{p+1}} \big( \Lie^{\tensor p+1} \big)(n) ,
              \diff^{\LieD[k]} \circ 1 \Big)  \\
          &= \Hm_{p+q}\big( \LieD[k](p+1) , \diff^{\LieD[k]} \big)
              \tensor_{\Sy_{p+1}} \big( \Lie^{\tensor p+1} \big)(n)  .
        \end{align*}
        By the construction of \Cref{S:LieD:SMod}, we have for $n \leq k+1$:
        \begin{align*}
          \Hm_r\big(\LieD[k](n)\big) = \begin{cases}
              \LieK(n)_r    & \textrm{if $r \leq k$,}  \\
              \ker\diff_r   & \textrm{if $r = k+1$,}  \\
              0             & \textrm{otherwise.}
            \end{cases}
        \end{align*}
        Hence, for $p+q \leq k$ we obtain
        \begin{equation*}
          E(n)^1_{p,q}
            = \Big( \LieK(p+1) \tensor_{\Sy_{p+1}} \big( \Lie^{\tensor p+1} \big)(n) \Big)_{p+q}
            = \begin{cases}
                \big( \LieK \circ \Lie \big)(n)_{p} & \textrm{if $q=0$,}  \\
                0 & \textrm{otherwise.}
              \end{cases}
        \end{equation*}
        Since $\psi$ is a morphism of dg cooperads and $\tilde\kappa = \kappa\circ\psi$, it follows from the
        definition of the twisted differential in \cref{eq:TwDiff}, that the following diagram commutes:
        \begin{equation*}
          \begin{tikzcd}
            \LieD[k]\circ\Lie \ar[r,"\diff^{\tilde\kappa}"] \ar[d,"\psi\circ\id"]
            & \LieD[k]\circ\Lie \ar[d,"\psi\circ\id"]  \\
            \LieK\circ\Lie \ar[r,"\diff_{\kappa}"]
            & \LieK\circ\Lie  .
          \end{tikzcd}
        \end{equation*}
        Since $\psi$ is a quasi-isomorphism, this shows that $\diff^1_{p,0} = (\diff_\kappa)_p$ for $p \leq k+1$.
        Thus we obtain $E^2_{p,q} = 0$ for $p+q \leq k$, since the operad $\Lie$ is Koszul.

        The proof for $n > k+1$ is essentially the same, except for the computation of $E(n)^2_{k,0}$.  We have
        \begin{equation*}
          \Hm_r\big(\LieD[k](k+2)\big) = \LieD[k](k+2)_r = \begin{cases}
              \LieD[k](k+2)_{k+1} & \textrm{if $r=k+1$,}  \\
              0 & \textrm{otherwise,}
            \end{cases}
        \end{equation*}
        and hence $\diff^1_{k+1,0} \neq (\diff_\kappa)_{k+1}$.
        However, we do find
        \begin{align*}
          \im\left( E(n)^1_{k+1,0} \xto{\diff^1_{k+1,0}} E(n)^1_{k,0} \right)
          &= \im\left( \big(\LieD[k]\circ\Lie\big)(n)_{k+1}
            \xto{(\psi\circ\id)\circ\diff_{\tilde\kappa}} \big(\LieK\circ\Lie\big)(n)_k \right)  \\
          &= \im\left( \big(\LieK\circ\Lie\big)(n)_{k+1}
            \xto{\diff_\kappa} \big(\LieK\circ\Lie\big)(n)_k \right)  ,
        \end{align*}
        since $\psi$ is surjective.  Thus we obtain $E^2_{p,q} = 0$ for $p+q \leq k$ again.
      \end{proof}

  \section{The category of weak Lie 3-algebras} \label{S:ELie3}
    Assume again that we had an $\Sy$-free resolution $\psi\colon \LieD \to \LieK$ of dg cooperads, and in addition
    that $\LieD\circ_{(\kappa\circ\psi)}\Lie$ is acyclic and therefore $\ELinf = \COBAR\LieD \qito \Lie$ is a
    cofibrant resolution.  Now consider a 3-term complex
    $(L,\diff) = L_0 \xleftarrow{\diff} L_1 \xleftarrow{\diff} L_2$ and note that its endomorphism operad $\End_L$
    vanishes in degrees $r > 2$.  Since the Maurer--Cartan equation for a twisting morphism
    $\LieDtw\colon \LieD \to \End_L$ is of degree $-2$, only degrees $\leq 4$ of $\LieD$ play a role in the definition
    of 3-term $\ELinf$-algebras.  A homotopy transfer theorem for such 3-term $\ELinf$-algebras holds as a special
    case of the general HTT for algebras over a cofibrant operad.  For our resolution $\LieD[3]$, however, we do not
    know that it is the truncation of such a dg cooperad $\LieD$.  Nonetheless we show constructively that our weak
    Lie 3-algebras satisfy a homotopy transfer theorem (\Cref{pp:ELie3:HTT}.)

    This section is organized as follows.  In \Cref{S:ELie3:htyLeib}, we recall the definition of a homotopy Leibniz
    algebra and of its homotopy morphisms.  We spell out the details for the case of a homotopy Leibniz algebra on a
    3-term complex or Leibniz 3-algebra, and for morphisms between such.  In \Cref{S:ELie3:Alg}, we introduce the
    notion of a weak Lie 3-algebra as extra structure on a Leibniz 3-algebra.  In \Cref{S:ELie3:Mor}, the accompanying
    notion of morphisms is made explicit.  Finally, in \Cref{S:ELie3:HTT}, we prove a version of the homotopy transfer
    theorem for weak Lie 3-algebras.

    \subsection{Homotopy Leibniz algebras} \label{S:ELie3:htyLeib}
      From the description of the Koszul dual cooperad $\LeibK$ in \Cref{S:LieD:LieKLeibK}, we obtain the following
      explicit definitions of algebras and morphisms over $\Leibinf = \COBAR\LeibK$ via the Maurer--Cartan
      \cref{eq:Tw,eq:htymor:MCE}.  For a more thorough exposition, we refer the reader to
      \cite{khudaverdyan2014infinity}.

      \begin{df} 
        A \emph{homotopy Leibniz algebra} or \emph{$\Leibinf$-algebra} $(L,\diff,\LeibKtw)$ consists of a chain
        complex $(L,\diff)$ with structure maps
        \begin{align*}
          &\LeibKstr{n}\colon L^{\tensor n} \to L[n-2] , && \forall n\geq 2 ,
        \end{align*}
        satisfying the following generalized Jacobi identities:
        \begin{align*}
          \partial(\LeibKstr{n}) &= - \sum_{i+j=n+1} (-1)^{(j-1)i}
            \sum_{p=1}^j (-1)^{(p-1)(i-1)}
            \sum_{\sigma\in\ruSh(p-1,i)} (-1)^{|\sigma|} \cdot (\LeibKstr{j} \tensor_p \LeibKstr{i})^\sigma ,
          && \forall n\geq 3 .
        \end{align*}

        Let $(L,\diff,\LeibKtw)$, $(L',\diff',\LeibKtw')$ be $\Leibinf$-algebras.  A \emph{homotopy morphism} or
        \emph{$\Leibinf$-morphism} $f\colon L \to L'$ consists of maps
        \begin{align*}
          &f_n\colon L^{\tensor n} \to L'[n-1] , && \forall n\geq 1 ,
        \end{align*}
        satisfying the following equations:
        \begin{align*}
          \begin{split}
          \partial(f_n)
            &= \sum_{i+j=n+1} (-1)^{(j-1)i} \sum_{p=1}^j (-1)^{(p-1)(i-1)}
              \sum_{\sigma\in\ruSh(p-1,i)} (-1)^{|\sigma|} \cdot (f_j \tensor_p \LeibKstr{i})^\sigma  \\
            &\quad- \sum_{\substack{1\leq j\leq n\\i_1+\dotsb+i_j=n}}
                (-1)^{(j-1)(n-j)} \cdot (-1)^{\sum_{p=1}^j (p-1)(i_p-1)}
              \sum_{\sigma\in\ruSh(i_1,\dotsc,i_j)}
                (-1)^{|\sigma|} \cdot \LeibKstr{j}[\prime]\circ(f_{i_1},\dotsc,f_{i_j})^\sigma  .
          \end{split}
        \end{align*}
      \end{df}

      \subsubsection{Leibniz 3-algebras}  A Leibniz 3-algebra is just a homotopy Leibniz algebra on a 3-term complex.
      Since this is the foundation for our definition of weak Lie 3-algebra, we spell out the definition here.

      \begin{df} 
        A \emph{Leibniz 3-algebra} $(L,\diff,\LeibKtw)$ consists of a 3-term chain complex
        $(L,\diff) = L_0 \xleftarrow{\diff} L_1 \xleftarrow{\diff} L_2$, equipped with structure maps
        \begin{align*}
          &\LeibKstr{2}\colon L^{\tensor 2} \to L ,
          &&\LeibKstr{3}\colon L^{\tensor 3} \to L[1] ,
          &&\LeibKstr{4}\colon L^{\tensor 4} \to L[2] ,
        \end{align*}
        satisfying the following generalized Jacobi identities:
        \begin{align}
          \dEnd(\LeibKstr{2})  &=  0,
            \label{eq:Leib3:l2}  \\
          \dEnd(\LeibKstr{3})  &=
            \LeibKstr{2}\circ_2\LeibKstr{2} - \LeibKstr{2}\circ_1\LeibKstr{2}
            - (\LeibKstr{2}\circ_2\LeibKstr{2})^{(12)}  ,
            \label{eq:Leib3:l3}  \\
          \begin{split}
            \dEnd(\LeibKstr{4}) &= \LeibKstr{2}\circ_1\LeibKstr{3} + \LeibKstr{2}\circ_2\LeibKstr{3}
              - (\LeibKstr{2}\circ_2\LeibKstr{3})^{(12)}
              + (\LeibKstr{2}\circ_2\LeibKstr{3})^{(123)} - \LeibKstr{3}\circ_1\LeibKstr{2}  \\
              &\quad+ \LeibKstr{3}\circ_2\LeibKstr{2} - (\LeibKstr{3}\circ_2\LeibKstr{2})^{(12)}
              - \LeibKstr{3}\circ_3\LeibKstr{2} + (\LeibKstr{3}\circ_3\LeibKstr{2})^{(23)}
              - (\LeibKstr{3}\circ_3\LeibKstr{2})^{(132)}  ,
          \end{split} \label{eq:Leib3:l4}  \\
          \begin{split}
            0 &= \LeibKstr{2}\circ_1\LeibKstr{4} - \LeibKstr{2}\circ_2\LeibKstr{4}
              + (\LeibKstr{2}\circ_2\LeibKstr{4})^{(12)}
              - (\LeibKstr{2}\circ_2\LeibKstr{4})^{(123)} + (\LeibKstr{2}\circ_2\LeibKstr{4})^{(1234)}  \\
              &\quad+ \LeibKstr{3}\circ_1\LeibKstr{3} + \LeibKstr{3}\circ_2\LeibKstr{3}
              - (\LeibKstr{3}\circ_2\LeibKstr{3})^{(12)}
              + (\LeibKstr{3}\circ_2\LeibKstr{3})^{(123)} + \LeibKstr{3}\circ_3\LeibKstr{3}  \\
              &\quad- (\LeibKstr{3}\circ_3\LeibKstr{3})^{(23)} + (\LeibKstr{3}\circ_3\LeibKstr{3})^{(132)}
              + (\LeibKstr{3}\circ_3\LeibKstr{3})^{(234)} - (\LeibKstr{3}\circ_3\LeibKstr{3})^{(1342)}
              + (\LeibKstr{3}\circ_3\LeibKstr{3})^{(13)(24)}  \\
              &\quad+ \LeibKstr{4}\circ_1\LeibKstr{2} - \LeibKstr{4}\circ_2\LeibKstr{2}
              + (\LeibKstr{4}\circ_2\LeibKstr{2})^{(12)}
              + \LeibKstr{4}\circ_3\LeibKstr{2} - (\LeibKstr{4}\circ_3\LeibKstr{2})^{(23)}  \\
              &\quad+ (\LeibKstr{4}\circ_3\LeibKstr{2})^{(132)} - \LeibKstr{4}\circ_4\LeibKstr{2}
              + (\LeibKstr{4}\circ_4\LeibKstr{2})^{(34)} - (\LeibKstr{4}\circ_4\LeibKstr{2})^{(243)}
              + (\LeibKstr{4}\circ_4\LeibKstr{2})^{(1432)}  .
          \end{split} \label{eq:Leib3:l5}
        \end{align}
      \end{df}

      \subsubsection{Morphisms of Leibniz 3-algebras}  A morphism of Leibniz 3-algebras is just a homotopy morphism
      between 3-term homotopy Leibniz algebras.

      \begin{df}
        Let $(L,\diff,\LeibKtw)$, $(L',\diff',\LeibKtw')$ be Leibniz 3-algebras.  A \emph{weak morphism}
        $f\colon L \to L'$ consists of maps
        \begin{align}
          &f_1\colon L \to L' ,
          &&f_2\colon L^{\tensor 2} \to L'[1] ,
          &&f_3\colon L^{\tensor 3} \to L'[2] ,
        \end{align}
        satisfying the following equations,
        \begin{align}
          \dEnd(f_1) &= 0 ,
            \label{eq:Leib3:f1}  \\
          \dEnd(f_2) &= f_1\circ\LeibKstr{2} - \LeibKstr{2}[\prime]\circ(f_1,f_1) ,
            \label{eq:Leib3:f2}  \\
          \begin{split}
            \dEnd(f_3) &= f_1\circ\LeibKstr{3}
              - f_2\circ_2\LeibKstr{2} + f_2\circ_1\LeibKstr{2} + (f_2\circ_2\LeibKstr{2})^{(12)}  \\
              &\quad- \LeibKstr{3}[\prime]\circ(f_1,f_1,f_1)
              - \LeibKstr{2}[\prime]\circ(f_1,f_2) + \LeibKstr{2}[\prime]\circ(f_2,f_1) + \LeibKstr{2}[\prime]\circ(f_1,f_2)^{(12)} ,
          \end{split} \label{eq:Leib3:f3}  \\
          \begin{split}
            f_1\circ\LeibKstr{4} - \LeibKstr{4}[\prime]\circ(f_1,f_1,f_1,f_1)
            &= f_2\circ_1\LeibKstr{3} + f_2\circ_2\LeibKstr{3} - (f_2\circ_2\LeibKstr{3})^{(12)}
              + (f_2\circ_2\LeibKstr{3})^{(123)} - f_3\circ_1\LeibKstr{2}  \\
              &\quad+ f_3\circ_2\LeibKstr{2} - (f_3\circ_2\LeibKstr{2})^{(12)} - f_3\circ_3\LeibKstr{2}
              + (f_3\circ_3\LeibKstr{2})^{(23)} - (f_3\circ_3\LeibKstr{2})^{(132)}  \\
              &\quad+ \LeibKstr{2}[\prime]\circ(f_3,f_1) + \LeibKstr{2}[\prime]\circ(f_1,f_3) - \LeibKstr{2}[\prime]\circ(f_1,f_3)^{(12)}
              + \LeibKstr{2}[\prime]\circ(f_1,f_3)^{(123)}  \\
              &\quad- \LeibKstr{2}[\prime]\circ(f_2,f_2) + \LeibKstr{2}[\prime]\circ(f_2,f_2)^{(23)} - \LeibKstr{2}[\prime]\circ(f_2,f_2)^{(132)}
              + \LeibKstr{3}[\prime]\circ(f_2,f_1,f_1)  \\
              &\quad- \LeibKstr{3}[\prime]\circ(f_1,f_2,f_1) + \LeibKstr{3}[\prime]\circ(f_1,f_2,f_1)^{(12)}
              + \LeibKstr{3}[\prime]\circ(f_1,f_1,f_2)  \\
              &\quad- \LeibKstr{3}[\prime]\circ(f_1,f_1,f_2)^{(23)} + \LeibKstr{3}[\prime]\circ(f_1,f_1,f_2)^{(132)} .
          \end{split} \label{eq:Leib3:f4}
        \end{align}

        Let $L \xto{f} L' \xto{f'} L''$ be weak morphisms of Leibniz 3-algebras.  Their composition is defined by the
        following components:
        \begin{align}
          (f'\circ f)_1 &\defeq f'_1\circ f_1 ,
          \label{eq:Leib3:ff1}  \\
          (f'\circ f)_2 &\defeq f'_2\circ(f_1,f_1) + f'_1\circ f_2 ,
          \label{eq:Leib3:ff2}  \\
          (f'\circ f)_3 &\defeq f'_3\circ(f_1,f_1,f_1) - f'_2\circ(f_2,f_1) + f'_2\circ(f_1,f_2)
            - f'_2\circ(f_1,f_2)^{(12)} + f'_1\circ f_3 .
          \label{eq:Leib3:ff3}
        \end{align}
      \end{df}

    \subsection{Weak Lie 3-algebras} \label{S:ELie3:Alg}
      In this section, we spell out the definition of weak Lie 3-algebras as solutions to the Maurer--Cartan
      \cref{eq:Tw} on a 3-term complex.  To do this, we evaluate the Maurer--Cartan equation for a twisting morphism
      $\LieDtw\colon \LieD[3] \to \End_L$ on the $\Fk[\Sy_n]$-generators introduced in \Cref{S:LieD:SMod}, and use the
      shorthand notation $\LieDstr{*}$ for $\LieDtw(\LieDgen{*})$.  E.g.\@ for $\LieDgen{3;1}$, we find
      \begin{align*}
        0 = \big(\partial\LieDtw+\LieDtw\star\LieDtw\big)(\LieDgen{3;1})
        &= \dEnd\big(\LieDtw(\LieDgen{3;1})\big) + \LieDtw(\diff\LieDgen{3;1})
          + \big(\gamma_{(1)}\circ(\LieDtw\pcirc\LieDtw)\circ\pdecomp\big)(\LieDgen{3;1})  \\
        &= \dEnd(\LieDstr{3;1}) + \LieDtw\big(- \LieDgen{3} - \LieDgen{3}[(12)]\big)
          + \big(\gamma_{(1)}\circ(\LieDtw\pcirc\LieDtw)\big)(\LieDgen{2}\circ_1\LieDgen{2;1})  \\
        &= \dEnd(\LieDstr{3;1}) - \LieDstr{3} - \LieDstr{3}[(12)] - \LieDstr{2}\circ_1\LieDstr{2;1}  ,
      \end{align*}
      which is equivalent to \cref{eq:ELie3:l31}.  Doing this for all generators of $\LieD[3]$ and restricting to 3-term
      complexes leads to the following Definition.

      \begin{df} \label{df:wLie3}
        A \emph{weak Lie 3-algebra} $(L,\diff,\LieDtw)$ is a 3-term chain complex
        $(L,\diff) = L_0 \xleftarrow{\diff} L_1 \xleftarrow{\diff} L_2$, equipped with structure maps
        \begin{align*}
          \LieDstr{2}       &\colon  L^{\tensor 2} \to L ,     &
          \LieDstr{2;1}     &\colon  L^{\tensor 2} \to L[1] ,  &
          \LieDstr{2;1,1}   &\colon  L^{\tensor 2} \to L[2] ,  \\
          &&
          \LieDstr{3}       &\colon  L^{\tensor 3} \to L[1] ,  &
          \LieDstr{3;1}     &\colon  L^{\tensor 3} \to L[2] ,  \\
          &&
          &&
          \LieDstr{3;2}     &\colon  L^{\tensor 3} \to L[2] ,  \\
          &&
          &&
          \LieDstr{4}       &\colon  L^{\tensor 4} \to L[2] .
        \end{align*}
        We require these to satisfy the \cref{eq:Leib3:l2,eq:Leib3:l3,eq:Leib3:l4,eq:Leib3:l5}, i.e.\@
        $(L,\diff,\LeibKtw)$ to be a Leibniz 3-algebra, and in addition we require the following
        equations to hold,
        \begin{align}
          \dEnd(\LieDstr{2;1}) &= \LieDstr{2} + \LieDstr{2}[(12)] ,
            \label{eq:ELie3:l21}  \\
          \dEnd(\LieDstr{2;1,1}) &= \LieDstr{2;1} - \LieDstr{2;1}[(12)] ,
            \label{eq:ELie3:l211}  \\
          \dEnd(\LieDstr{3;1}) &= \LieDstr{3} + \LieDstr{3}[(12)] + \LieDstr{2}\circ_1\LieDstr{2;1} ,
            \label{eq:ELie3:l31}  \\
          \dEnd(\LieDstr{3;2}) &= \LieDstr{3} + \LieDstr{3}[(23)] - \LieDstr{2}\circ_2\LieDstr{2;1}
						+ \LieDstr{2;1}\circ_1\LieDstr{2} + (\LieDstr{2;1}\circ_2\LieDstr{2})^{(12)} ,
            \label{eq:ELie3:l32}  \\
          \LieDstr{2;1,1} + \LieDstr{2;1,1}[(12)] &= 0,
            \label{eq:ELie3:l2111}  \\
          \LieDstr{3;1} - \LieDstr{3;1}[(12)] &= \LieDstr{2}\circ_1\LieDstr{2;1,1} ,
            \label{eq:ELie3:l311}  \\
          \begin{split}
            \LieDstr{3;1} - \LieDstr{3;2}[(12)] &+ \LieDstr{3;1}[(132)]
              - \LieDstr{3;2} + \LieDstr{3;1}[(23)] - \LieDstr{3;2}[(123)]  \\
              &= \LieDstr{2;1}\circ_2\LieDstr{2;1} + \LieDstr{2;1}\circ_1\LieDstr{2;1}
              + (\LieDstr{2;1}\circ_2\LieDstr{2;1})^{(12)} + (\LieDstr{2;1,1}\circ_2\LieDstr{2})^{(132)} ,
          \end{split} \label{eq:ELie3:l312}  \\
          \LieDstr{3;2} - \LieDstr{3;2}[(23)] &= - \LieDstr{2}\circ_2\LieDstr{2;1,1}
            + \LieDstr{2;1,1}\circ_1\LieDstr{2} + (\LieDstr{2;1,1}\circ_2\LieDstr{2})^{(12)} ,
            \label{eq:ELie3:l322}  \\
          \LieDstr{4} + \LieDstr{4}[(12)] &= \LieDstr{2}\circ_1\LieDstr{3;1} - \LieDstr{3;1}\circ_3\LieDstr{2}
            + (\LieDstr{2}\circ_2\LieDstr{3;1})^{(123)} + \LieDstr{3}\circ_1\LieDstr{2;1} ,
            \label{eq:ELie3:l41}  \\
          \begin{split}
            \LieDstr{4} + \LieDstr{4}[(23)]  &=  \LieDstr{2}\circ_1\LieDstr{3;2}
              - \LieDstr{3}\circ_2\LieDstr{2;1} + \LieDstr{2}\circ_2\LieDstr{3;1}
              - \LieDstr{3;1}\circ_1\LieDstr{2} - (\LieDstr{3;1}\circ_2\LieDstr{2})^{(12)}
              - (\LieDstr{3;1}\circ_3\LieDstr{2})^{(132)} ,
          \end{split} \label{eq:ELie3:l42}  \\
          \begin{split}
            \LieDstr{4} + \LieDstr{4}[(34)] &= \LieDstr{2;1}\circ_1\LieDstr{3} - \LieDstr{3;2}\circ_1\LieDstr{2}
              + \LieDstr{3;2}\circ_2\LieDstr{2} - (\LieDstr{3;2}\circ_2\LieDstr{2})^{(12)}
              + \LieDstr{3}\circ_3\LieDstr{2;1}
              + (\LieDstr{3;2}\circ_3\LieDstr{2})^{(23)} \\&\quad - (\LieDstr{3;2}\circ_3\LieDstr{2})^{(132)}
              + \LieDstr{2}\circ_2\LieDstr{3;2} - (\LieDstr{2}\circ_2\LieDstr{3;2})^{(12)}
              + (\LieDstr{2;1}\circ_2\LieDstr{3})^{(123)} .
          \end{split} \label{eq:ELie3:l43}
        \end{align}
      \end{df}

      When in the above definition we assume $L_2 = 0$, this forces the structure maps $\LieDstr{2;1,1}$,
      $\LieDstr{3;1}$, $\LieDstr{3;2}$, and $\LieDstr{4}$ to vanish for degree reasons.  The
      \cref{eq:ELie3:l2111,eq:ELie3:l311,eq:ELie3:l312,eq:ELie3:l322,eq:ELie3:l41,eq:ELie3:l42,eq:ELie3:l43,eq:Leib3:l5} hold for degree
      reasons in this case, and the left-hand sides of \cref{eq:ELie3:l211,eq:ELie3:l31,eq:ELie3:l32,eq:Leib3:l4} become $0$.  In
      this way, we recover Roytenberg's definition of a \emph{2-term $\ELinf$-algebra}
      \cite[Definition~2.16]{roytenberg_weak_2007} (which we will call \emph{weak Lie 2-algebras}.)

    \subsection{Morphisms of weak Lie 3-algebras} \label{S:ELie3:Mor}
      Since weak Lie 3-algebras are algebras over the operad $\COBAR\LieD[3]$, they come with a general notion of
      morphism of operadic algebras.  Such a morphism of weak Lie 3-algebras $L \to L'$ is a morphisms of the
      underlying chain complexes $(L,\diff) \to (L',\diff')$ commuting with all structure maps.  This notion is of
      limited use however and we will introduce another type of morphisms below.

      Consider again weak Lie 3-algebras $L$ and $L'$, i.e.\@ chain complexes $(L,\diff)$, $(L',\diff')$ equipped with
      structure maps given by twisting morphisms
      \begin{align*}
        &  \LieDtw \in \Tw(\LieD[3],\End_L) ,
        && \LieDtw' \in \Tw(\LieD[3],\End_{L'}) .
      \end{align*}
      Following the general theory described in \Cref{S:Prelim:wMor}, a \emph{weak morphism} of weak Lie $3$-algebras
      is a degree 0 solution to the following Maurer--Cartan equation:
      \begin{align}
        f &\colon \LieD[3] \to \End^L_{L'} ,
        && \partial(f) - f \ast \LieDtw + \LieDtw' \circledast f = 0 .
        \label{eq:ELie3:Mor:Tw}
      \end{align}
      We again evaluate this Maurer--Cartan equation for the $\Fk[\Sy_n]$-generators of $\LieD[3]$, using the
      short-hand notation $f_*$ for $f(\LieDgen{*})$.  E.g.\@ for $\LieDgen{2;1}$, this amounts to
      \begin{align*}
        0 = \big( \partial(f) - f \ast \LieDtw + \LieDtw' \circledast f \big)(\LieDgen{2;1})
        &= \dEnd\big(f(\LieDgen{2;1})\big) - f(\diff\LieDgen{2;1})
          - \big( f \ast \LieDtw \big)(\LieDgen{2;1}) + \big( \LieDtw' \circledast f )(\LieDgen{2;1})  \\
        &= \dEnd(f_{2;1}) + f(\LieDgen{2}) + f(\LieDgen{2})^{(12)}
          - f(\id)\circ_{1}\LieDtw(\LieDgen{2;1}) + \LieDtw'(\LieDgen{2;1})\circ(f(\id),f(\id))  \\
          &= \dEnd(f_{2;1}) + f_2 + f_2^{(12)} - f_1\circ\LieDstr{2;1} + \LieDstr{2;1}[\prime]\circ(f_1,f_1) ,
      \end{align*}
      which is equivalent to \cref{eq:ELie3:f21}.  In summary, we obtain the following definition.

      \begin{df}
        A \emph{weak morphism} of weak Lie 3-algebras $f\colon (L,\diff,\LieDtw) \to (L',\diff',\LieDtw')$ consists of
        a collection of $\Fk$-linear maps,
        \begin{align*}
          f_1       &\colon  L \to L'                , &
          f_2       &\colon  L^{\tensor 2} \to L'[1] , &
          f_{2;1}   &\colon  L^{\tensor 2} \to L'[2] , \\
          &&
          &&
          f_3       &\colon  L^{\tensor 3} \to L'[2] .
        \end{align*}
        We assume that the \cref{eq:Leib3:f1,eq:Leib3:f2,eq:Leib3:f3,eq:Leib3:f4} hold, i.e.\@ $f$ is a morphism of
        Leibniz 3-algebras.  In addition the maps are required to satisfy the following equations:
        \begin{align}
          \dEnd(f_{2;1}) &= - f_2 - f_2^{(12)} + f_1\circ\LieDstr{2;1} - \LieDstr{2;1}[\prime]\circ(f_1,f_1) ,
            \label{eq:ELie3:f21}  \\
          f_1\circ\LieDstr{2;1,1} - \LieDstr{2;1,1}[\prime]\circ(f_1,f_1) &= f_{2;1} - f_{2;1}^{(12)} ,
            \label{eq:ELie3:f211}  \\
          f_1\circ\LieDstr{3;1} - \LieDstr{3;1}[\prime]\circ(f_1,f_1,f_1)
            &= f_3 + f_3^{(12)} + f_2\circ_1\LieDstr{2;1} + \LieDstr{2}[\prime]\circ(f_{2;1},f_1) ,
            \label{eq:ELie3:f31}  \\
          \begin{split}
            f_1\circ\LieDstr{3;2} - \LieDstr{3;2}[\prime]\circ(f_1,f_1,f_1)
              &= f_3 + f_3^{(23)} - f_2\circ_2\LieDstr{2;1}
                + f_{2;1}\circ_1\LieDstr{2} + (f_{2;1}\circ_2\LieDstr{2})^{(12)}  \\
              &\quad- \LieDstr{2}[\prime]\circ(f_1,f_{2;1}) - \LieDstr{2;1}[\prime]\circ(f_2,f_1)
                - \LieDstr{2;1}[\prime]\circ(f_1,f_2)^{(12)} .
          \end{split} \label{eq:ELie3:f32}
        \end{align}

        Let $L \xto{f} L' \xto{f'} L''$ be weak morphisms of weak Lie 3-algebras.  Their composition is defined to be
        the composition of the underlying weak morphisms of Leibniz 3-algebras, i.e.\@ by
        \cref{eq:Leib3:ff1,eq:Leib3:ff2,eq:Leib3:ff3}, with the additional component
        \begin{equation}
          (f'\circ f)_{2;1} \defeq f'_{2;1}\circ(f_1,f_1) + f'_1\circ f_{2;1} .
          \label{eq:ELie3:ff21}
        \end{equation}
      \end{df}

      If we assume that $L$ and $L'$ in the above definition are in fact weak Lie 2-algebras, then for degree reasons
      $f_{2;1}$ and $f_3$ must vanish.  \Cref{eq:ELie3:f211,eq:ELie3:f31,eq:ELie3:f32,eq:Leib3:f4} hold again for
      degree reasons, and the left-hand sides of \cref{eq:Leib3:f3,eq:ELie3:f21} are zero.  In this way we recover
      Roytenberg's notion of \emph{morphism of 2-term $\ELinf$-algebras} \cite[Definition 2.18]{roytenberg_weak_2007}.

    \subsection{The homotopy transfer theorem} \label{S:ELie3:HTT}
      While we explicitly prove a homotopy transfer result for weak Lie 3-algebras below, let us remark again that
      for a cofibrant resolution of any operad such a result is automatic.

      Let $(L,\diff,\LieDtw)$ be a weak Lie 3-algebra.  Assume that we are given a deformation retract of the
      underlying chain complex, i.e.\@ chain maps $p$ and $i$, and a chain homotopy $h$ as follows:
      \begin{align}
        \label{eq:defretract}
        \begin{tikzcd}[ampersand replacement=\&,/tikz/baseline=-0.5ex]
          (L,\diff)
            \arrow[loop left]{l}{h}
            \arrow[transform canvas={yshift=0.5ex}]{r}{p}
          \& (L',\diff')
            \arrow[transform canvas={yshift=-0.5ex}]{l}{i}
        \end{tikzcd} ,
        &&\textrm{such that }
        \begin{cases}
          \id_L - i\circ p = [\diff,h] ,  \\
          \id_{L'} - p\circ i = 0  .
        \end{cases}
      \end{align}
      In this setting, the following homotopy transfer theorem for weak Lie 3-algebras holds.

      \begin{pp}\label{pp:ELie3:HTT}
        Let $(L,\diff,\LieDtw)$ be a weak Lie 3-algebra and let $(L',\diff')$ be a deformation retract of $(L,\diff)$
        as in \cref{eq:defretract}.  Then $(L',\diff')$ can be equipped with a transferred weak Lie 3-algebra
        structure in such a way, that the map $i$ admits an extension to a weak morphism of weak Lie 3-algebras.
      \end{pp}

      \begin{proof}
        This follows directly from \Cref{lm:htt:str} and \Cref{lm:htt:mor} below.
      \end{proof}

      By the homotopy transfer theorem for $\Leibinf$-algebras we know that one can define a Leibniz 3-algebra
      structure on $(L',\diff')$ by the structure maps
      \begin{align}
        \LieDstr{2}[\prime] &\defeq p \circ \LieDstr{2} \circ (i,i) ,
          \label{eq:htt:l2}  \\
        \LieDstr{3}[\prime] &\defeq p \circ \LieDstr{3} \circ (i,i,i)
          + p \circ \big(
            \LieDstr{2}\circ_1(h\circ\LieDstr{2})
            - \big(\LieDstr{2}\circ_2(h\circ\LieDstr{2})\big)^{1-(12)}
          \big) \circ (i,i,i) ,
          \label{eq:htt:l3}  \\
        \begin{split}
          \LieDstr{4}[\prime] &\defeq p \circ \LieDstr{4} \circ (i,i,i,i)
            - p \circ \left(\begin{aligned}
                &\LieDstr{2}\circ_1(h\circ\LieDstr{3})
                + \big(\LieDstr{2}\circ_2(h\circ\LieDstr{3})\big)^{1-(12)+(123)}  \\
                &+ \LieDstr{3}\circ_1(h\circ\LieDstr{2})
                - \big(\LieDstr{3}\circ_2(h\circ\LieDstr{2})\big)^{1-(12)}  \\
                &+ \big(\LieDstr{3}\circ_3(h\circ\LieDstr{2})\big)^{1-(23)+(132)}
              \end{aligned}\right) \circ (i,i,i,i)  \\
            &\quad- p \circ \left(\begin{aligned}
                &\LieDstr{2}\circ_1(h\circ\LieDstr{2})\circ_1(h\circ\LieDstr{2})
                - \big(\LieDstr{2}\circ_1(h\circ\LieDstr{2})\circ_2(h\circ\LieDstr{2})\big)^{1-(12)}  \\
                &+ \LieDstr{2}\circ\big(h\circ\LieDstr{2},h\circ\LieDstr{2}\big)^{1-(23)+(132)}
                + \big(\LieDstr{2}\circ_2(h\circ\LieDstr{2})\circ_2(h\circ\LieDstr{2})\big)^{1-(12)+(123)}  \\
                &- \big(
                    \LieDstr{2}\circ_2(h\circ\LieDstr{2})\circ_3(h\circ\LieDstr{2})
                  \big)^{1-(12)+(123)-(23)+(132)-(13)}
              \end{aligned}\right) \circ (i,i,i,i) ,
        \end{split} \label{eq:htt:l4}
      \end{align}
      and that the following components define an extension of $i$ to a homotopy morphism of Leibniz 3-algebras
      $i\colon (L',\diff',\LieDtw') \to (L,\diff,\LieDtw)$:
      \begin{align}
        i_1     &= i ,
          \label{eq:htt:i1}  \\
        i_2     &= - h \circ \LieDstr{2} \circ (i,i) ,
          \label{eq:htt:i2}  \\
        i_3     &= - h \circ \LieDstr{3} \circ (i,i,i)
          - h \circ \left(
              \LieDstr{2}\circ_1(h\circ\LieDstr{2})
              - (\LieDstr{2}\circ_2(h\circ\LieDstr{2}))^{1-(12)}
            \right) \circ (i,i,i)  .
          \label{eq:htt:i3}
      \end{align}
      We extend these constructions to weak Lie 3-algebras following the same pattern.

      \begin{lm} \label{lm:htt:str}
        The Leibniz 3-algebra $(L',\diff')$ with structure maps defined by \cref{eq:htt:l2,eq:htt:l3,eq:htt:l4} admits
        an extension to a weak Lie 3-algebra.  Explicitly, such an extension is given by the structure maps
        \begin{align}
          \LieDstr{2;1}[\prime] &\defeq p \circ \LieDstr{2;1} \circ (i,i) ,
            \label{eq:htt:l21}  \\
          \LieDstr{2;1,1}[\prime] &\defeq p \circ \LieDstr{2;1,1} \circ (i,i) ,
            \label{eq:htt:l211}  \\
          \LieDstr{3,1}[\prime] &\defeq p \circ \LieDstr{3;1} \circ (i,i,i)
            - p \circ \big(\LieDstr{2}\circ_1(h\circ\LieDstr{2;1})\big) \circ (i,i,i) ,
            \label{eq:htt:l31}  \\
          \LieDstr{3,2}[\prime] &\defeq p \circ \LieDstr{3;2} \circ (i,i,i)
            + p \circ \big(
                \LieDstr{2;1}\circ_1(h\circ\LieDstr{2})
                + \LieDstr{2}\circ_2(h\circ\LieDstr{2;1})
                + (\LieDstr{2;1}\circ_2(h\circ\LieDstr{2}))^{(12)}
              \big) \circ (i,i,i) .
            \label{eq:htt:l32}
        \end{align}
      \end{lm}

      The proof is a straight-forward verification of
      \cref{eq:ELie3:l21,eq:ELie3:l211,eq:ELie3:l31,eq:ELie3:l32,eq:ELie3:l2111,eq:ELie3:l311,eq:ELie3:l312,eq:ELie3:l322,eq:ELie3:l41,eq:ELie3:l42,eq:ELie3:l43},
      which we postpone to \Cref{A:comp:htt}.

      \begin{lm} \label{lm:htt:mor}
        The Leibniz 3-algebra morphism $i\colon (L',\diff') \to (L,\diff)$ with components defined by
        \cref{eq:htt:i1,eq:htt:i2,eq:htt:i3} extends to a morphism of weak Lie 3-algebras with the additional
        component
        \begin{align}
          i_{2;1} &= - h \circ \LieDstr{2;1} \circ (i,i) .
            \label{eq:htt:i21}
        \end{align}
      \end{lm}

      \begin{proof}
        We verify that with this definition of the additional component $i_{2;1}$
        \cref{eq:ELie3:f21,eq:ELie3:f211,eq:ELie3:f31,eq:ELie3:f32} hold:
        \begin{align*}
          \dEnd(i_{2;1}) &= - \dEnd(h \circ \LieDstr{2;1} \circ (i,i))  \\
            &= - (\id - ip) \circ \LieDstr{2;1} \circ (i,i)
              + h \circ \big(\LieDstr{2} + \LieDstr{2}[(12)]\big) \circ (i,i)  \\
            &= - i_2 - i_2^{(12)} + i_1 \circ \LieDstr{2;1}[\prime] - \LieDstr{2;1} \circ (i_1,i_1) ,  \\
          \begin{split}
          \hspace{2em}&\hspace{-2em}
            i_1 \circ \LieDstr{2;1,1}[\prime] - \LieDstr{2;1,1} \circ (i_1,i_1)  \\
            &= i \circ (p \circ \LieDstr{2;1,1} \circ (i,i)) - \LieDstr{2;1,1} \circ (i,i)
          \end{split}  \\
            &= - \dEnd h \circ \LieDstr{2;1,1} \circ (i,i)  \\
            &= - h \circ \big(\LieDstr{2;1} - \LieDstr{2;1}[(12)]\big) \circ (i,i)  \\
            &= i_{2;1} - i_{2;1}^{(12)} ,  \\
          \begin{split}
          \hspace{2em}&\hspace{-2em}
            i_1 \circ \LieDstr{3;1}[\prime] - \LieDstr{3;1} \circ (i_1,i_1,i_1)  \\
            &= i \circ \big( p \circ (\LieDstr{3;1} - \LieDstr{2} \circ_1 (h\circ\LieDstr{2;1})) \circ (i,i,i) \big)
              - \LieDstr{3;1} \circ (i,i,i)
          \end{split}  \\
            &= - \dEnd h \circ \LieDstr{3;1} \circ (i,i,i)
              - (\id - \dEnd h) \circ (\LieDstr{2}\circ_1(h\circ\LieDstr{2;1})) \circ (i,i,i)  \\
            \begin{split}
              &= - h \circ \big(\LieDstr{3} + \LieDstr{3}[(12)] + \LieDstr{2}\circ_1\LieDstr{2;1}\big) \circ (i,i,i)
                + \LieDstr{2} \circ (i_{2;1}, i_1)  \\
                &\quad+ h \circ (\LieDstr{2} \circ_1 (\dEnd h \circ \LieDstr{2;1})) \circ (i,i,i)
                - h \circ \big(\LieDstr{2} \circ_1 \big(
                    h \circ \big(\LieDstr{2} + \LieDstr{2}[(12)]\big)
                  \big)\big) \circ (i,i,i)
            \end{split}  \\
            \begin{split}
              &= - h \circ (\LieDstr{3} + \LieDstr{2}\circ_1(h\circ\LieDstr{2;1})) \circ (i,i,i)
                - h \circ (\LieDstr{3} + \LieDstr{2}\circ_1(h\circ\LieDstr{2;1}))^{(12)} \circ (i,i,i)  \\
                &\quad+ \LieDstr{2} \circ (i_{2;1}, i_1)
                - h \circ (\LieDstr{2} \circ_1 (ip \circ \LieDstr{2;1})) \circ (i,i,i)
            \end{split}  \\
            &= i_3 + i_3^{(12)} + \LieDstr{2}\circ(i_{2;1},i_1) + i_2\circ_1\LieDstr{2;1}[\prime] ,  \\
          \begin{split}
          \hspace{2em}&\hspace{-2em}
            i_1 \circ \LieDstr{3;2}[\prime] - \LieDstr{3;2} \circ (i_1,i_1,i_1)  \\
            &= i \circ \big( p \circ \big(
                \LieDstr{3;2}
                + \LieDstr{2;1} \circ_1 (h\circ\LieDstr{2})
                + \LieDstr{2} \circ_2 (h\circ\LieDstr{2;1})
                + (\LieDstr{2;1} \circ_2 (h\circ\LieDstr{2}))^{(12)}
              \big) \circ (i,i,i) \big)
              - \LieDstr{3;2} \circ (i,i,i)
          \end{split}  \\
          &= i \circ \big( p \circ \LieDstr{3;2} \circ (i,i,i)
              + p \circ \big(
                  \LieDstr{2;1}\circ_1(h\circ\LieDstr{2})
                  + \LieDstr{2}\circ_2(h\circ\LieDstr{2;1})
                  + (\LieDstr{2;1}\circ_2(h\circ\LieDstr{2}))^{(12)}
                \big)
            \big) - \LieDstr{3;2} \circ (i,i,i)  \\
          &= - \dEnd h \circ \LieDstr{3;2} \circ (i,i,i)
            + (\id - \dEnd h) \circ \big(
                  \LieDstr{2;1}\circ_1(h\circ\LieDstr{2})
                  + \LieDstr{2}\circ_2(h\circ\LieDstr{2;1})
                  + (\LieDstr{2;1}\circ_2(h\circ\LieDstr{2}))^{(12)}
                \big) \circ (i,i,i)  \\
          \begin{split}
            &= -h \circ \big(
                \LieDstr{3} + \LieDstr{3}[(23)]
                + \LieDstr{2;1}\circ_1\LieDstr{2}
                - \LieDstr{2}\circ_2\LieDstr{2;1}
                + (\LieDstr{2;1}\circ_2\LieDstr{2})^{(12)}
              \big) \circ (i,i,i)  \\
              &\quad+ \big(
                  \LieDstr{2;1}\circ_1(h\circ\LieDstr{2})
                  + \LieDstr{2}\circ_2(h\circ\LieDstr{2;1})
                  + (\LieDstr{2;1}\circ_2(h\circ\LieDstr{2}))^{(12)}
                \big) \circ (i,i,i)  \\
              &\quad- h \circ \big(
                  \big( \LieDstr{2} + \LieDstr{2}[(12)] \big) \circ_1 (h \circ \LieDstr{2})
                \big) \circ (i,i,i)
              + h \circ (\LieDstr{2;1}\circ_1((\id-ip)\circ\LieDstr{2}) \circ (i,i,i)  \\
              &\quad- h \circ (\LieDstr{2}\circ_2((\id-ip)\circ\LieDstr{2;1}) \circ (i,i,i)
              + h \circ \big(
                  \LieDstr{2} \circ_2 \big(h \circ \big( \LieDstr{2} + \LieDstr{2}[(12)] \big)\big)
                \big) \circ (i,i,i)  \\
              &\quad- h \circ \big(
                  \big( \LieDstr{2} + \LieDstr{2}[(12)] \big) \circ_2 (h \circ \LieDstr{2})
                \big)^{(12)} \circ (i,i,i)
              + h \circ (\LieDstr{2;1}\circ_2((\id-ip)\circ\LieDstr{2})^{(12)} \circ (i,i,i)
          \end{split}  \\
          \begin{split}
            &= - h \circ \big(
                \LieDstr{3} + \LieDstr{2}\circ_1(h\circ\LieDstr{2}) - (\LieDstr{2}\circ_2(h\circ\LieDstr{2}))^{1-(12)}
              \big) \circ (i,i,i)  \\
            &\quad- h \circ \big(
                \LieDstr{3} + \LieDstr{2}\circ_1(h\circ\LieDstr{2}) - (\LieDstr{2}\circ_2(h\circ\LieDstr{2}))^{1-(12)}
              \big)^{(23)} \circ (i,i,i)  \\
            &\quad-\LieDstr{2;1}\circ(i_2,i_1) - \LieDstr{2}\circ(i_1,i_{2;1}) - \LieDstr{2;1}\circ(i_1,i_2)^{(12)} \\
            &\quad+ (-h\circ\LieDstr{2;1}\circ(i,i)) \circ_1 (p\circ\LieDstr{2}\circ(i,i))
            - (-h\circ\LieDstr{2}\circ(i,i)) \circ_2 (p\circ\LieDstr{2;1}\circ(i,i))  \\
            &\quad+ \big( (-h\circ\LieDstr{2;1}\circ(i,i)) \circ_2 (p\circ\LieDstr{2}\circ(i,i)) \big)^{(12)}
          \end{split}  \\
          \begin{split}
            &= i_3 + i_3^{(23)}
              - \LieDstr{2;1}\circ(i_2,i_1)
              - \LieDstr{2}\circ(i_1,i_{2;1})
              - \LieDstr{2;1}\circ(i_1,i_2)^{(12)}  \\
              &\quad+ i_{2;1}\circ_1\LieDstr{2}[\prime]
              - i_{2}\circ_2\LieDstr{2;1}[\prime]
              + (i_{2;1}\circ_2\LieDstr{2}[\prime])^{(12)} .
          \end{split}
        \end{align*}
        This concludes the proof.
      \end{proof}

  \section{Skew-symmetrization} \label{S:SS}
    Since an $\Linf$-algebra is just a $\Leibinf$-algebra with skew-symmetric structure maps, it seems natural to try
    and construct an $\Linf$-algebra by skew-symmetrizing the structure maps of a $\Leibinf$-algebra.  One way to
    put this more formally is in terms of the Koszul dual cooperads: Assume for a moment that we can construct a right
    inverse to the morphism $\psi$ defined by \cref{eq:LeibKtoLieK}, i.e.\@ a morphism $\phi$ of dg cooperads
    \begin{align*}
      \begin{tikzcd}[ampersand replacement=\&]
        \psi : \LeibK
          \arrow[transform canvas={yshift=0.5ex}]{r}{}
        \& \LieK : \phi
          \arrow[transform canvas={yshift=-0.5ex}]{l}{}
      \end{tikzcd} ,
      &&\textrm{such that $\psi\circ\phi = \id$.}
    \end{align*}
    In this case, we obtain for any $\Leibinf$-algebra $(L,\diff,\LeibKtw)$ given by a twisting morphism
    $\LeibKtw\colon \LeibK \to \End_L$, an $\Linf$-algebra $(L,\diff,\LieKtw)$ via precomposition of the twisting
    morphism with $\phi$, i.e.\@ $\LieKtw \defeq \LeibKtw\circ\phi$.  We shall make a naive attempt at defining
    such a morphism $\phi$ below to see how it fails.

    Define $\phi$ on $\Fk[\Sy_n]$-generators by
    \begin{equation*}
      \phi(\LieKgen{n}) \defeq \frac{1}{n!} \sum_{\sigma\in\Sy_n} (-1)^{|\sigma|}\cdot \LieDgen{n}[\sigma] .
    \end{equation*}
    Clearly, this is a well-defined morphism of dg $\Sy$-modules and satisfies $\psi\circ\phi = \id$, provided the
    $1/n!$ exist.  The calculation
    below, however, shows that the map $\phi$ does not define a morphism of dg cooperads since $\phi$ does not commute
    with the decomposition map $\decomp$ already in arity~3.  We find
    \begin{align*}
      \decomp\big(\phi(\LieKgen{3})\big)
      &= \frac{1}{6} \sum_{\sigma\in\Sy_3} (-1)^{|\sigma|} \cdot \big(
          - \LieDgen{2}\circ(\LieDgen{2},1)
          + \LieDgen{2}\circ(1,\LieDgen{2})
          - \LieDgen{2}\circ(1,\LieDgen{2})^{(12)}
        \big)^\sigma  \\
      &= - \frac{1}{6} \sum_{\sigma\in\Sy_3} (-1)^{|\sigma|} \cdot \LieDgen{2}\circ(\LieDgen{2},1)^\sigma
        + \frac{1}{3} \sum_{\sigma\in\Sy_3} (-1)^{|\sigma|} \cdot \LieDgen{2}\circ(1,\LieDgen{2})^\sigma ,
    \intertext{while}
      \phi\big(\decomp(\LieKgen{3})\big)
      &= - \frac{1}{4} \sum_{\sigma\in\Sy_3} (-1)^{|\sigma|} \cdot \big(
          \LieDgen{2}\circ(\LieDgen{2},1)
          - \LieDgen{2}\circ(1,\LieDgen{2})
        \big)^\sigma .
    \end{align*}
    Note that the difference
    \begin{align*}
      \big(\decomp\circ\phi - \phi\circ\decomp\big)(\LieKgen{3})
      &= \frac{1}{12} \sum_{\sigma\in\Sy_3} (-1)^{|\sigma|} \cdot \big(
          \LieDgen{2}\circ(\LieDgen{2},1)
          + \LieDgen{2}\circ(1,\LieDgen{2})
        \big)^\sigma
    \intertext{is actually a coboundary when we view $\LeibK[3] \subset \LieD[3]$ as a dg subcooperad,}
      &= - \frac{1}{12} \sum_{\sigma\in\Sy_3} (-1)^{|\sigma|} \cdot \big(
          \diff\LieDgen{2;1}\circ(\LieDgen{2},1)
        \big)^\sigma  .
    \end{align*}
    This suggests that
    \begin{enumerate*}
      \item while $\Leibinf$-algebras do not admit a skew-symmetrization in general, for weak Lie 3-algebras (and more
        generally $\ELinf$-algebras) such a construction may exist, and
      \item we should try to extend $\phi$ to a homotopy morphism of dg cooperads.
    \end{enumerate*}

    The remainder of this section is organized as follows.  In \Cref{S:SS:Phi}, we construct a right inverse $\Phi$
    for $\COBAR\psi\colon \COBAR\LieD[3] \to \COBAR\LieK[3]$.  In \Cref{S:SS:ELie3}, we use $\Phi$ to define a
    skew-symmetrization construction for weak Lie 3-algebras.  In \Cref{S:SS:ELie3:mor}, we define an ad hoc
    skew-symmetrization for morphisms of weak Lie 3-algebras and show that it is functorial up to homotopy.

    In this entire section, we assume that $2,3 \in \Fk^\times$ are units.

    \subsection{A right inverse homotopy morphism for the cooperad resolution} \label{S:SS:Phi}
      Below we construct a right inverse for $\COBAR\psi\colon \COBAR\LieD[3] \to \COBAR\LieK[3]$.  We think of such a
      map as a homotopy morphism of dg cooperads.

      \begin{lm}\label{lm:SS:Phi}
        The morphism $\COBAR\psi$ admits a right inverse, i.e.\@ a morphism $\Phi$ of dg operads
        \begin{align*}
          \begin{tikzcd}[ampersand replacement=\&]
            \COBAR\psi : \COBAR\LieD[3]
              \arrow[transform canvas={yshift=0.5ex}]{r}{}
            \& \COBAR\LieK[3] : \Phi
              \arrow[transform canvas={yshift=-0.5ex}]{l}{}
          \end{tikzcd} ,
          &&\textrm{such that $\COBAR\psi\circ\Phi = \id$.}
        \end{align*}
        One such morphism $\Phi$ is defined by
        \begin{align*}
          \Phi(s^{-1}\LieKgen{2})
            &= \frac{1}{2} \sum_{\sigma\in\Sy_2} (-1)^{|\sigma|} \cdot s^{-1}\LieDgen{2}[\sigma]  , \\
          \Phi(s^{-1}\LieKgen{3})
            &= \frac{1}{6} \sum_{\sigma\in\Sy_3} (-1)^{|\sigma|} \cdot s^{-1}\LieDgen{3}[\sigma]
              - \frac{1}{24} \sum_{\sigma\in\Sy_3} (-1)^{|\sigma|} \cdot \big(
                  s^{-1}\LieDgen{2;1} \circ_1 s^{-1}\LieDgen{2}
                  + s^{-1}\LieDgen{2;1} \circ_2 s^{-1}\LieDgen{2}
                \big)^\sigma  , \\
          \Phi(s^{-1}\LieKgen{4})
            &= \frac{1}{24} \sum_{\sigma\in\Sy_4} (-1)^{|\sigma|} \cdot s^{-1}\LieDgen{4}[\sigma]
              + \frac{1}{48} \sum_{\sigma\in\Sy_4} (-1)^{|\sigma|} \cdot \left(\begin{aligned}
                  s^{-1}\LieDgen{2;1} \circ_1 s^{-1}\LieDgen{3}
                  - s^{-1}\LieDgen{3;1} \circ_1 s^{-1}\LieDgen{2}
                  + s^{-1}\LieDgen{3;2} \circ_2 s^{-1}\LieDgen{2}  \\
                  - s^{-1}\LieDgen{2;1} \circ_2 s^{-1}\LieDgen{3}
                  - s^{-1}\LieDgen{3;1} \circ_2 s^{-1}\LieDgen{2}
                  + s^{-1}\LieDgen{3;2} \circ_3 s^{-1}\LieDgen{2}
                \end{aligned}\right)^\sigma  .
        \end{align*}
      \end{lm}

      \begin{proof}
        Any morphism of free dg operads is completely determined by its value on generators, i.e.\@ it is sufficient to
        define $\Phi|_{s^{-1}\LieK[3]}$ and its extension $\Phi$ is then automatically a morphism of operads.  It
        remains to verify that $\Phi$ commutes with the differential, which we do below:
        \begin{align*}
          \diff\Phi(s^{-1}\LieKgen{2})
            &= \frac{1}{2} \sum_{\sigma\in\Sy_2} (-1)^{|\sigma|} \cdot \diff s^{-1}\LieDgen{2}[\sigma]
            = 0
            = \Phi(0)
            = \Phi(\diff s^{-1}\LieKgen{2}) ,  \\
          \diff\Phi(s^{-1}\LieKgen{3})
            &= - \frac{1}{6} \sum_{\sigma\in\Sy_3} (-1)^{|\sigma|} \cdot \diff_2 s^{-1}\LieDgen{3}[\sigma]
              - \frac{1}{24} \sum_{\sigma\in\Sy_3} (-1)^{|\sigma|} \cdot \big(
                  \diff_1 s^{-1}\LieDgen{2;1} \circ_1 s^{-1}\LieDgen{2}
                  + \diff_1 s^{-1}\LieDgen{2;1} \circ_2 s^{-1}\LieDgen{2}
                \big)^\sigma  \\
            &= - \frac{1}{6} \sum_{\sigma\in\Sy_3} (-1)^{|\sigma|} \cdot \big(
                - s^{-1}\LieDgen{2} \circ_1 s^{-1}\LieDgen{2}
                + s^{-1}\LieDgen{2} \circ_2 s^{-1}\LieDgen{2}
                - (s^{-1}\LieDgen{2} \circ_2 s^{-1}\LieDgen{2})^{(12)}
              \big)^\sigma  \\
              &\quad- \frac{1}{12} \sum_{\sigma\in\Sy_3} (-1)^{|\sigma|} \cdot
                \big( (s^{-1}\LieDgen{2} + s^{-1}\LieDgen{2}[(12)]) \circ_1 s^{-1}\LieDgen{2} \big)^\sigma  \\
            &= - \frac{1}{4} \sum_{\sigma\in\Sy_3} (-1)^{|\sigma|} \cdot \big(
                s^{-1}\LieDgen{2} \circ_1 s^{-1}\LieDgen{2}
                - s^{-1}\LieDgen{2} \circ_2 s^{-1}\LieDgen{2}
              \big)^\sigma  \\
            &= - \Phi(s^{-1}\LieKgen{2}) \circ_1 \Phi(s^{-1}\LieKgen{2})
              + \Phi(s^{-1}\LieKgen{2}) \circ_2 \Phi(s^{-1}\LieKgen{2})
              - (\Phi(s^{-1}\LieKgen{2}) \circ_2 \Phi(s^{-1}\LieKgen{2}))^{(12)}  \\
            &= \Phi\big(
                - s^{-1}\LieKgen{2} \circ_1 s^{-1}\LieKgen{2}
                + s^{-1}\LieKgen{2} \circ_2 s^{-1}\LieKgen{2}
                - (s^{-1}\LieKgen{2} \circ_2 s^{-1}\LieKgen{2})^{(12)}
              \big)
            = \Phi(- \diff_2 s^{-1}\LieKgen{3})
            = \Phi(\diff s^{-1}\LieKgen{3}) ,  \\
          \diff\Phi(s^{-1}\LieKgen{4})
            &= - \frac{1}{24} \sum_{\sigma\in\Sy_4} (-1)^{|\sigma|} \cdot \diff_2 s^{-1}\LieDgen{4}[\sigma]
              + \frac{1}{48} \sum_{\sigma\in\Sy_4} (-1)^{|\sigma|} \cdot \left(\begin{aligned}
                  (\diff_1 s^{-1}\LieDgen{2;1}) \circ_1 s^{-1}\LieDgen{3}
                  + s^{-1}\LieDgen{2;1} \circ_1 (\diff_2 s^{-1}\LieDgen{3})  \\
                  - (\diff_1 s^{-1}\LieDgen{2;1}) \circ_2 s^{-1}\LieDgen{3}
                  - s^{-1}\LieDgen{2;1} \circ_2 (\diff_2 s^{-1}\LieDgen{3})
                \end{aligned}\right)^\sigma  \\
              &\quad- \frac{1}{48} \sum_{\sigma\in\Sy_4} (-1)^{|\sigma|} \cdot \left(\begin{aligned}
                  (\diff_1 s^{-1}\LieDgen{3;1}) \circ_1 s^{-1}\LieDgen{2}
                  - (\diff_2 s^{-1}\LieDgen{3;1}) \circ_1 s^{-1}\LieDgen{2}  \\
                  + (\diff_1 s^{-1}\LieDgen{3;1}) \circ_2 s^{-1}\LieDgen{2}
                  - (\diff_2 s^{-1}\LieDgen{3;1}) \circ_2 s^{-1}\LieDgen{2}
                \end{aligned}\right)^\sigma  \\
              &\quad+ \frac{1}{48} \sum_{\sigma\in\Sy_4} (-1)^{|\sigma|} \cdot \left(\begin{aligned}
                  (\diff_1 s^{-1}\LieDgen{3;2}) \circ_2 s^{-1}\LieDgen{2}
                  - (\diff_2 s^{-1}\LieDgen{3;2}) \circ_2 s^{-1}\LieDgen{2}  \\
                  + (\diff_1 s^{-1}\LieDgen{3;2}) \circ_2 s^{-1}\LieDgen{2}
                  - (\diff_2 s^{-1}\LieDgen{3;2}) \circ_2 s^{-1}\LieDgen{2}
                \end{aligned}\right)^\sigma  \\
            &= - \frac{1}{24} \sum_{\sigma\in\Sy_4} (-1)^{|\sigma|} \cdot \left(\begin{aligned}
                &- s^{-1}\LieDgen{2} \circ_1 s^{-1}\LieDgen{3}
                - 3 s^{-1}\LieDgen{2} \circ_2 s^{-1}\LieDgen{3}  \\
                &+ s^{-1}\LieDgen{3} \circ_1 s^{-1}\LieDgen{2}
                - 2 s^{-1}\LieDgen{3} \circ_2 s^{-1}\LieDgen{2}
                + 3 s^{-1}\LieDgen{3} \circ_3 s^{-1}\LieDgen{2}
              \end{aligned}\right)^\sigma  \\
              &\quad+ \frac{1}{48} \sum_{\sigma\in\Sy_4} (-1)^{|\sigma|} \cdot \left(\begin{aligned}
                  &s^{-1}\LieDgen{2} \circ_1 s^{-1}\LieDgen{3}
                  + (s^{-1}\LieDgen{2} \circ_2 s^{-1}\LieDgen{3})^{(1234)}
                  - s^{-1}\LieDgen{2} \circ_2 s^{-1}\LieDgen{3}
                  - (s^{-1}\LieDgen{2} \circ_1 s^{-1}\LieDgen{3})^{(1432)}  \\
                  &+ s^{-1}\LieDgen{2;1} \circ_1 \big(
                      s^{-1}\LieKgen{2} \circ_1 s^{-1}\LieKgen{2}
                      - s^{-1}\LieKgen{2} \circ_2 s^{-1}\LieKgen{2}
                      + (s^{-1}\LieKgen{2} \circ_2 s^{-1}\LieKgen{2})^{(12)}
                    \big)  \\
                  &- s^{-1}\LieDgen{2;1} \circ_2 \big(
                      s^{-1}\LieKgen{2} \circ_1 s^{-1}\LieKgen{2}
                      - s^{-1}\LieKgen{2} \circ_2 s^{-1}\LieKgen{2}
                      + (s^{-1}\LieKgen{2} \circ_2 s^{-1}\LieKgen{2})^{(12)}
                    \big)
                \end{aligned}\right)^\sigma  \\
              &\quad- \frac{1}{48} \sum_{\sigma\in\Sy_4} (-1)^{|\sigma|} \cdot \left(\begin{aligned}
                  s^{-1}\LieDgen{3} \circ_1 s^{-1}\LieDgen{2}
                  + (s^{-1}\LieDgen{3} \circ_2 s^{-1}\LieDgen{2})^{(123)}
                  + (s^{-1}\LieDgen{2} \circ_1 s^{-1}\LieDgen{2;1}) \circ_1 s^{-1}\LieDgen{2}  \\
                  s^{-1}\LieDgen{3} \circ_2 s^{-1}\LieDgen{2}
                  + (s^{-1}\LieDgen{3} \circ_1 s^{-1}\LieDgen{2})^{(132)}
                  + (s^{-1}\LieDgen{2} \circ_1 s^{-1}\LieDgen{2;1}) \circ_2 s^{-1}\LieDgen{2}
                \end{aligned}\right)^\sigma  \\
              &\quad+ \frac{1}{48} \sum_{\sigma\in\Sy_4} (-1)^{|\sigma|} \cdot \left(\begin{aligned}
                  &s^{-1}\LieDgen{3} \circ_2 s^{-1}\LieDgen{2}
                  + (s^{-1}\LieDgen{3} \circ_3 s^{-1}\LieDgen{2})^{(234)}  \\
                  &+ (s^{-1}\LieDgen{2;1} \circ_1 s^{-1}\LieDgen{2}
                      - s^{-1}\LieDgen{2} \circ_2 s^{-1}\LieDgen{2;1}
                      + (s^{-1}\LieDgen{2;1} \circ_2 s^{-1}\LieDgen{2})^{(12)}
                    ) \circ_2 s^{-1}\LieDgen{2}  \\
                  &+ s^{-1}\LieDgen{3} \circ_3 s^{-1}\LieDgen{2}
                  + (s^{-1}\LieDgen{3} \circ_2 s^{-1}\LieDgen{2})^{(243)}  \\
                  &+ (s^{-1}\LieDgen{2;1} \circ_1 s^{-1}\LieDgen{2}
                      - s^{-1}\LieDgen{2} \circ_2 s^{-1}\LieDgen{2;1}
                      + (s^{-1}\LieDgen{2;1} \circ_2 s^{-1}\LieDgen{2})^{(12)}
                    ) \circ_3 s^{-1}\LieDgen{2}
                \end{aligned}\right)^\sigma  \\
            &= \frac{1}{12} \sum_{\sigma\in\Sy_4} (-1)^{|\sigma|} \cdot \left(
                s^{-1}\LieDgen{2} \circ_1 s^{-1}\LieDgen{3}
                + s^{-1}\LieDgen{2} \circ_2 s^{-1}\LieDgen{3}
                - s^{-1}\LieDgen{3} \circ_1 s^{-1}\LieDgen{2}
                + s^{-1}\LieDgen{3} \circ_2 s^{-1}\LieDgen{2}
                - s^{-1}\LieDgen{3} \circ_3 s^{-1}\LieDgen{2}
              \right)^\sigma  \\
              &\quad- \frac{1}{48} \sum_{\sigma\in\Sy_4} (-1)^{|\sigma|} \cdot \left(\begin{aligned}
                  & s^{-1}\LieDgen{2} \circ_1 s^{-1}\LieDgen{2;1} \circ_1 s^{-1}\LieDgen{2}
                  + s^{-1}\LieDgen{2} \circ_1 s^{-1}\LieDgen{2;1} \circ_2 s^{-1}\LieDgen{2}  \\
                  &+ s^{-1}\LieDgen{2} \circ_2 s^{-1}\LieDgen{2;1} \circ_2 s^{-1}\LieDgen{2}
                  + s^{-1}\LieDgen{2} \circ_2 s^{-1}\LieDgen{2;1} \circ_3 s^{-1}\LieDgen{2}  \\
                  &- s^{-1}\LieDgen{2;1} \circ_1 s^{-1}\LieDgen{2} \circ_1 s^{-1}\LieDgen{2}
                  - s^{-1}\LieDgen{2;1} \circ_2 s^{-1}\LieDgen{2} \circ_1 s^{-1}\LieDgen{2}  \\
                  &+ s^{-1}\LieDgen{2;1} \circ_1 s^{-1}\LieDgen{2} \circ_2 s^{-1}\LieDgen{2}
                  + s^{-1}\LieDgen{2;1} \circ_2 s^{-1}\LieDgen{2} \circ_2 s^{-1}\LieDgen{2}  \\
                  &- s^{-1}\LieDgen{2;1} \circ_1 s^{-1}\LieDgen{2} \circ_3 s^{-1}\LieDgen{2}
                  - s^{-1}\LieDgen{2;1} \circ_2 s^{-1}\LieDgen{2} \circ_3 s^{-1}\LieDgen{2}
                \end{aligned}\right)^\sigma  \\
            &= \Phi\left(\begin{aligned}
                &s^{-1}\LieKgen{2} \circ_1 s^{-1}\LieKgen{3}
                + s^{-1}\LieKgen{2} \circ_2 s^{-1}\LieKgen{3}
                - (s^{-1}\LieKgen{2} \circ_2 s^{-1}\LieKgen{3})^{(12)}
                + (s^{-1}\LieKgen{2} \circ_2 s^{-1}\LieKgen{3})^{(123)}  \\
                &- s^{-1}\LieKgen{3} \circ_1 s^{-1}\LieKgen{2}
                + s^{-1}\LieKgen{3} \circ_2 s^{-1}\LieKgen{2}
                - (s^{-1}\LieKgen{3} \circ_2 s^{-1}\LieKgen{2})^{(12)}  \\
                &- s^{-1}\LieKgen{3} \circ_3 s^{-1}\LieKgen{2}
                + (s^{-1}\LieKgen{3} \circ_3 s^{-1}\LieKgen{2})^{(23)}
                - (s^{-1}\LieKgen{3} \circ_3 s^{-1}\LieKgen{2})^{(132)}
              \end{aligned}\right)
            = \Phi(\diff s^{-1}\LieKgen{4}) .
        \end{align*}
        Finally, note that indeed $\COBAR\psi \circ \Phi = \id$.
      \end{proof}

    \subsection{Skew-symmetrization for weak Lie 3-algebras} \label{S:SS:ELie3}
      In \Cref{S:ELie3:Alg} we defined a weak Lie 3-algebra $(L,\diff,\LieDtw)$ as a 3-term complex $(L,\diff)$ with a
      twisting morphism $\LieDtw\colon \LieD[3] \to \End_L$.  By \cref{eq:Tw:Iso}, such a twisting morphism
      corresponds to a morphism of dg operads $g_\LieDtw\colon \COBAR\LieD[3] \to \End_L$ via
      $g_\LieDtw|_{s^{-1}\LieD[3]}(s^{-1}\LieDgen{*}) = \LieDstr{*}$.  By precomposition with $\Phi$ we obtain a
      morphism
      \begin{equation*}
        \Phi^*g_\LieDtw\colon \COBAR\LieK[3] \to \COBAR\LieD[3] \to \End_L
      \end{equation*}
      of dg operads, which in turn corresponds to a 3-term $\Linf$-algebra or (semi-strict) Lie 3-algebra.  We make
      the result of this construction explicit below.

      \begin{df}\label{df:SS}
        Let $(L,\diff,\LieDtw)$ be a weak Lie 3-algebra.  We define its \emph{skew-symmetrization} to be the
        (semi-strict) Lie 3-algebra $(L,\diff,\LieKstr{})$ given by the following structure maps:
        \begin{align}
          \LieKstr{2} &\defeq \frac{1}{2} \sum_{\sigma\in\Sy_2} (-1)^{|\sigma|} \cdot \LieDstr{2}[\sigma] ,
            \label{df:SS:l2}  \\
          \LieKstr{3} &\defeq \frac{1}{6} \sum_{\sigma\in\Sy_3} (-1)^{|\sigma|} \cdot \LieDstr{3}[\sigma]
              - \frac{1}{24} \sum_{\sigma\in\Sy_3} (-1)^{|\sigma|} \cdot
                \big( \LieDstr{2;1} \circ_1 \LieDstr{2} + \LieDstr{2;1} \circ_2 \LieDstr{2} \big)^\sigma ,
            \label{df:SS:l3}  \\
          \LieKstr{4} &\defeq \frac{1}{24} \sum_{\sigma\in\Sy_4} (-1)^{|\sigma|} \cdot \LieDstr{4}[\sigma]
              + \frac{1}{48} \sum_{\sigma\in\Sy_4} (-1)^{|\sigma|} \cdot \left(\begin{aligned}
                  \LieDstr{2;1} \circ_1 \LieDstr{3}
                  - \LieDstr{3;1} \circ_1 \LieDstr{2}
                  + \LieDstr{3;2} \circ_2 \LieDstr{2}  \\
                  - \LieDstr{2;1} \circ_2 \LieDstr{3}
                  - \LieDstr{3;1} \circ_2 \LieDstr{2}
                  + \LieDstr{3;2} \circ_3 \LieDstr{2}
                \end{aligned}\right)^\sigma .
            \label{df:SS:l4}
        \end{align}
      \end{df}

      Note that for weak Lie 2-algebras, $\LieDstr{2;1}$ is symmetric and we recover Roytenberg's skew-symmetrization
      construction in this case.

    \subsection{Skew-symmetrization for morphisms of weak Lie 3-algebras} \label{S:SS:ELie3:mor}
      A morphism $f\colon (L,\diff,\LieDtw) \to (L',\diff',\LieDtw')$ of weak Lie 3-algebras was defined in
      \Cref{S:ELie3:Mor} as a morphism $f\colon\LieD[3] \to \End^L_{L'}$ satisfying a certain Maurer--Cartan type
      \cref{eq:ELie3:Mor:Tw}.  Such a morphism in general does not correspond to a morphism
      $\COBAR\LieD[3] \to \End^L_{L'}$ and there is no obvious way to precompose $f$ with $\Phi$.  Below we give an
      ad hoc construction for a skew-symmetrization of morphisms instead.

      \begin{lm} \label{df:SS:mor}
        Let $f\colon (L,\diff,\LieDtw) \to (L',\diff',\LieDtw')$ be a morphism of weak Lie 3-algebras.
        The following components define a morphism of (semi-strict) Lie 3-algebras
        $\SSy{f}\colon (L,\diff,\LieKtw) \to (L',\diff',\LieKtw')$:
        \begin{align}
          \SSy{f}_1 &\defeq f_1 ,  \\
          \SSy{f}_2 &\defeq \frac{1}{2} \sum_{\sigma\in\Sy_2} (-1)^{|\sigma|} \cdot f_2^\sigma ,  \\
          \SSy{f}_3 &\defeq \frac{1}{6} \sum_{\sigma\in\Sy_3} (-1)^{|\sigma|} \cdot f_3^\sigma
              - \frac{1}{24} \sum_{\sigma\in\Sy_3} (-1)^{|\sigma|} \cdot
                \left(\begin{aligned}
                    f_{2;1} \circ_1 \LieDstr{2} + f_{2;1} \circ_2 \LieDstr{2}
                    - \LieDstr{2;1}[\prime] \circ (f_2,f_1) - \LieDstr{2;1}[\prime] \circ (f_1,f_2)
                  \end{aligned}\right)^\sigma .
        \end{align}
        We call $\SSy{f}$ the \emph{skew-symmetrization} of $f$.
      \end{lm}

      The proof is a direct verification of \cref{eq:Leib3:f1,eq:Leib3:f2,eq:Leib3:f3,eq:Leib3:f4} and is given in
      \Cref{A:comp:SSmor}.

      \subsubsection{Functoriality} In low degrees, skew-symmetrization commutes with composition of morphisms,
      \begin{align*}
        (\SSy{f'\circ f})_1 &= (f'\circ f)_1 = f'_1\circ f_1 =
          \SSy{f'}_1\circ\SSy{f}_1 = (\SSy{f'}\circ\SSy{f})_1 ,  \\
        (\SSy{f'\circ f})_2 &= \frac{1}{2}\left( (f'\circ f)_2 - (f'\circ f)_2^{(12)} \right)  \\
          &= \frac{1}{2}\left( f'_2\circ(f_1,f_1) + f'_1\circ f_2
            - f'_2\circ(f_1,f_1)^{(12)} - f'_1\circ f_2^{(12)} \right)  \\
          &= \SSy{f'}_2\circ(\SSy{f}_1,\SSy{f}_1) + \SSy{f'}_1\circ \SSy{f}_2  \\
          &= (\SSy{f'}\circ\SSy{f})_2 .
      \end{align*}
      In particular, this implies that skew-symmetrization of weak Lie 2-algebras forms a functor
      \cite[Theorem~3.2]{roytenberg_weak_2007}.  However, for weak Lie 3-algebras this is no longer the case as the
      following computation shows:
      \begin{align*}
        (\SSy{f'\circ f} - \SSy{f'}\circ\SSy{f})_3
          &= \frac{1}{6} \sum_{\sigma\in\Sy_3} (-1)^{|\sigma|} \cdot (f'\circ f)_3^\sigma
              - \frac{1}{24} \sum_{\sigma\in\Sy_3} (-1)^{|\sigma|} \cdot
                \left(\begin{aligned}
                    (f'\circ f)_{2;1} \circ_1 \LieDstr{2}
                    + (f'\circ f)_{2;1} \circ_2 \LieDstr{2}  \\
                    {}- \LieDstr{2;1}[\dprime] \circ ((f'\circ f)_2,(f'\circ f)_1)  \\
                    {}- \LieDstr{2;1}[\dprime] \circ ((f'\circ f)_1,(f'\circ f)_2)
                  \end{aligned}\right)^\sigma  \\
            &\quad- \left(
                \SSy{f'}_3\circ(\SSy{f}_1,\SSy{f}_1,\SSy{f}_1)
                - \SSy{f'}_2\circ(\SSy{f}_2,\SSy{f}_1)
                + \SSy{f'}_2\circ(\SSy{f}_1,\SSy{f}_2)
                - \SSy{f'}_2\circ(\SSy{f}_1,\SSy{f}_2)^{(12)}
                + \SSy{f'}_1\circ \SSy{f}_3
              \right)  \\
          &= \frac{1}{6} \sum_{\sigma\in\Sy_3} (-1)^{|\sigma|} \cdot \left(
              f'_3\circ(f_1,f_1,f_1)
              - f'_2\circ(f_2,f_1) + f'_2\circ(f_1,f_2) - f'_2\circ(f_1,f_2)^{(12)}
              + f'_1\circ f_3
            \right)^\sigma  \\
              &\qquad- \frac{1}{24} \sum_{\sigma\in\Sy_3} (-1)^{|\sigma|} \cdot
                \left(\begin{aligned}
                    (f'_{2;1} \circ (f_1,f_1) + f'_1 \circ f_{2;1}) \circ_1 \LieDstr{2}
                    {}+ (f'_{2;1} \circ (f_1,f_1) + f'_1 \circ f_{2;1}) \circ_2 \LieDstr{2}  \\
                    {}- \LieDstr{2;1}[\dprime] \circ (f'_2 \circ (f_1,f_1) + f'_1 \circ f_2,f'_1 \circ f_1)  \\
                    {}- \LieDstr{2;1}[\dprime] \circ (f'_1 \circ f_1,f'_2 \circ (f_1,f_1) + f'_1 \circ f_2)
                  \end{aligned}\right)^\sigma  \\
            &\quad- \frac{1}{6} \sum_{\sigma\in\Sy_3} (-1)^{|\sigma|} \cdot f'_3\circ(f_1,f_1,f_1)^\sigma
              + \frac{1}{24} \sum_{\sigma\in\Sy_3} (-1)^{|\sigma|} \cdot
                \left(\begin{aligned}
                  &\phantom{{}+{}}f'_{2;1} \circ_1 \LieDstr{2}[\prime] \circ (f_1,f_1,f_1)  \\
                    &+ f'_{2;1} \circ_2 \LieDstr{2}[\prime] \circ (f_1,f_1,f_1)  \\
                    &- \LieDstr{2;1}[\dprime] \circ (f'_2,f'_1) \circ (f_1,f_1,f_1)  \\
                    &- \LieDstr{2;1}[\dprime] \circ (f'_1,f'_2) \circ (f_1,f_1,f_1)
                  \end{aligned}\right)^\sigma  \\
            &\quad- \frac{1}{4} \left(\begin{aligned}
                - \big(f'_2 - f_2^{\prime(12)}\big) \circ \big(f_2 - f_2^{(12)},f_1\big)
                + \big(f'_2 - f_2^{\prime(12)}\big) \circ \big(f_1,f_2 - f_2^{(12)}\big)\phantom{^{(12)}}  \\
                {}- \big(f'_2 - f_2^{\prime(12)}\big) \circ \big(f_1,f_2 - f_2^{(12)}\big)^{(12)}
              \end{aligned}\right)  \\
            &\quad- \frac{1}{6} \sum_{\sigma\in\Sy_3} (-1)^{|\sigma|} \cdot f'_1 \circ f_3^\sigma
              + \frac{1}{24} \sum_{\sigma\in\Sy_3} (-1)^{|\sigma|} \cdot
                \left(\begin{aligned}
                    &f'_1 \circ f_{2;1} \circ_1 \LieDstr{2}
                    + f'_1 \circ f_{2;1} \circ_2 \LieDstr{2}  \\
                    &- f'_1 \circ \LieDstr{2;1}[\dprime] \circ (f_2,f_1)
                    - f'_1 \circ \LieDstr{2;1}[\dprime] \circ (f_1,f_2)
                  \end{aligned}\right)^\sigma  \\
          &= \frac{1}{12} \sum_{\sigma\in\Sy_3} (-1)^{|\sigma|} \cdot \left(
              f'_2\circ(f_2,f_1) + f'_2\circ(f_1,f_2)
            \right)^\sigma  \\
          &\quad- \frac{1}{24} \sum_{\sigma\in\Sy_3} (-1)^{|\sigma|} \cdot \left(\begin{aligned}
              &f'_{2;1}\circ(f_1\circ\LieDstr{2},f_1)
              + f'_{2;1}\circ(f_1,f_1\circ\LieDstr{2})  \\
              &- f'_{2;1}\circ(\LieDstr{2}[\prime]\circ(f_1,f_1),f_1)
              - f'_{2;1}\circ(f_1,\LieDstr{2}[\prime]\circ(f_1,f_1))  \\
              &- \LieDstr{2;1}[\dprime]\circ(f'_1\circ f_2,f'_1\circ f_1)
              - \LieDstr{2;1}[\dprime]\circ(f'_1\circ f_1,f'_1\circ f_2)  \\
              &+ f'_1\circ\LieDstr{2;1}[\prime]\circ(f_2,f_1)
              + f'_1\circ\LieDstr{2;1}[\prime]\circ(f_1,f_2)
            \end{aligned}\right)  \\
          &= \dEnd\left( - \frac{1}{24} \sum_{\sigma\in\Sy_3} (-1)^{|\sigma|} \cdot \left(
              f'_{2;1}\circ(f_2,f_1) + f'_{2;1}\circ(f_1,f_2)
            \right)^\sigma \right) .
      \end{align*}
      Note that the defect of functoriality is a coboundary.  We say that skew-symmetrization is functorial up to
      homotopy.

  \section{Applications} \label{S:Appl}
    In this section, we give two examples of applications of the theory developed in the earlier sections.  The first
    is an extension of a result of \cite{rogers_thesis_2011} on algebraic structures on $n$-plectic manifolds.  The
    second is a construction of a weak Lie 3-algebra associated to an CLWX 2-algebroid, whose skew-symmetrization is
    precisely the Lie 3-algebra of \cite[Theorem~3.10]{liu2016qp}.  We thereby give an alternative proof for the
    theorem cited.

    \subsection{Higher symplectic geometry} 
      In \cite{rogers_thesis_2011,baez_hoffnung_rogers_2008} the concept of $n$-plectic manifolds is introduced.  We
      recall some of the basic definitions here.  We then consider the case of a 3-plectic manifold and compare two
      associated algebraic structures, an $\Linf$-algebra and a dg Leibniz algebra with a certain hidden skew-symmetry.
      Both structures are examples of weak Lie 3-algebras and turn out to be isomorphic as such.  The analogous result
      for 2-plectic manifolds was shown in \cite[Appendix A]{rogers_thesis_2011}.

      \begin{df}
        An \emph{n-plectic manifold} $(M,\omega)$ is a smooth manifold $M$ with a closed, non-degenerate $(n+1)$-form
        $\omega \in \Omega^{n+1}(M)$, i.e.\@ $\diff\omega = 0$ and $\iota(v)\omega = 0$ implies $v=0$.  An
        $(n-1)$-form $\alpha \in \Omega^{n-1}(M)$ is called \emph{Hamiltonian}, if there exists a vector field
        $v_\alpha \in \VF(M)$ such that $\diff\alpha = -\iota(v_\alpha)\omega$.  In order to simplify notation, we let
        $v_\alpha = 0$ whenever $\alpha$ is not Hamiltonian, in particular when $\alpha$ is not an $(n-1)$-form.
      \end{df}

      Let $(M,\omega)$ be an $n$-plectic manifold.  Two algebraic structures are introduced on the chain complex
      \begin{align}
        L_i \defeq \begin{cases}
            \Omega^{n-1}_\text{Ham}(M) ,  & i=0,  \\
            \Omega^{n-1-i}(M) ,           & 0 < i \leq n-1 ,
          \end{cases}
        \label{nPlec:Cx}
      \end{align}
      with differential the usual de Rham differential.  Note that this is well-defined since closed forms are always
      Hamiltonian.  The first structure we introduce is that of an $n$-term $\Linf$-algebra or (semi-strict) $\Lie$
      $n$-algebra $\Linf(M,\omega)$ given by the brackets
      \begin{equation*}
        \LeibKstr{k}(\alpha_1,\dotsc,\alpha_k) \defeq \pm \iota(v_{\alpha_1},\dotsc,v_{\alpha_k}) \omega .
      \end{equation*}
      We will provide the sign in low degrees in the proof of \Cref{pp:Appl:HSG:iso}.  For more details, including the
      general definition of the sign, we refer the reader to loc.\@ cit.  The second structure on the complex of
      \cref{nPlec:Cx} is a dg Leibniz algebra given by
      \begin{equation*}
        \LeibKstr{2}[\prime](\alpha,\beta) \defeq \dL(v_\alpha)\beta ,
      \end{equation*}
      and denoted by $\Leib(M,\omega)$.  This dg Leibniz algebra actually satisfies a certain symmetry up to homotopy,
      and in the case $n=2$ is shown to be a weak Lie 2-algebra with alternator bracket
      \begin{equation*}
        \LieDstr{2;1}[\prime](\alpha,\beta) \defeq \iota(v_\alpha)\beta + \iota(v_\beta)\alpha .
      \end{equation*}
      This result can be extended to the case $n=3$, i.e.\@ the same binary bracket and alternator form a weak Lie
      3-algebra on the underlying 3-term chain complex in that case.

      \begin{pp}\label{pp:Appl:HSG:iso}
        Let $(M,\omega)$ be a 3-plectic manifold.  Then $\Linf(M,\omega)$ and $\Leib(M,\omega)$ are isomorphic as weak
        Lie 3-algebras.
      \end{pp}

      \begin{proof}
        For a 3-plectic manifold $(M,\omega)$, the chain complex underlying both $\Linf(M,\omega)$ and
        $\Leib(M,\omega)$ is
        \begin{equation*}
          (L,\diff^L) \defeq \left(
              \begin{tikzcd}
                \Omega^2_\text{Ham}(M) & \ar[l,"\diff^L"'] \Omega^1(M) & \ar[l,"\diff^L"'] \Omega^0(M) = C^\infty(M)
              \end{tikzcd}
            \right) .
        \end{equation*}
        Note that for a Hamiltonian form $\alpha$, $\diff\alpha$ need not vanish, while $\diff^L\alpha = 0$ for degree
        reasons.  The structure maps for $\Linf(M,\omega)$ are explicitly given by
        \begin{align*}
          \LeibKstr{2}(\alpha,\beta) &\defeq \iota(v_\alpha,v_\beta)\omega ,
          &\LeibKstr{3}(\alpha,\beta,\gamma) &\defeq - \iota(v_\alpha,v_\beta,v_\gamma)\omega ,
          &\LeibKstr{4}(\alpha,\beta,\gamma,\eta) &\defeq - \iota(v_\alpha,v_\beta,v_\gamma,v_\eta)\omega .
        \end{align*}

        We define a weak morphism $f\colon \Linf(M,\omega) \to \Leib(M,\omega)$ of weak Lie 3-algebras by
        \begin{align*}
          f_1(\alpha) &= \alpha ,
          & f_2(\alpha,\beta) &= - \iota(v_\alpha)\beta ,
          & f_{2;1}(\alpha,\beta) &= 0 ,  \\
          &&&& f_3(\alpha,\beta,\gamma) &= - \iota(v_\alpha,v_\beta)\gamma .
        \end{align*}
        To show that these maps define a weak morphism of weak Lie 3-algebras, we begin by verifying that they indeed
        satisfy \cref{eq:Leib3:f1,eq:Leib3:f2,eq:Leib3:f3,eq:Leib3:f4}, i.e.\@ define a morphism of Leibniz
        3-algebras.

        Clearly $\dEnd f_1 = 0$ and so \cref{eq:Leib3:f1} holds.  When $|\alpha|>0$, \cref{eq:Leib3:f2} is trivially
        satisfied.  In case $|\alpha| = |\beta| = 0$ we obtain
        \begin{align*}
          \big( \dEnd f_2 - f_1\circ\LeibKstr{2} + \LieDstr{2}[\prime]\circ(f_1,f_1) \big)(\alpha,\beta)
            &= \diff^L f_2(\alpha,\beta) - \LeibKstr{2}(\alpha,\beta) + \LieDstr{2}[\prime](\alpha,\beta)  \\
            &= \diff^L( - \iota(v_\alpha)\beta ) - \iota(v_\alpha,v_\beta)\omega + \dL(v_\alpha)\beta ,
        \intertext{which, using the fact that $\beta$ is Hamiltonian and Cartan's formula, can be written as}
            &= - \diff\iota(v_\alpha)\beta - \iota(v_\alpha)\diff\beta + \diff\iota(v_\alpha)\beta + \iota(v_\alpha)\diff\beta
            = 0 ,
        \intertext{while for $|\alpha| = 0$, $|\beta| > 0$ we have}
          \big( \dEnd f_2 - f_1\circ\LeibKstr{2} + \LieDstr{2}[\prime]\circ(f_1,f_1) \big)(\alpha,\beta)
            &= \diff^L f_2(\alpha,\beta) + f_2(\alpha,\diff^L\beta)
              - \LeibKstr{2}(\alpha,\beta) + \LieDstr{2}[\prime](\alpha,\beta)  \\
            &= \diff^L( - \iota(v_\alpha)\beta ) - \iota(v_\alpha)\diff^L\beta
              - \iota(v_\alpha,v_\beta)\omega + \dL(v_\alpha)\beta ,
        \intertext{which, using that $\beta$ is not Hamiltonian, reduces to}
            &= - \diff\iota(v_\alpha)\beta - \iota(v_\alpha)\diff\beta + \dL(v_\alpha)\beta
            = 0 .
        \end{align*}
        Evaluating \cref{eq:Leib3:f3} for $|\alpha| = |\beta| = 0$, we obtain---omitting $f_1=\id$ for brevity from
        now on
        \begin{align*}
            &\big( \dEnd(f_3) - \LeibKstr{3} 
              + f_2\circ_2\LeibKstr{2} - f_2\circ_1\LeibKstr{2} - (f_2\circ_2\LeibKstr{2})^{(12)}
              + \LeibKstr{2}[\prime]\circ_2{}f_2 - \LeibKstr{2}[\prime]\circ_1{}f_2
              - \LeibKstr{2}[\prime]\circ_2{}f_2^{(12)} \big)(\alpha,\beta,\gamma)  \\
            \begin{split}
              &\quad= \diff^L f_3(\alpha,\beta,\gamma) - f_3(\alpha,\beta,\diff^L\gamma)
                - \LeibKstr{3}(\alpha,\beta,\gamma)
                + f_2(\alpha,\LeibKstr{2}(\beta,\gamma)) - f_2(\LeibKstr{2}(\alpha,\beta),\gamma)
                - f_2(\beta,\LeibKstr{2}(\alpha,\gamma))  \\
              &\qquad+ \LieDstr{2}[\prime](\alpha,f_2(\beta,\gamma)) - \LieDstr{2}[\prime](f_2(\alpha,\beta),\gamma)
                - \LieDstr{2}[\prime](\beta,f_2(\alpha,\gamma))
            \end{split}  \\
            \begin{split}
              &\quad= \diff^L(-\iota(v_\alpha,v_\beta)\gamma) + \iota(v_\alpha,v_\beta)\diff^L\gamma
                + \iota(v_\alpha,v_\beta,v_\gamma)\omega - \iota(v_\alpha)\iota(v_\beta,v_\gamma)\omega
                + \iota(v_{\LeibKstr{2}(\alpha,\beta)})\gamma + \iota(v_\beta)\iota(v_\alpha,v_\gamma)\omega  \\
              &\qquad- \dL(v_\alpha)\iota(v_\beta)\gamma + \dL(v_\beta)\iota(v_\alpha)\gamma
            \end{split}
        \intertext{which, using the identities $v_{\LeibKstr{2}(\alpha,\beta)} = [v_\alpha,v_\beta]$ and
      $\iota([v_\alpha,v_\beta])\omega = [\dL(v_\alpha),\iota(v_\beta)]\omega$, becomes}
            &\quad= - \diff\iota(v_\alpha,v_\beta)\gamma + \iota(v_\alpha,v_\beta)\diff^L\gamma
              - \iota(v_\alpha,v_\beta,v_\gamma)\omega
              - \iota(v_\beta)\dL(v_\alpha)\gamma + \dL(v_\beta)\iota(v_\alpha)\gamma .
        \intertext{Now consider two cases: when $|\gamma| = 0$, $\diff^L\gamma = 0$ and we have}
            &\quad= - \diff\iota(v_\alpha,v_\beta)\gamma - \iota(v_\alpha,v_\beta,v_\gamma)\omega
              - \iota(v_\alpha,v_\beta)\diff\gamma + \diff\iota(v_\alpha,v_\beta)\gamma
            = 0 ,
        \intertext{while for $|\gamma| > 0$, $\iota(v_\alpha,v_\beta)\gamma = 0$ and $v_\gamma = 0$, and we obtain}
            &\quad= \iota(v_\alpha,v_\beta)\diff\gamma - \iota(v_\alpha,v_\beta)\diff\gamma
            = 0 .
        \end{align*}
        To verify \cref{eq:Leib3:f4}, it is sufficient to consider the case $|\alpha|=|\beta|=|\gamma|=|\eta|=0$.
        Leaving out the terms vanishing for degree reasons, we obtain
        \begin{align*}
          \begin{split}
            &\big( \LeibKstr{4}
                - f_2\circ_2\LeibKstr{3} + (f_2\circ_2\LeibKstr{3})^{(12)} - (f_2\circ_2\LeibKstr{3})^{(123)}
                + f_3\circ_1\LeibKstr{2} - f_3\circ_2\LeibKstr{2} + (f_3\circ_2\LeibKstr{2})^{(12)}  \\
              &\quad+ f_3\circ_3\LeibKstr{2} - (f_3\circ_3\LeibKstr{2})^{(23)} + (f_3\circ_3\LeibKstr{2})^{(132)}
                - \LeibKstr{2}[\prime]\circ_2{}f_3
                + (\LeibKstr{2}[\prime]\circ_2{}f_3)^{(12)} - (\LeibKstr{2}[\prime]\circ_2{}f_3)^{(123)}
              \big)(\alpha,\beta,\gamma,\eta)
          \end{split}  \\
          \begin{split}
            &\quad= \LeibKstr{4}(\alpha,\beta,\gamma,\eta)
              - f_2(\alpha,\LeibKstr{3}(\beta,\gamma,\eta))
              + f_2(\beta,\LeibKstr{3}(\alpha,\gamma,\eta))
              - f_2(\gamma,\LeibKstr{3}(\alpha,\beta,\eta))
              + f_3(\LeibKstr{2}(\alpha,\beta),\gamma,\eta)  \\
            &\qquad- f_3(\alpha,\LeibKstr{2}(\beta,\gamma),\eta)
              + f_3(\beta,\LeibKstr{2}(\alpha,\gamma),\eta)
              + f_3(\alpha,\beta,\LeibKstr{2}(\gamma,\eta))
              - f_3(\alpha,\gamma,\LeibKstr{2}(\beta,\eta))
              + f_3(\beta,\gamma,\LeibKstr{2}(\alpha,\eta))  \\
            &\qquad- \LeibKstr{2}[\prime](\alpha,f_3(\beta,\gamma,\eta))
              + \LeibKstr{2}[\prime](\beta,f_3(\alpha,\gamma,\eta))
              - \LeibKstr{2}[\prime](\gamma,f_3(\alpha,\beta,\eta)) .
          \end{split}
        \intertext{Note that the terms $f_2\circ_2\LeibKstr{3}$ cancel with $f_3\circ_3\LeibKstr{2}$, leaving us
        with}
          \begin{split}
            &\quad= - \iota(v_\alpha,v_\beta,v_\gamma,v_\eta)\omega
          - \iota(v_{\LeibKstr{2}(\alpha,\beta)},v_\gamma)\eta
              + \iota(v_\alpha,v_{\LeibKstr{2}(\beta,\gamma)})\eta
              - \iota(v_\beta,v_{\LeibKstr{2}(\alpha,\gamma)})\eta  \\
            &\qquad+ \dL(v_\alpha)\iota(v_\beta,v_\gamma)\eta
              - \dL(v_\beta)\iota(v_\alpha,v_\gamma)\eta
              + \dL(v_\gamma)\iota(v_\alpha,v_\beta)\eta ,
          \end{split}  \\
          &\quad= - \iota(v_\alpha,v_\beta,v_\gamma,v_\eta)\omega
            + \iota(v_\gamma)\iota(v_\beta)\dL(v_\alpha)\eta
            - \iota(v_\gamma)\dL(v_\beta)\iota(v_\alpha)\eta
            + \dL(v_\gamma)\iota(v_\alpha,v_\beta)\eta  \\
          &\quad= - \iota(v_\alpha,v_\beta,v_\gamma,v_\eta)\omega
            + \iota(v_\gamma)\iota(v_\beta)\iota(v_\alpha)\diff\eta
            + \diff\iota(v_\gamma)\iota(v_\alpha,v_\beta)\eta
          = 0 .
        \end{align*}

        We proceed to show that $f$ additionally satisfies
        \cref{eq:ELie3:f21,eq:ELie3:f211,eq:ELie3:f31,eq:ELie3:f32} and hence defines a morphism of weak Lie
        3-algebras.  For \cref{eq:ELie3:f21} we find
        \begin{align*}
          \big( f_2 + f_2^{(12)} + \LieDstr{2;1}[\prime] \big)(\alpha,\beta)
            = - \iota(v_\alpha)\beta - \iota(v_\beta)\alpha + \iota(v_\alpha)\beta + \iota(v_\beta)\alpha
            = 0 .
        \end{align*}
        Since all involved maps vanish, \cref{eq:ELie3:f211} is trivially satisfied.
        Finally we consider \cref{eq:ELie3:f31,eq:ELie3:f32},
        \begin{align*}
          \big( f_3 + f_3^{(12)} \big)(\alpha,\beta,\gamma)
            &= - \iota(v_\alpha,v_\beta)\gamma - \iota(v_\beta,v_\alpha)\gamma
            = 0 ,  \\
          \big( f_3 + f_3^{(23)} - \LieDstr{2;1}[\prime]\circ_1{}f_2
            - (\LieDstr{2;1}[\prime]\circ_2{}f_2)^{(12)} \big)(\alpha,\beta,\gamma)
            &= - \iota(v_\alpha,v_\beta)\gamma - \iota(v_\alpha,v_\gamma)\beta
              + \iota(v_\gamma)\iota(v_\alpha)\beta + \iota(v_\beta)\iota(v_\alpha)\gamma
            = 0 .
        \end{align*}
        This concludes the proof.
      \end{proof}

    \subsection{Higher Courant algebroids}
      In \cite{liu2016qp}, the notion of an \emph{CLWX 2-algebroid} is introduced as a higher analogue of a Courant
      algebroid.  A Lie 3-algebra is associated to any CLWX 2-algebroid in \cite[Theorem~3.10]{liu2016qp}.
      In this section, we give a construction of a weak Lie 3-algebra for any CLWX 2-algebroid.  We then show that the
      Lie 3-algebra constructed in loc.\@ cit.\@ is in fact the skew-symmetrization of our weak Lie 3-algebra.

      \begin{df}\label{df:CLWX2}
        Let $E = (E_0 \xleftarrow{\partial} E_1)$ be a 2-term dg vector bundle over $M$ equipped with
        a morphism $\rho\colon E \to TM$ of dg vector bundles,
        a (graded) bilinear map $\circ\colon \Gamma E \tensor \Gamma E \to \Gamma E$ which is skew-symmetric on
        $\Gamma E_0 \tensor \Gamma E_0$,
        a (graded) 3-form $\Omega\colon \Gamma E \tensor \Gamma E \tensor \Gamma E \to \Gamma E[1]$,
        and a non-degenerate symmetric bilinear form $S\colon \Gamma E \tensor \Gamma E \to C^\infty(M)$.
        Using these data, we define a map $\mathcal{D}\colon C^\infty(M) \to \Gamma E_1$ by
        \begin{equation}
          S(e,\mathcal{D}f) = \rho(e)(f) , \quad \forall e \in \Gamma E .
          \label{df:CLWX2:D}
        \end{equation}
        We call $\mathcal{E} = (E_0 \xleftarrow{\partial} E_1,\rho,\circ,\Omega,S)$ an \emph{CLWX 2-algebroid}, if the
        following conditions are satisfied:
        \begin{conditions}
          \item $E_0$, $E_1$ are isotropic, i.e.\@ $S(E_i,E_i) = 0$ for $i=0,1$,
          \item $(\Gamma E_0 \xleftarrow{\partial} \Gamma E_1, \circ, \Omega)$ is a Leibniz 2-algebra,
            \label{df:CLWX2:Leib2}
          \item $e \circ e = \frac{1}{2} \mathcal{D} S(e,e)$ for all $e \in \Gamma E$,
            \label{df:CLWX2:l21:01}
          \item $S(\partial e_1, e_2) = S(e_1, \partial e_2)$ for all $e_i \in \Gamma E$,
            \label{df:CLWX2:l21:11}
          \item $\rho(e_1)S(e_2,e_3) = S(e_1 \circ e_2, e_3) + S(e_2, e_1 \circ e_3)$ for all $e_i \in \Gamma E$, and
            \label{df:CLWX2:l32}
          \item $S(\Omega(e_1,e_2,e_3),e_4) = -S(e_3,\Omega(e_1,e_2,e_4))$ for all $e_i \in \Gamma E$.
            \label{df:CLWX2:l43}
        \end{conditions}
      \end{df}

      \begin{pp}\label{pp:CLWX2:ELie3}
        For any CLWX 2-algebroid, there is an associated complex
        \begin{equation*}
          (L,\diff) \defeq \left(
              \begin{tikzcd}
                \Gamma E_0 & \ar[l,"\partial"'] \Gamma E_1 & \ar[l,"\mathcal{D}"'] C^\infty(M)
              \end{tikzcd}
            \right) ,
        \end{equation*}
        which, equipped with structure maps
        \begin{align*}
          \LieDstr{2} &= (-\circ-) + S\circ_2\mathcal{D} ,
          & \LieDstr{2;1} &= S ,
          & \LieDstr{2;1,1} &= 0 ,  \\
          && \LieDstr{3} &= \Omega ,
          & \LieDstr{3;1} &= 0 ,  \\
          &&&& \LieDstr{3;2} &= 0 ,  \\
          &&&& \LieDstr{4} &= 0 ,
        \end{align*}
        defines a weak Lie 3-algebra.
      \end{pp}

      \begin{proof}
        The fact that $(L,\diff)$ is indeed a complex is equivalent to $\partial\mathcal{D} = 0$, which is shown in
        \cite[Lemma 3.6]{liu2016qp}.

        We begin by showing that $(L,\diff,\LieDtw)$ is a Leibniz 3-algebra, i.e.\@ by checking
        \cref{eq:Leib3:l2,eq:Leib3:l3,eq:Leib3:l4,eq:Leib3:l5}.  Since, by \cref{df:CLWX2:Leib2}, we already know that
        the 2-term truncation of $L$ forms a Leibniz 2-algebra, it is sufficient to verify that these equations hold
        on tuples containing at least one degree 2 element.  To keep notation as simple as possible, we denote
        arbitrary sections of $E$ by $e \in \Gamma E$ and indicate degree using superscript, i.e.\@
        $e^i \in \Gamma E_i$, and use $f$ for functions $f \in L_2 = C^\infty(M)$.  We use subscript to distinguish
        multiple elements of the same degree, e.g.\@ $e^0_1$ and $e^0_2$.

        Evaluating \cref{eq:Leib3:l2}, we obtain
        \begin{align}
          \partial\big(\LieDstr{2}\big)(e^0,f) &= \diff\big(\LieDstr{2}(e^0,f)\big) - \LieDstr{2}(e^0,\mathcal{D}f)
            = \mathcal{D}S(e^0,\mathcal{D}f) - e^0\circ\mathcal{D}f = 0 ,
            \label{eq:CLWX2:ELie3:l2:02}
        \intertext{and}
          \partial\big(\LieDstr{2}\big)(f,e^0) &= - \LieDstr{2}(\mathcal{D}f,e^0) = - \mathcal{D}f\circ e^0 = 0 ,
            \label{eq:CLWX2:ELie3:l2:20}
        \intertext{which are both shown to hold in \cite[Lemma 3.6]{liu2016qp}.  In addition we have}
          \partial\big(\LieDstr{2}\big)(e^1,f) &= - \LieDstr{2}(\partial e^1,f)
            = - S(\partial e^0,\mathcal{D}f)
            = - S(e^0,\partial\mathcal{D}f) = 0 ,
        \end{align}
        which holds since $\partial\mathcal{D} = 0$, as we have seen above.  In the remaining cases, i.e.\@ for
        $(f,e^1)$ and $(f_1,f_2)$, the equation holds trivially since all structure maps vanish.
        For \cref{eq:Leib3:l3} we obtain
        \begin{align*}
          &\big( \partial\LieDstr{3} - \LieDstr{2}\circ_2\LieDstr{2} + \LieDstr{2}\circ_1\LieDstr{2}
            + (\LieDstr{2}\circ_2\LieDstr{2})^{(12)} \big)(e^0_1,e^0_2,f)  \\
          &\quad= - \LieDstr{2}(e^0_1,\LieDstr{2}(e^0_2,f)) + \LieDstr{2}(\LieDstr{2}(e^0_1,e^0_2),f)
            + \LieDstr{2}(e^0_2,\LieDstr{2}(e^0_1,f))  \\
          &\quad= - S(e^0_1,\mathcal{D}S(e^0_2,\mathcal{D}f)) + S(e^0_1\circ e^0_2,\mathcal{D}f)
            + S(e^0_2,\mathcal{D}S(e^0_1,\mathcal{D}f)) ,
          \intertext{which, by \cref{eq:CLWX2:ELie3:l2:02}, can be written as}
          &\quad= - S(e^0_1,\mathcal{D}S(e^0_2,\mathcal{D}f)) + S(e^0_1\circ e^0_2,\mathcal{D}f)
            + S(e^0_2,e^0_1\circ\mathcal{D}f) ,
          \intertext{and using \cref{df:CLWX2:l32} becomes}
          &\quad= - S(e^0_1,\mathcal{D}S(e^0_2,\mathcal{D}f)) + \rho(e^0_1)S(e^0_2,\mathcal{D}f) = 0 ,
        \end{align*}
        where the last equality holds by definition of $\mathcal{D}$.
        Since \cref{eq:Leib3:l4} is of degree 1, it holds for degree reasons when evaluated on any tuple containing a
        degree 2 element.  Similarly, \cref{eq:Leib3:l5} is always satisfied for degree reasons.  This completes the
        proof that $L$ is a Leibniz 3-algebra.

        We proceed to show that $L$ is in fact a weak Lie 3-algebra.  Consider \cref{eq:ELie3:l21}.
        By isotropy of $E_0$ and skew-symmetry of $\circ$ on $L_0 \tensor L_0$ we see that
        \begin{align*}
          \big( \partial(\LieDstr{2;1}) - \LieDstr{2} - \LieDstr{2}[(12)] \big)(e^0_1,e^0_2)
            &= \partial S(e^0_1,e^0_2) - e^0_1 \circ e^0_2 - e^0_2 \circ e^0_1 = 0 .
        \intertext{In addition we find}
          \big( \partial(\LieDstr{2;1}) - \LieDstr{2} - \LieDstr{2}[(12)] \big)(e^0,e^1)
            &= \mathcal{D} S(e^0,e^1) - e^0 \circ e^1 - e^1 \circ e^0  \\
            &= \frac{1}{2} \mathcal{D} S(e^0+e^1,e^0+e^1) - (e^0+e^1) \circ (e^0+e^1) = 0 ,
        \intertext{which holds using \cref{df:CLWX2:l21:01}, and}
          \big( \partial(\LieDstr{2;1}) - \LieDstr{2} - \LieDstr{2}[(12)] \big)(e^1,e^1)
            &= S(\partial e^1_1, e^1_2) - S(e^1_1, \partial e^1_2) = 0 ,
        \intertext{which holds by \cref{df:CLWX2:l21:11}.  In the remaining cases the equation is trivially satisfied,
        e.g.}
          \big( \partial(\LieDstr{2;1}) - \LieDstr{2} - \LieDstr{2}[(12)] \big)(e^0,f)
            &= S(e^0,\mathcal{D}f) - S(e^0,\mathcal{D}f) = 0 .
        \end{align*}
        Next, we consider \cref{eq:ELie3:l211}.  Since $\LieDstr{2;1,1} = 0$, the only non-trivial case is
        \begin{equation*}
          \big( \partial(\LieDstr{2;1,1}) + \LieDstr{2;1} - \LieDstr{2;1}[(12)] \big)(e^0,e^1)
            = S(e^0,e^1) - S(e^1,e^0) = 0 ,
        \end{equation*}
        which holds by symmetry of $S$.  For \cref{eq:ELie3:l31,eq:ELie3:l32} we obtain
        \begin{equation*}
          \big( \partial(\LieDstr{3;1}) - \LieDstr{3} - \LieDstr{3}[(12)] - \LieDstr{2}\circ_1\LieDstr{2;1} \big)
              (e^0_1,e^0_2,e^0_3)
            = - \Omega(e^0_1,e^0_2,e^0_3) - \Omega(e^0_2,e^0_1,e^0_3) = 0 ,
        \end{equation*}
        which holds by skew-symmetry of $\Omega$, and
        \begin{align*}
          &\big( \partial(\LieDstr{3;2}) - \LieDstr{3} - \LieDstr{3}[(23)] + \LieDstr{2}\circ_2\LieDstr{2;1}
              - \LieDstr{2;1}\circ_1\LieDstr{2} - (\LieDstr{2;1}\circ_2\LieDstr{2})^{(12)} \big)(e^0_1,e^0_2,e^0_3) \\
          &\quad= - \Omega(e^0_1,e^0_2,e^0_3) - \Omega(e^0_1,e^0_3,e^0_2)
              + S(e^0_1,\mathcal{D}S(e^0_2,e^0_3)) - S(e^0_1\circ{}e^0_2,e^0_3) - S(e^0_2,e^0_1\circ{}e^0_3) = 0 ,
        \end{align*}
        which follows again from skew-symmetry of $\Omega$ and using \cref{df:CLWX2:l32}.
        \Cref{eq:ELie3:l2111,eq:ELie3:l311,eq:ELie3:l312,eq:ELie3:l322,eq:ELie3:l41,eq:ELie3:l42} hold trivially since
        each term vanishes for degree reasons or because one of the relevant structure maps is zero.  Finally,
        \cref{eq:ELie3:l43} reduces to
        \begin{equation*}
          \big( \LieDstr{2;1}\circ_1\LieDstr{3} + (\LieDstr{2;1}\circ_2\LieDstr{3})^{(123)} \big)
              (e^0_1,e^0_2,e^0_3,e^0_4)
            = S(\Omega(e^0_1,e^0_2,e^0_3),e^0_4) + S(e^0_3,\Omega(e^0_1,e^0_2,e^0_4)) = 0 ,
        \end{equation*}
        which is precisely \cref{df:CLWX2:l43}.
      \end{proof}

      \begin{cor}\label{cor:CLWX2:Lie3}
        The skew-symmetrization $\SSy{L}$ of $L$ is a Lie 3-algebra.  Its structure maps are explicitly given by
        \begin{align*}
          \LieKstr{2} &= \frac{1}{2} \left( (-\circ-) - (-\circ-)^{(12)}
              - S\circ_1{}\mathcal{D} + S\circ_2{}\mathcal{D} \right) ,  \\
          \LieKstr{3} &= \Omega - \frac{1}{12} \sum_{\sigma\in\Sy_3} (-1)^{|\sigma|} \cdot
                \big( S \circ_1 (-\circ-) \big)^\sigma ,  \\
          \LieKstr{4} &= S \circ_1 \Omega .
        \end{align*}
      \end{cor}

      \begin{proof}
        Taking into account the skew-symmetry of $\Omega$, the symmetry of $S$, and \cref{df:CLWX2:l43}, it is clear
        that the given terms are precisely the skew-symmetrized structure maps of \Cref{df:SS}.
      \end{proof}

      The structure maps of \Cref{cor:CLWX2:Lie3} are precisely the ones given in \cite[Theorem~3.10]{liu2016qp}.
      Hence, our construction in \Cref{pp:CLWX2:ELie3} gives an alternative proof of the cited theorem, based on the
      methods developed in \Cref{S:SS}.

  \appendix

  \section{Computations}
    This appendix contains some long and tedious computations that were removed from the main text as to not disturb
    its flow.

    \subsection{Proof of Lemma \ref{lm:LieD:coopd}} \label{A:comp:coass}
      We verify that $(\decomp\circ\id)\decomp = (\id\circ\decomp)\decomp$ when evaluated on $\LieDgen{4;1}$,
      $\LieDgen{4;2}$, or $\LieDgen{4;3}$ by computing both sides of the equation individually:
      \begin{align*}
        (\decomp\circ\id)\decomp(\LieDgen{4;1}) &=
          1 \circ (1) \circ (\LieDgen{4;1})
          - \big( 1 \circ (\LieDgen{2})
            + \LieDgen{2} \circ (1,1) \big) \circ (\LieDgen{3;1},1)
          - \big( 1 \circ (\LieDgen{2})
            + \LieDgen{2} \circ (1,1) \big) \circ (1,\LieDgen{3;1})^{(123)}  \\
          &\quad+ \big( 1 \circ (\LieDgen{2})
            + \LieDgen{2} \circ (1,1) \big) \circ (\LieDgen{2;1},\LieDgen{2})
          + \big( 1 \circ (\LieDgen{3})
            - \LieDgen{2} \circ (\LieDgen{2},1)
            + \LieDgen{2} \circ (1,\LieDgen{2})
            - \LieDgen{2} \circ (1,\LieDgen{2})^{(12)}  \\
            &\quad+ \LieDgen{3}\circ(1,1,1) \big) \circ (\LieDgen{2;1},1,1)
          + \big( 1 \circ (\LieDgen{3;1})
            + \LieDgen{2} \circ (\LieDgen{2;1},1)
            + \LieDgen{3;1} \circ (1,1,1) \big) \circ (1,1,\LieDgen{2})  \\
          &\quad+ \big( 1 \circ (\LieDgen{4;1})
            - \LieDgen{2} \circ (\LieDgen{3;1},1)
            - \LieDgen{2} \circ (1,\LieDgen{3;1})^{(123)}
            + \LieDgen{2} \circ (\LieDgen{2;1},\LieDgen{2})
            + \LieDgen{3} \circ (\LieDgen{2;1},1,1)  \\
            &\quad+ \LieDgen{3;1} \circ (1,1,\LieDgen{2})
            + \LieDgen{4;1} \circ (1,1,1,1) \big) \circ (1,1,1,1)  , \\
        (\id\circ\decomp)\decomp(\LieDgen{4;1}) &=
          1 \circ \big( 1 \circ (\LieDgen{4;1})
            - \LieDgen{2} \circ (\LieDgen{3;1},1)
            - \LieDgen{2} \circ (1,\LieDgen{3;1})^{(123)}
            + \LieDgen{2} \circ (\LieDgen{2;1},\LieDgen{2})
            + \LieDgen{3} \circ (\LieDgen{2;1},1,1)  \\
            &\quad+ \LieDgen{3;1} \circ (1,1,\LieDgen{2})
            + \LieDgen{4;1} \circ (1,1,1,1) \big)
          - \LieDgen{2} \circ \big( 1 \circ (\LieDgen{3;1})
            + \LieDgen{2} \circ (\LieDgen{2;1},1)
            + \LieDgen{3;1} \circ (1,1,1), 1 \circ (1) \big)  \\
          &\quad- \LieDgen{2} \circ \big( 1 \circ (1), 1 \circ (\LieDgen{3;1})
            + \LieDgen{2} \circ (\LieDgen{2;1},1)
            + \LieDgen{3;1} \circ (1,1,1) \big)^{(123)}
          + \LieDgen{2} \circ \big( 1 \circ (\LieDgen{2;1})
            + \LieDgen{2;1} \circ (1,1),  \\
            &\quad1 \circ (\LieDgen{2})
            + \LieDgen{2} \circ (1,1) \big)
          + \LieDgen{3} \circ \big( 1 \circ (\LieDgen{2;1})
            + \LieDgen{2;1} \circ (1,1), 1 \circ (1), 1 \circ (1) \big)
          + \LieDgen{3;1} \circ \big( 1 \circ (1), 1 \circ (1),  \\
            &\quad 1 \circ (\LieDgen{2})
            + \LieDgen{2} \circ (1,1) \big)
          + \LieDgen{4;1} \circ \big( 1 \circ (1), 1 \circ (1), 1 \circ (1), 1 \circ (1) \big)  ,
      \intertext{and see (using the associativity of the composite product $\circ$) that they agree.  In the same way
      we compute}
        (\decomp\circ\id)\decomp(\LieDgen{4;2}) &=
          1 \circ (1) \circ (\LieDgen{4;2})
          - \big( 1 \circ (\LieDgen{2}) + \LieDgen{2} \circ (1,1) \big) \circ (\LieDgen{3;2},1)
          - \big( 1 \circ (\LieDgen{2}) + \LieDgen{2} \circ (1,1) \big) \circ (1,\LieDgen{3;1})
          + \big( 1 \circ (\LieDgen{2})  \\
            &\quad+ \LieDgen{2} \circ (1,1) \big) \circ (\LieDgen{2;1},\LieDgen{2})^{(132)}
          + \big( 1 \circ (\LieDgen{3;1})
            + \LieDgen{2} \circ (\LieDgen{2;1},1)
            + \LieDgen{3;1} \circ (1,1,1) \big) \circ (\LieDgen{2},1,1)
          + \big( 1 \circ (\LieDgen{3;1})  \\
            &\quad+ \LieDgen{2} \circ (\LieDgen{2;1},1)
            + \LieDgen{3;1} \circ (1,1,1) \big) \circ (1,\LieDgen{2},1)^{(12)}
          - \big( 1 \circ (\LieDgen{3})
            - \LieDgen{2} \circ (\LieDgen{2},1)
            + \LieDgen{2} \circ (1,\LieDgen{2})
            - \LieDgen{2} \circ (1,\LieDgen{2})^{(12)}  \\
            &\quad+ \LieDgen{3}\circ(1,1,1) \big) \circ (1,\LieDgen{2;1},1)
          + \big( 1 \circ (\LieDgen{3;1})
            + \LieDgen{2} \circ (\LieDgen{2;1},1)
            + \LieDgen{3;1} \circ (1,1,1) \big) \circ (1,1,\LieDgen{2})^{(132)}  \\
          &\quad+ \big( 1 \circ (\LieDgen{4;2})
            - \LieDgen{2}\circ(\LieDgen{3;2},1)
            - \LieDgen{2}\circ(1,\LieDgen{3;1})
            + \LieDgen{2}\circ(\LieDgen{2;1},\LieDgen{2})^{(132)}
            + \LieDgen{3;1}\circ(\LieDgen{2},1,1)
            + \LieDgen{3;1}\circ(1,\LieDgen{2},1)^{(12)}  \\
            &\quad- \LieDgen{3}\circ(1,\LieDgen{2;1},1)
            + \LieDgen{3;1}\circ(1,1,\LieDgen{2})^{(132)}
            + \LieDgen{4;2} \circ (1,1,1,1) \big) \circ (1,1,1,1)  , \\
        (\id\circ\decomp)\decomp(\LieDgen{4;2}) &=
          1 \circ \big( 1 \circ (\LieDgen{4;2})
            - \LieDgen{2}\circ(\LieDgen{3;2},1)
            - \LieDgen{2}\circ(1,\LieDgen{3;1})
            + \LieDgen{2}\circ(\LieDgen{2;1},\LieDgen{2})^{(132)}
            + \LieDgen{3;1}\circ(\LieDgen{2},1,1)
            + \LieDgen{3;1}\circ(1,\LieDgen{2},1)^{(12)}  \\
            &\quad- \LieDgen{3}\circ(1,\LieDgen{2;1},1)
            + \LieDgen{3;1}\circ(1,1,\LieDgen{2})^{(132)}
            + \LieDgen{4;2} \circ (1,1,1,1) \big)
          - \LieDgen{2} \circ \big( 1 \circ (\LieDgen{3;2})
            - \LieDgen{2;1} \circ (\LieDgen{2},1)  \\
            &\quad- \LieDgen{2} \circ (1,\LieDgen{2;1})
            - \LieDgen{2;1} \circ (1,\LieDgen{2})^{(12)}
            + \LieDgen{3;2} \circ (1,1,1)
            , 1 \circ (1) \big)
          - \LieDgen{2} \circ \big( 1 \circ (1),
            1 \circ (\LieDgen{3;1})
            + \LieDgen{2} \circ (\LieDgen{2;1},1)  \\
            &\quad+ \LieDgen{3;1} \circ (1,1,1) \big)
          + \LieDgen{2} \circ \big( 1 \circ (\LieDgen{2;1})
            + \LieDgen{2;1} \circ (1,1) ,
            1 \circ (\LieDgen{2}) + \LieDgen{2} \circ (1,1) \big)^{(132)}  \\
          &\quad+ \LieDgen{3;1} \circ \big( 1 \circ (\LieDgen{2})
            + \LieDgen{2} \circ (1,1) ,
            1 \circ (1),
            1 \circ (1) \big)
          + \LieDgen{3;1} \circ \big( 1 \circ (1),
            1 \circ (\LieDgen{2}) + \LieDgen{2} \circ (1,1) ,
            1 \circ (1) \big)^{(12)}  \\
          &\quad- \LieDgen{3} \circ \big( 1 \circ (1),
            1 \circ (\LieDgen{2;1}) + \LieDgen{2;1} \circ (1,1) ,
            1 \circ (1) \big)
          + \LieDgen{3;1} \circ \big( 1 \circ (1),
            1 \circ (1),
            1 \circ (\LieDgen{2}) + \LieDgen{2} \circ (1,1) \big)^{(132)}  \\
          &\quad+ \LieDgen{4;2} \circ \big( 1 \circ (1),
            1 \circ (1),
            1 \circ (1),
            1 \circ (1) \big)  ,
        \intertext{and}
        (\decomp\circ\id)\decomp(\LieDgen{4;3}) &=
          1 \circ (1) \circ (\LieDgen{4;3})
            + \big( 1 \circ (\LieDgen{2;1}) + \LieDgen{2;1} \circ (1,1) \big) \circ (1,\LieDgen{3})^{(123)}
            - \big( 1 \circ (\LieDgen{2}) + \LieDgen{2} \circ (1,1) \big) \circ (1,\LieDgen{3;2})
            + \big( 1 \circ (\LieDgen{2})  \\
              &\quad+ \LieDgen{2} \circ (1,1) \big) \circ (1,\LieDgen{3;2})^{(12)}
            + \big( 1 \circ (\LieDgen{2;1}) + \LieDgen{2;1} \circ (1,1) \big) \circ (\LieDgen{3},1)
            - \big( 1 \circ (\LieDgen{2}) + \LieDgen{2} \circ (1,1) \big) \circ (\LieDgen{2},\LieDgen{2;1})  \\
            &\quad- \big( 1 \circ (\LieDgen{2;1}) + \LieDgen{2;1} \circ (1,1) \big)
              \circ (\LieDgen{2},\LieDgen{2})^{(132)}
            + \big( 1 \circ (\LieDgen{2;1}) + \LieDgen{2;1} \circ (1,1) \big)
              \circ (\LieDgen{2},\LieDgen{2})^{(23)}  \\
            &\quad+ \big( 1 \circ (\LieDgen{3;2})
              - \LieDgen{2;1} \circ (\LieDgen{2},1)
              - \LieDgen{2} \circ (1,\LieDgen{2;1})
              - \LieDgen{2;1} \circ (1,\LieDgen{2})^{(12)}
              + \LieDgen{3;2} \circ (1,1,1) \big) \circ (\LieDgen{2},1,1)  \\
            &\quad- \big( 1 \circ (\LieDgen{3;2})
              - \LieDgen{2;1} \circ (\LieDgen{2},1)
              - \LieDgen{2} \circ (1,\LieDgen{2;1})
              - \LieDgen{2;1} \circ (1,\LieDgen{2})^{(12)}
              + \LieDgen{3;2} \circ (1,1,1) \big) \circ (1,1,\LieDgen{2})^{(23)}  \\
            &\quad+ \big( 1 \circ (\LieDgen{3;2})
              - \LieDgen{2;1} \circ (\LieDgen{2},1)
              - \LieDgen{2} \circ (1,\LieDgen{2;1})
              - \LieDgen{2;1} \circ (1,\LieDgen{2})^{(12)}
              + \LieDgen{3;2} \circ (1,1,1) \big) \circ (1,1,\LieDgen{2})^{(132)}  \\
            &\quad+ \big( 1 \circ (\LieDgen{3})
              - \LieDgen{2} \circ (\LieDgen{2},1)
              + \LieDgen{2} \circ (1,\LieDgen{2})
              - \LieDgen{2} \circ (1,\LieDgen{2})^{(12)}
              + \LieDgen{3} \circ (1,1,1) \big) \circ (1,1,\LieDgen{2;1})  \\
            &\quad- \big( 1 \circ (\LieDgen{3;2})
              - \LieDgen{2;1} \circ (\LieDgen{2},1)
              - \LieDgen{2} \circ (1,\LieDgen{2;1})
              - \LieDgen{2;1} \circ (1,\LieDgen{2})^{(12)}
              + \LieDgen{3;2} \circ (1,1,1) \big) \circ (1,\LieDgen{2},1)  \\
            &\quad+ \big( 1 \circ (\LieDgen{3;2})
              - \LieDgen{2;1} \circ (\LieDgen{2},1)
              - \LieDgen{2} \circ (1,\LieDgen{2;1})
              - \LieDgen{2;1} \circ (1,\LieDgen{2})^{(12)}
              + \LieDgen{3;2} \circ (1,1,1) \big) \circ (1,\LieDgen{2},1)^{(12)}  \\
          &\quad+ \big( 1 \circ (\LieDgen{4;3})
            + \LieDgen{2;1} \circ (1,\LieDgen{3})^{(123)}
            - \LieDgen{2} \circ (1,\LieDgen{3;2})
            + \LieDgen{2} \circ (1,\LieDgen{3;2})^{(12)}
            + \LieDgen{2;1} \circ (\LieDgen{3},1)
            - \LieDgen{2} \circ (\LieDgen{2},\LieDgen{2;1})  \\
            &\quad- \LieDgen{2;1} \circ (\LieDgen{2},\LieDgen{2})^{(132)}
            + \LieDgen{2;1} \circ (\LieDgen{2},\LieDgen{2})^{(23)}
            + \LieDgen{3;2} \circ (\LieDgen{2},1,1)
            - \LieDgen{3;2} \circ (1,1,\LieDgen{2})^{(23)}
            + \LieDgen{3;2} \circ (1,1,\LieDgen{2})^{(132)}  \\
            &\quad+ \LieDgen{3} \circ (1,1,\LieDgen{2;1})
            - \LieDgen{3;2} \circ (1,\LieDgen{2},1)
            + \LieDgen{3;2} \circ (1,\LieDgen{2},1)^{(12)}
            + \LieDgen{4;3} \circ (1,1,1,1) \big) \circ (1,1,1,1)  , \\
        (\id\circ\decomp)\decomp(\LieDgen{4;3}) &=
          1 \circ \big( 1 \circ (\LieDgen{4;3})
            + \LieDgen{2;1} \circ (1,\LieDgen{3})^{(123)}
            - \LieDgen{2} \circ (1,\LieDgen{3;2})
            + \LieDgen{2} \circ (1,\LieDgen{3;2})^{(12)}
            + \LieDgen{2;1} \circ (\LieDgen{3},1)
            - \LieDgen{2} \circ (\LieDgen{2},\LieDgen{2;1})  \\
            &\quad- \LieDgen{2;1} \circ (\LieDgen{2},\LieDgen{2})^{(132)}
            + \LieDgen{2;1} \circ (\LieDgen{2},\LieDgen{2})^{(23)}
            + \LieDgen{3;2} \circ (\LieDgen{2},1,1)
            - \LieDgen{3;2} \circ (1,1,\LieDgen{2})^{(23)}
            + \LieDgen{3;2} \circ (1,1,\LieDgen{2})^{(132)}  \\
            &\quad+ \LieDgen{3} \circ (1,1,\LieDgen{2;1})
            - \LieDgen{3;2} \circ (1,\LieDgen{2},1)
            + \LieDgen{3;2} \circ (1,\LieDgen{2},1)^{(12)}
            + \LieDgen{4;3} \circ (1,1,1,1) \big)  \\
          &\quad+ \LieDgen{2;1} \circ \big(
            1 \circ (1) ,
            1 \circ (\LieDgen{3})
              - \LieDgen{2} \circ (\LieDgen{2},1)
              + \LieDgen{2} \circ (1,\LieDgen{2})
              - \LieDgen{2} \circ (1,\LieDgen{2})^{(12)}
              + \LieDgen{3} \circ (1,1,1)
            \big)^{(123)}  \\
          &\quad- \LieDgen{2} \circ \big(
            1 \circ (1) ,
            1 \circ (\LieDgen{3;2})
              - \LieDgen{2;1} \circ (\LieDgen{2},1)
              - \LieDgen{2} \circ (1,\LieDgen{2;1})
              - \LieDgen{2;1} \circ (1,\LieDgen{2})^{(12)}
              + \LieDgen{3;2} \circ (1,1,1) \big)  \\
          &\quad+ \LieDgen{2} \circ \big(
            1 \circ (1) ,
            1 \circ (\LieDgen{3;2})
              - \LieDgen{2;1} \circ (\LieDgen{2},1)
              - \LieDgen{2} \circ (1,\LieDgen{2;1})
              - \LieDgen{2;1} \circ (1,\LieDgen{2})^{(12)}
              + \LieDgen{3;2} \circ (1,1,1) \big)^{(12)}  \\
          &\quad+ \LieDgen{2;1} \circ \big(
            1 \circ (\LieDgen{3})
              - \LieDgen{2} \circ (\LieDgen{2},1)
              + \LieDgen{2} \circ (1,\LieDgen{2})
              - \LieDgen{2} \circ (1,\LieDgen{2})^{(12)}
              + \LieDgen{3} \circ (1,1,1) ,
            1 \circ (1) \big)  \\
          &\quad- \LieDgen{2} \circ \big(
            1 \circ (\LieDgen{2}) + \LieDgen{2} \circ (1,1) ,
            1 \circ (\LieDgen{2;1}) + \LieDgen{2;1} \circ (1,1) \big)
          - \LieDgen{2;1} \circ \big(
            1 \circ (\LieDgen{2}) + \LieDgen{2} \circ (1,1) ,  \\
            &\quad1 \circ (\LieDgen{2}) + \LieDgen{2} \circ (1,1) \big)^{(132)}
          + \LieDgen{2;1} \circ \big(
            1 \circ (\LieDgen{2}) + \LieDgen{2} \circ (1,1) ,
            1 \circ (\LieDgen{2}) + \LieDgen{2} \circ (1,1) \big)^{(23)}  \\
          &\quad+ \LieDgen{3;2} \circ \big(
            1 \circ (\LieDgen{2}) + \LieDgen{2} \circ (1,1) ,
            1 \circ (1) ,
            1 \circ (1) \big)
          - \LieDgen{3;2} \circ \big(
            1 \circ (1) ,
            1 \circ (1) ,
            1 \circ (\LieDgen{2}) + \LieDgen{2} \circ (1,1) \big)^{(23)}  \\
          &\quad+ \LieDgen{3;2} \circ \big(
            1 \circ (1) ,
            1 \circ (1) ,
            1 \circ (\LieDgen{2}) + \LieDgen{2} \circ (1,1) \big)^{(132)}
          + \LieDgen{3} \circ \big(
            1 \circ (1) ,
            1 \circ (1) ,
            1 \circ (\LieDgen{2;1}) + \LieDgen{2;1} \circ (1,1) \big)  \\
          &\quad- \LieDgen{3;2} \circ \big(
            1 \circ (1) ,
            1 \circ (\LieDgen{2}) + \LieDgen{2} \circ (1,1) ,
            1 \circ (1) \big)
          + \LieDgen{3;2} \circ \big(
            1 \circ (1) ,
            1 \circ (\LieDgen{2}) + \LieDgen{2} \circ (1,1) ,
            1 \circ (1) \big)^{(12)}  \\
          &\quad+ \LieDgen{4;3} \circ \big(
            1 \circ (1),
            1 \circ (1),
            1 \circ (1),
            1 \circ (1) \big)  ,
      \end{align*}
      and again we see that both sides agree.  This concludes the proof of \Cref{lm:LieD:coopd}.

    \subsection{Proof of Lemma \ref{lm:htt:str}} \label{A:comp:htt}
      We verify that the structure maps defined in \cref{eq:htt:l21,eq:htt:l211,eq:htt:l31,eq:htt:l32} satisfy
      \cref{eq:ELie3:l21,eq:ELie3:l211,eq:ELie3:l31,eq:ELie3:l32,eq:ELie3:l2111,eq:ELie3:l311,eq:ELie3:l312,eq:ELie3:l322,eq:ELie3:l41,eq:ELie3:l42,eq:ELie3:l43}
      and hence form a weak Lie 3-algebra:
      \begin{align*}
        \dEnd(\LieDstr{2;1}[\prime]) &= p \circ (\dEnd\LieDstr{2;1}) \circ (i,i)
          = p \circ \big( \LieDstr{2} + \LieDstr{2}[(12)] \big) \circ (i,i)
          = \LieDstr{2}[\prime] + \LieDstr{2}[\prime(12)] ,  \\
        \dEnd(\LieDstr{2;1,1}[\prime]) &= p \circ (\dEnd\LieDstr{2;1,1}) \circ (i,i)
          = p \circ \big( \LieDstr{2;1} - \LieDstr{2;1}[(12)] \big) \circ (i,i)
          = \LieDstr{2;1}[\prime] - \LieDstr{2;1}[\prime(12)] ,  \\
        \dEnd(\LieDstr{3;1}[\prime]) &= p \circ \big( \dEnd\LieDstr{3;1} \big) \circ (i,i,i)
          - p \circ \dEnd\big( \LieDstr{2}\circ_1(h\circ\LieDstr{2;1}) \big) \circ (i,i,i)  \\
        &= p \circ \big( \LieDstr{3} + \LieDstr{3}[(12)] + \LieDstr{2}\circ_1\LieDstr{2;1} \big) \circ (i,i,i)  \\
          &\quad- p \circ \big(
              \LieDstr{2}\circ_1(\dEnd h\circ\LieDstr{2;1})
              - \LieDstr{2}\circ_1(h\circ\dEnd\LieDstr{2;1})
            \big) \circ (i,i,i)  \\
        &= p \circ \big(
              \LieDstr{3} + \LieDstr{2}\circ_1(h\circ\LieDstr{2})
            \big) \circ (i,i,i)
          + p \circ \big(
              \LieDstr{3}[(12)] + (\LieDstr{2}\circ_1(h\circ\LieDstr{2}))^{(12)}
            \big) \circ (i,i,i)  \\
          &\quad+ p \circ \big( \LieDstr{2}\circ_1((\id - \dEnd h)\LieDstr{2;1}) \big) \circ (i,i,i)  \\
        &= \LieDstr{3}[\prime] + {\LieDstr{3}[\prime]}^{(12)} + \LieDstr{2}[\prime]\circ_1\LieDstr{2;1}[\prime] \\
        \dEnd(\LieDstr{3;2}[\prime]) &= p \circ \big( \dEnd\LieDstr{3;2} \big) \circ (i,i,i)
          + p \circ \dEnd\big(
              \LieDstr{2;1}\circ_1(h\circ\LieDstr{2})
              + \LieDstr{2}\circ_2(h\circ\LieDstr{2;1})
              + (\LieDstr{2;1}\circ_2(h\circ\LieDstr{2}))^{(12)}
            \big) \circ (i,i,i)  \\
        &= p \circ \big( \LieDstr{3} + \LieDstr{3}[(23)]
            + \LieDstr{2;1}\circ_1(h\circ\LieDstr{2})
            - \LieDstr{2}\circ_2(h\circ\LieDstr{2;1})
            + (\LieDstr{2;1}\circ_2(h\circ\LieDstr{2}))^{(12)}
          \big) \circ (i,i,i)  \\
          &\quad+ p \circ \left(\begin{aligned}
              &\dEnd\LieDstr{2;1}\circ_1(h\circ\LieDstr{2})
              + \LieDstr{2}\circ_2(\dEnd h\circ\LieDstr{2;1})
              + (\dEnd\LieDstr{2;1}\circ_2(h\circ\LieDstr{2}))^{(12)}  \\
              &- \LieDstr{2;1}\circ_1(\dEnd h\circ\LieDstr{2})
              - \LieDstr{2}\circ_2(h\circ\dEnd\LieDstr{2;1})
              - (\LieDstr{2;1}\circ_2(\dEnd h\circ\LieDstr{2}))^{(12)}
            \end{aligned}\right) \circ (i,i,i)  \\
        &= p \circ \big(
              \LieDstr{3}
              + \LieDstr{2}\circ_1(h\circ\LieDstr{2})
              - \LieDstr{2}\circ_2(h\circ\LieDstr{2})
              + (\LieDstr{2}\circ_2(h\circ\LieDstr{2}))^{(12)}
            \big) \circ (i,i,i)  \\
          &\quad+ p \circ \big(
              \LieDstr{3}[(23)]
              + \LieDstr{2}[(12)]\circ_1(h\circ\LieDstr{2})
              - \LieDstr{2}\circ_2(h\circ\LieDstr{2}[(12)])
              + (\LieDstr{2}[(12)]\circ_2(h\circ\LieDstr{2}))^{(12)}
            \big) \circ (i,i,i)  \\
          &\quad+ p \circ \big(
              \LieDstr{2;1}\circ_1((\id - \dEnd h)\circ\LieDstr{2})
              - \LieDstr{2}\circ_2((\id - \dEnd h)\circ\LieDstr{2;1})
              + (\LieDstr{2;1}\circ_2((\id - \dEnd h)\circ\LieDstr{2}))^{(12)}
            \big) \circ (i,i,i)  \\
        &= \LieDstr{3}[\prime] + {\LieDstr{3}[\prime]}^{(23)}
          + \LieDstr{2;1}[\prime]\circ_1\LieDstr{2}[\prime]
          - \LieDstr{2}[\prime]\circ_2\LieDstr{2;1}[\prime]
          + (\LieDstr{2;1}[\prime]\circ_2\LieDstr{2}[\prime])^{(12)} ,  \\
      \LieDstr{2;1,1}[\prime] + \LieDstr{2;1,1}[\prime(12)]
        &= p \circ \big( \LieDstr{2;1,1} + \LieDstr{2;1,1}[(12)] \big) \circ (i,i) = 0 ,  \\
        \LieDstr{3;1}[\prime] - \LieDstr{3;1}[\prime(12)]
        &= p \circ \big( \LieDstr{3;1} - \LieDstr{3;1}[(12)] \big) \circ (i,i,i)
          - p \circ \big(
              \LieDstr{2}\circ_1(h\circ\LieDstr{2;1}) - (\LieDstr{2}\circ_1(h\circ\LieDstr{2;1}))^{(12)}
            \big) \circ (i,i,i)  \\
        &= p \circ (\LieDstr{2}\circ_1\LieDstr{2;1,1}) \circ (i,i,i)
          - p \circ ( \LieDstr{2}\circ_1(h\circ\dEnd\LieDstr{2;1,1}) ) \circ (i,i,i)  \\
        &= p \circ \big(\LieDstr{2}\circ_1((\id - \dEnd h) \LieDstr{2;1,1})\big) \circ (i,i,i)  \\
        &= \LieDstr{2}[\prime]\circ_1\LieDstr{2;1,1}[\prime] ,  \\
      \hspace{2em}&\hspace{-2em}
        \LieDstr{3;1}[\prime]
        - \LieDstr{3;2}[\prime(12)]
        + \LieDstr{3;1}[\prime(132)]
        - \LieDstr{3;2}[\prime]
        + \LieDstr{3;1}[\prime(23)]
        - \LieDstr{3;2}[\prime(123)]  \\
        &= p \circ \big( 
            \LieDstr{3;1}
            - \LieDstr{3;2}[(12)]
            + \LieDstr{3;1}[(132)]
            - \LieDstr{3;2}
            + \LieDstr{3;1}[(23)]
            - \LieDstr{3;2}[(123)]
          \big) \circ (i,i,i)  \\
          &\quad- p \circ \big(\LieDstr{2}\circ_1(h\circ\LieDstr{2;1})\big)^{1+(23)+(132)} \circ (i,i,i)  \\
          &\quad- p \circ \big(
              \LieDstr{2;1}\circ_1(h\circ\LieDstr{2})
              + \LieDstr{2}\circ_2(h\circ\LieDstr{2;1})
              + (\LieDstr{2;1}\circ_2(h\circ\LieDstr{2}))^{(12)}
            \big)^{1+(12)+(123)} \circ (i,i,i)  \\
        &= p \circ \big(
            \LieDstr{2;1}\circ_1\LieDstr{2;1}
            + (\LieDstr{2;1}\circ_2\LieDstr{2;1})^{1+(12)}
            + (\LieDstr{2;1,1}\circ_2\LieDstr{2})^{(132)}
          \big) \circ (i,i,i)  \\
          &\quad- p \circ \left(\begin{aligned}
              &\dEnd\LieDstr{2;1}\circ_1(h\circ\LieDstr{2;1})
              + \LieDstr{2;1}\circ_1(h\circ\dEnd\LieDstr{2;1})  \\
              &+ \dEnd\LieDstr{2;1}\circ_2(h\circ\LieDstr{2;1})
              + \LieDstr{2;1}\circ_2(h\circ\dEnd\LieDstr{2;1})  \\
              &+ (\dEnd\LieDstr{2;1}\circ_2(h\circ\LieDstr{2;1}))^{(12)}
              + (\LieDstr{2;1}\circ_2(h\circ\dEnd\LieDstr{2;1}))^{(12)}  \\
              &+ (\dEnd\LieDstr{2;1,1}\circ_2(h\circ\LieDstr{2}))^{(132)}
            \end{aligned}\right)  \\
        &= \LieDstr{2;1}[\prime]\circ_1\LieDstr{2;1}[\prime]
          + (\LieDstr{2;1}[\prime]\circ_2\LieDstr{2;1}[\prime])^{1+(12)}
          + (\LieDstr{2;1,1}[\prime]\circ_2\LieDstr{2}[\prime])^{(132)} ,  \\
        \LieDstr{3;2}[\prime] - \LieDstr{3;2}[\prime(23)]
        &= p \circ \big( \LieDstr{3;2} - \LieDstr{3;2}[(23)] \big) \circ (i,i,i)  \\
          &\quad+ p \circ \left(\begin{aligned}
              &\LieDstr{2;1}\circ_1(h\circ\LieDstr{2})
              + \LieDstr{2}\circ_2(h\circ\LieDstr{2;1})
              + (\LieDstr{2;1}\circ_2(h\circ\LieDstr{2}))^{(12)}  \\
              &- (\LieDstr{2;1}\circ_1(h\circ\LieDstr{2}))^{(23)}
              - (\LieDstr{2}\circ_2(h\circ\LieDstr{2;1}))^{(23)}
              - (\LieDstr{2;1}\circ_2(h\circ\LieDstr{2}))^{(123)}
            \end{aligned}\right) \circ (i,i,i)  \\
        &= p \circ \big(
            \LieDstr{2;1,1}\circ_1\LieDstr{2}
            - \LieDstr{2}\circ_2\LieDstr{2;1,1}
            + (\LieDstr{2}\circ_2\LieDstr{2;1,1})^{(12)}
          \big) \circ (i,i,i)  \\
          &\quad- p \circ \big(
              \dEnd\LieDstr{2;1,1}\circ_1(h\circ\LieDstr{2})
              + \LieDstr{2}\circ_2(h\circ\dEnd\LieDstr{2;1,1})
              + ( \dEnd\LieDstr{2;1,1}\circ_2(h\circ\LieDstr{2}) )^{(12)}
            \big) \circ (i,i,i)  \\
        &= p \circ \big(
            \LieDstr{2;1,1}\circ_1\LieDstr{2}
            - \LieDstr{2}\circ_2\LieDstr{2;1,1}
            + (\LieDstr{2;1,1}\circ_2\LieDstr{2})^{(12)}
          \big) \circ (i,i,i)  \\
        &= p \circ \big(
            \LieDstr{2;1,1}\circ_1((\id - \dEnd h)\LieDstr{2})
            - \LieDstr{2}\circ_2((\id - \dEnd h)\LieDstr{2;1,1})
            + (\LieDstr{2;1,1}\circ_2(\id - \dEnd h)\LieDstr{2})^{(12)}
          \big) \circ (i,i,i)  \\
        &= \LieDstr{2;1,1}[\prime]\circ_1\LieDstr{2}[\prime]
          - \LieDstr{2}[\prime]\circ_2\LieDstr{2;1,1}[\prime]
          + (\LieDstr{2;1,1}[\prime]\circ_2\LieDstr{2}[\prime])^{(12)} ,  \\
        \LieDstr{4}[\prime] + \LieDstr{4}[\prime(12)]
        &= p \circ \big( \LieDstr{4} + \LieDstr{4}[(12)] \big) \circ (i,i,i,i)
          - p \circ \left(\begin{aligned}
              &\LieDstr{2}\circ_1(h\circ\LieDstr{3})
              + (\LieDstr{2}\circ_2(h\circ\LieDstr{3}))^{(123)}  \\
              &+ \LieDstr{3}\circ_1(h\circ\LieDstr{2})
              + \LieDstr{3}\circ_3(h\circ\LieDstr{2})
            \end{aligned}\right)^{\mathrlap{1+(12)}} \circ (i,i,i,i)  \\
          &\quad- p \circ \left(\begin{aligned}
              \LieDstr{2}\circ_1(h\circ\LieDstr{2})\circ_1(h\circ\LieDstr{2})
              + \LieDstr{2}\circ(h\circ\LieDstr{2},h\circ\LieDstr{2})  \\
              {}+ (\LieDstr{2}\circ_1(h\circ\LieDstr{2})\circ_1(h\circ\LieDstr{2}))^{(123)}
            \end{aligned}\right)^{\mathrlap{1+(12)}} \circ (i,i,i,i)  \\
        &= p \circ \big(
            \LieDstr{2}\circ_1\LieDstr{3;1}
            + (\LieDstr{2}\circ_2\LieDstr{3;1})^{(123)}
            + \LieDstr{3}\circ_1\LieDstr{2;1}
            - \LieDstr{3;1}\circ_3\LieDstr{2}
          \big) \circ (i,i,i,i)  \\
          &\quad- p \circ \left(\begin{aligned}
              &\LieDstr{2} \circ_1 \big(
                  h \circ \big(\LieDstr{3} + \LieDstr{3}[(12)] + \LieDstr{2} \circ_1 \LieDstr{2;1}\big)
                \big)  \\
              &+ \big(
                  \LieDstr{2} \circ_2 \big(
                      h \circ \big(\LieDstr{3} + \LieDstr{3}[(12)] + \LieDstr{2} \circ_1 \LieDstr{2;1}\big)
                    \big)
                \big)^{(123)}  \\
              &- \big(\LieDstr{2} \circ_1 \LieDstr{2} - (\LieDstr{2} \circ_2 \LieDstr{2})^{1-(12)}\big)
                \circ_1 (h \circ \LieDstr{2;1})
              + \LieDstr{3} \circ_1 \big(h \circ \big(\LieDstr{2} + \LieDstr{2}[(12)]\big)\big)  \\
              &+ \big(\LieDstr{3} + \LieDstr{3}[(12)] + \LieDstr{2} \circ_1 \LieDstr{2;1}\big)
                \circ_3 (h \circ \LieDstr{2})
            \end{aligned}\right) \circ (i,i,i,i)  \\
          &\quad- p \circ \left(\begin{aligned}
              & \LieDstr{2}\circ_1\LieDstr{2}\circ_1(h\circ\LieDstr{2;1})
              - \LieDstr{2}\circ_1(\dEnd h\circ\LieDstr{2})\circ_1(h\circ\LieDstr{2;1})  \\
              &+ \big(\LieDstr{2}\circ_2\LieDstr{2}\circ_2(h\circ\LieDstr{2;1})\big)^{(123)}
              - \big(\LieDstr{2}\circ_2(\dEnd h\circ\LieDstr{2})\circ_2(h\circ\LieDstr{2;1})\big)^{(123)}  \\
              &- \LieDstr{2}\circ_1(h\circ\LieDstr{2})\circ_1\LieDstr{2;1}
              + \LieDstr{2}\circ_1(h\circ\LieDstr{2})\circ_1(\dEnd h\circ\LieDstr{2;1})  \\
              &+ (\LieDstr{2}\circ_2(h\circ\LieDstr{2}))^{1-(12)}\circ_1\LieDstr{2;1}
              - (\LieDstr{2}\circ_2(h\circ\LieDstr{2}))^{1-(12)}\circ_1(\dEnd h\circ\LieDstr{2;1})  \\
              &- \LieDstr{2}\circ_1(h\circ\LieDstr{2;1})\circ_3\LieDstr{2}
              + \LieDstr{2}\circ_1(h\circ\LieDstr{2;1})\circ_3(\dEnd h\circ\LieDstr{2})
            \end{aligned}\right) \circ (i,i,i,i)  \\
        &= p \circ \left(\begin{aligned}
            &\LieDstr{2}\circ_1((\id-\dEnd h)\circ\LieDstr{3;1})
            - \LieDstr{2}\circ_1((\id-\dEnd h)\circ\LieDstr{2})\circ_1(h\circ\LieDstr{2;1})  \\
            &+ (\LieDstr{2}\circ_2((\id-\dEnd h)\circ\LieDstr{3;1}))^{(123)}
            - \big(\LieDstr{2}\circ_2((\id-\dEnd h)\circ\LieDstr{2})\circ_2(h\circ\LieDstr{2;1})\big)^{(123)}  \\
            &+ \LieDstr{3}\circ_1((\id-\dEnd h)\circ\LieDstr{2;1})
            + \LieDstr{2}\circ_1(h\circ\LieDstr{2})\circ_1((\id-\dEnd h)\circ\LieDstr{2;1})  \\
            &- (\LieDstr{2}\circ_2(h\circ\LieDstr{2}))^{1-(12)}\circ_1((\id-\dEnd h)\circ\LieDstr{2;1})  \\
            &- \LieDstr{3;1}\circ_3((\id-\dEnd h)\circ\LieDstr{2})
            + \LieDstr{2}\circ_1(h\circ\LieDstr{2;1})\circ_3((\id-\dEnd h)\circ\LieDstr{2})
          \end{aligned}\right) \circ (i,i,i,i)  \\
        &= p \circ \left(\begin{aligned}
            &\LieDstr{2} \circ_1 (i \circ p
              \circ (\LieDstr{3;1} - \LieDstr{2}\circ_1(h\circ\LieDstr{2;1})) )  \\
            &+ (\LieDstr{2} \circ_2 (i \circ p
              \circ (\LieDstr{3;1} - \LieDstr{2}\circ_1(h\circ\LieDstr{2;1})) ))^{(123)}  \\
            &+ \big(\LieDstr{3} + \LieDstr{2}\circ_1(h\circ\LieDstr{2})
                - (\LieDstr{2}\circ_2(h\circ\LieDstr{2}))^{1-(12)}
              \big) \circ_1 (i \circ p \circ \LieDstr{2;1})  \\
            &- (\LieDstr{3;1} - \LieDstr{2}\circ_1(h\circ\LieDstr{2;1})) \circ_3 (i \circ p \circ \LieDstr{2})
          \end{aligned}\right) \circ (i,i,i,i)  \\
        &= \LieDstr{2}[\prime]\circ_1\LieDstr{3;1}[\prime]
            + (\LieDstr{2}[\prime]\circ_2\LieDstr{3;1}[\prime])^{(123)}
            + \LieDstr{3}[\prime]\circ_1\LieDstr{2;1}[\prime]
            - \LieDstr{3;1}[\prime]\circ_3\LieDstr{2}[\prime] ,  \\
        \LieDstr{4}[\prime] + \LieDstr{4}[\prime(23)]
        &= p \circ \big( \LieDstr{4} + \LieDstr{4}[(23)] \big) \circ (i,i,i,i)
          - p \circ \left(\begin{aligned}
              &\LieDstr{2}\circ_1(h\circ\LieDstr{3})
              + \LieDstr{2}\circ_2(h\circ\LieDstr{3})
              + \LieDstr{3}\circ_1(h\circ\LieDstr{2})  \\
              &- (\LieDstr{3}\circ_2(h\circ\LieDstr{2}))^{1-(12)}
              + (\LieDstr{3}\circ_3(h\circ\LieDstr{2}))^{(132)}
            \end{aligned}\right)^{\mathrlap{1+(23)}} \circ (i,i,i,i)  \\
          &\quad- p \circ \left(\begin{aligned}
              &\LieDstr{2}\circ_1(h\circ\LieDstr{2})\circ_1(h\circ\LieDstr{2})
              - (\LieDstr{2}\circ_1(h\circ\LieDstr{2})\circ_2(h\circ\LieDstr{2}))^{1-(12)}  \\
              &+ \LieDstr{2}\circ(h\circ\LieDstr{2},h\circ\LieDstr{2})^{(132)}
              + \LieDstr{2}\circ_2(h\circ\LieDstr{2})\circ_2(h\circ\LieDstr{2})
            \end{aligned}\right)^{\mathrlap{1+(23)}} \circ (i,i,i,i)  \\
        &= p \circ \big(
            \LieDstr{2}\circ_1\LieDstr{3;2}
            + \LieDstr{2}\circ_2\LieDstr{3;1}
            - \LieDstr{3;1}\circ_1\LieDstr{2}
            - \LieDstr{3}\circ_2\LieDstr{2;1}
            - (\LieDstr{3;1}\circ_2\LieDstr{2})^{(12)}
            - (\LieDstr{3;1}\circ_3\LieDstr{2})^{(132)}
          \big) \circ (i,i,i,i)  \\
          &\quad- p \circ \left(\begin{aligned}
              &\LieDstr{2} \circ_1 \big(
                  h \circ \big(
                      \LieDstr{3} + \LieDstr{3}[(23)]
                      + \LieDstr{2;1} \circ_1 \LieDstr{2}
                      - \LieDstr{2} \circ_2 \LieDstr{2;1}
                      + (\LieDstr{2;1} \circ_2 \LieDstr{2})^{(12)}
                    \big)
                \big)  \\
              &+ \LieDstr{2} \circ_2 \big(
                  h \circ \big(\LieDstr{3} + \LieDstr{3}[(12)] + \LieDstr{2} \circ_1 \LieDstr{2;1}\big)
                \big)
              + \big(
                  \LieDstr{3} + \LieDstr{3}[(12)] + \LieDstr{2} \circ_1 \LieDstr{2;1}
                \big) \circ_1 (h \circ \LieDstr{2})  \\
              &+ \big(\LieDstr{2} \circ_1 \LieDstr{2} - (\LieDstr{2} \circ_2 \LieDstr{2})^{1-(12)}\big)
                \circ_2 (h \circ \LieDstr{2;1})
              - \LieDstr{3} \circ_2 \big(h \circ \big(\LieDstr{2} + \LieDstr{2}[(12)]\big)\big)  \\
              &+ \big(\big(\LieDstr{3} + \LieDstr{3}[(12)] + \LieDstr{2} \circ_1 \LieDstr{2;1}\big)
                \circ_2 (h \circ \LieDstr{2}) \big)^{(12)}  \\
              &+ \big(\big(\LieDstr{3} + \LieDstr{3}[(12)] + \LieDstr{2} \circ_1 \LieDstr{2;1}\big)
                \circ_3 (h \circ \LieDstr{2}) \big)^{(132)}
            \end{aligned}\right) \circ (i,i,i,i)  \\
          &\quad+ p \circ \left(\begin{aligned}
              & \LieDstr{2} \circ_1 \big(
                  \LieDstr{2;1} \circ_1 (h\circ\LieDstr{2})
                  + \LieDstr{2} \circ_2 (h\circ\LieDstr{2;1})
                  + (\LieDstr{2;1} \circ_2 (h\circ\LieDstr{2}))^{(12)}
                \big)  \\
              &- \LieDstr{2} \circ_1 \big( \dEnd h \circ \big(
                  \LieDstr{2;1} \circ_1 (h\circ\LieDstr{2})
                  + \LieDstr{2} \circ_2 (h\circ\LieDstr{2;1})
                  + (\LieDstr{2;1} \circ_2 (h\circ\LieDstr{2}))^{(12)}
                \big) \big)  \\
              &- \LieDstr{2}\circ_2(\LieDstr{2}\circ_1(h\circ\LieDstr{2;1}))
              + \LieDstr{2}\circ_2(\dEnd h\circ(\LieDstr{2}\circ_1(h\circ\LieDstr{2;1})))  \\
              &+ (\LieDstr{2}\circ_1(h\circ\LieDstr{2;1})) \circ_1 \LieDstr{2}
              - (\LieDstr{2}\circ_1(h\circ\LieDstr{2;1})) \circ_1 (\dEnd h\circ\LieDstr{2})  \\
              &- \big( \LieDstr{2}\circ_1(h\circ\LieDstr{2})
                - (\LieDstr{2}\circ_2(h\circ\LieDstr{2}))^{1-(12)} \big) \circ_2 \LieDstr{2;1}  \\
              &+ \big( \LieDstr{2}\circ_1(h\circ\LieDstr{2})
                - (\LieDstr{2}\circ_2(h\circ\LieDstr{2}))^{1-(12)} \big) \circ_2 (\dEnd h \circ \LieDstr{2;1})  \\
              &+ ( (\LieDstr{2}\circ_1(h\circ\LieDstr{2;1})) \circ_2 \LieDstr{2} )^{(12)}
              - ( (\LieDstr{2}\circ_1(h\circ\LieDstr{2;1})) \circ_2 (\dEnd h \circ \LieDstr{2}) )^{(12)}  \\
              &+ ((\LieDstr{2}\circ_1(h\circ\LieDstr{2;1}))\circ_3\LieDstr{2})^{(132)}
              - ((\LieDstr{2}\circ_1(h\circ\LieDstr{2;1}))\circ_3(\dEnd h \circ \LieDstr{2}))^{(132)}
            \end{aligned}\right) \circ (i,i,i,i)  \\
        &= p \circ \left(\begin{aligned}
            &\LieDstr{2} \circ_1 ((\id-\dEnd h) \circ \LieDstr{3;2})  \\
              &\quad+ \LieDstr{2} \circ_1 \left((\id-\dEnd h) \circ \left(\begin{aligned}
                  \LieDstr{2;1}\circ_1(h\circ\LieDstr{2})
                  &+ \LieDstr{2}\circ_2(h\circ\LieDstr{2;1})  \\
                  &+ (\LieDstr{2;1}\circ_2(h\circ\LieDstr{2}))^{(12)}
                \end{aligned}\right) \right)  \\
            &+ \LieDstr{2} \circ_2 ((\id-\dEnd h) \circ \LieDstr{3;1})
              - \LieDstr{2} \circ_2 ((\id-\dEnd h) \circ ( \LieDstr{2}\circ_1(h\circ\LieDstr{2;1})) )  \\
            &- \LieDstr{3;1} \circ_1 ((\id-\dEnd h) \circ \LieDstr{2})
              + ( \LieDstr{2}\circ_1(h\circ\LieDstr{2;1})) \circ_1 ((\id-\dEnd h) \circ \LieDstr{2})  \\
            &- \LieDstr{3} \circ_2 ((\id-\dEnd h) \circ \LieDstr{2;1})  \\
              &\quad- \big( \LieDstr{2} \circ_1 (h\circ\LieDstr{2})
                - (\LieDstr{2} \circ_2 (h\circ\LieDstr{2}))^{1-(12)} \big) \circ_2
                ((\id-\dEnd h) \circ \LieDstr{2;1})  \\
            &- (\LieDstr{3;1} \circ_2 ((\id-\dEnd h) \circ \LieDstr{2}))^{(12)}
              + ( ( \LieDstr{2}\circ_1(h\circ\LieDstr{2;1})) \circ_2 ((\id-\dEnd h) \circ \LieDstr{2}) )^{(12)}  \\
            &- (\LieDstr{3;1} \circ_3 ((\id-\dEnd h) \circ \LieDstr{2}))^{(132)}
              + ( ( \LieDstr{2}\circ_1(h\circ\LieDstr{2;1})) \circ_3 ((\id-\dEnd h) \circ \LieDstr{2}) )^{(132)}
          \end{aligned}\right) \circ (i,i,i,i)  \\
        &= \LieDstr{2}[\prime]\circ_1\LieDstr{3;2}[\prime]
            + \LieDstr{2}[\prime]\circ_2\LieDstr{3;1}[\prime]
            - \LieDstr{3;1}[\prime]\circ_1\LieDstr{2}[\prime]
            - \LieDstr{3}[\prime]\circ_2\LieDstr{2;1}[\prime]
            - (\LieDstr{3;1}[\prime]\circ_2\LieDstr{2}[\prime])^{(12)}
            - (\LieDstr{3;1}[\prime]\circ_3\LieDstr{2}[\prime])^{(132)} ,  \\
        \LieDstr{4}[\prime] + \LieDstr{4}[\prime(34)]
        &= p \circ \big( \LieDstr{4} + \LieDstr{4}[(34)] \big) \circ (i,i,i,i)  \\
          &\quad- p \circ \left(\begin{aligned}
              &\LieDstr{2}\circ_1(h\circ\LieDstr{3})
              + (\LieDstr{2}\circ_2(h\circ\LieDstr{3}))^{1-(12)+(123)}
              + \LieDstr{3}\circ_1(h\circ\LieDstr{2})  \\
              &- (\LieDstr{3}\circ_2(h\circ\LieDstr{2}))^{1-(12)}
              + (\LieDstr{3}\circ_3(h\circ\LieDstr{2}))^{1-(23)+(132)}
            \end{aligned}\right)^{\mathrlap{1+(34)}} \circ (i,i,i,i)  \\
          &\quad- p \circ \left(\begin{aligned}
              &\LieDstr{2}\circ_1(h\circ\LieDstr{2})\circ_1(h\circ\LieDstr{2})
              - (\LieDstr{2}\circ_1(h\circ\LieDstr{2})\circ_2(h\circ\LieDstr{2}))^{1-(12)}  \\
              &+ \LieDstr{2}\circ(h\circ\LieDstr{2},h\circ\LieDstr{2})^{1-(23)+(132)}
              + (\LieDstr{2}\circ_2(h\circ\LieDstr{2})\circ_2(h\circ\LieDstr{2}))^{1-(12)+(123)}  \\
              &- (\LieDstr{2}\circ_2(h\circ\LieDstr{2})\circ_3(h\circ\LieDstr{2}))^{1-(12)+(123)-(23)+(132)-(13)}
            \end{aligned}\right)^{\mathrlap{1+(34)}} \circ (i,i,i,i)  \\
        &= p \circ \left(\begin{aligned}
            &\LieDstr{2;1}\circ_1\LieDstr{3}
            + (\LieDstr{2}\circ_2\LieDstr{3;2})^{1-(12)}
            + (\LieDstr{2;1}\circ_2\LieDstr{3})^{(123)}
            - \LieDstr{3;2}\circ_1\LieDstr{2}  \\
            &+ (\LieDstr{3;2}\circ_2\LieDstr{2})^{1-(12)}
            + \LieDstr{3}\circ_3\LieDstr{2;1}
            + (\LieDstr{3;2}\circ_3\LieDstr{2})^{(23)-(132)}
          \end{aligned}\right) \circ (i,i,i,i)  \\
          &\quad- p \circ \left(\begin{aligned}
              & \big( \LieDstr{2} + \LieDstr{2}[(12)] \big) \circ_1 (h \circ \LieDstr{3})
              + \LieDstr{2;1} \circ_1 \big(h \circ \big(
                  - \LieDstr{2} \circ_1 \LieDstr{2}
                  + (\LieDstr{2} \circ_2 \LieDstr{2})^{1-(12)}
                \big) \big)  \\
              &+ \big( \LieDstr{2} \circ_2 \big(
                  h \circ \big(
                      \LieDstr{3} + \LieDstr{3}[(23)]
                      + \LieDstr{2;1} \circ_1 \LieDstr{2}
                      - \LieDstr{2} \circ_2 \LieDstr{2;1}
                      + (\LieDstr{2;1} \circ_2 \LieDstr{2})^{(12)}
                    \big)
                \big) \big)^{1-(12)}  \\
              &+ \big( \big( \LieDstr{2} + \LieDstr{2}[(12)] \big) \circ_2 (h \circ \LieDstr{3}) \big)^{(123)}  \\
              &\quad+ \big( \LieDstr{2;1} \circ_2 \big( h \circ \big(
                  - \LieDstr{2} \circ_1 \LieDstr{2}
                  + (\LieDstr{2} \circ_2 \LieDstr{2})^{1-(12)}
                \big) \big) \big)^{(123)}  \\
              &+ \big( \LieDstr{3} + \LieDstr{3}[(23)]
                  + \LieDstr{2;1} \circ_1 \LieDstr{2}
                  - \LieDstr{2} \circ_2 \LieDstr{2;1}
                  + (\LieDstr{2;1} \circ_2 \LieDstr{2})^{(12)}
                \big) \circ_1 (h \circ \LieDstr{2})  \\
              &- \big( \big( \LieDstr{3} + \LieDstr{3}[(23)]
                  + \LieDstr{2;1} \circ_1 \LieDstr{2}
                  - \LieDstr{2} \circ_2 \LieDstr{2;1}
                  + (\LieDstr{2;1} \circ_2 \LieDstr{2})^{(12)}
                \big) \circ_2 (h \circ \LieDstr{2}) \big)^{1-(12)}  \\
              &+ \big( - \LieDstr{2} \circ_1 \LieDstr{2} + (\LieDstr{2} \circ_2 \LieDstr{2})^{1-(12)} \big)
                \circ_3 (h \circ \LieDstr{2;1})
              + \LieDstr{3} \circ_3 \big( h \circ \big( \LieDstr{2} + \LieDstr{2}[(12)] \big) \big)  \\
              &- \left( \left(\begin{aligned}
                  \LieDstr{3} + \LieDstr{3}[(23)]
                  + \LieDstr{2;1} \circ_1 \LieDstr{2}
                  &- \LieDstr{2} \circ_2 \LieDstr{2;1}  \\
                  &+ (\LieDstr{2;1} \circ_2 \LieDstr{2})^{(12)}
                \end{aligned}\right) \circ_3 (h \circ \LieDstr{2}) \right)^{(23)-(132)}  \\
            \end{aligned}\right) \circ (i,i,i,i)  \\
          &\quad+ p \circ \left(\begin{aligned}
              &- \LieDstr{2;1} \circ_1 \big(
                  - \LieDstr{2} \circ_1 (h \circ \LieDstr{2})
                  + (\LieDstr{2} \circ_2 (h \circ \LieDstr{2}))^{1-(12)}
                \big)  \\
              &\quad+ \LieDstr{2;1} \circ_1 \big( \dEnd h \circ \big(
                  - \LieDstr{2} \circ_1 (h \circ \LieDstr{2})
                  + (\LieDstr{2} \circ_2 (h \circ \LieDstr{2}))^{1-(12)}
                \big) \big)  \\
              &+ \big( \LieDstr{2} \circ_2 \big(
                  \LieDstr{2;1} \circ_1 (h\circ\LieDstr{2})
                  + \LieDstr{2} \circ_2 (h\circ\LieDstr{2;1})
                  + (\LieDstr{2;1} \circ_2 (h\circ\LieDstr{2}))^{(12)}
                \big) \big)^{1-(12)}  \\
              &\quad- \left( \LieDstr{2} \circ_2 \left( \dEnd h \circ \left(\begin{aligned}
                  \LieDstr{2;1} \circ_1 (h\circ\LieDstr{2})
                  &+ \LieDstr{2} \circ_2 (h\circ\LieDstr{2;1})  \\
                  &+ (\LieDstr{2;1} \circ_2 (h\circ\LieDstr{2}))^{(12)}
                \end{aligned}\right) \right) \right)^{1-(12)}  \\
              &- \big( \LieDstr{2;1} \circ_2 \big(
                  - \LieDstr{2} \circ_1 (h \circ \LieDstr{2})
                  + (\LieDstr{2} \circ_2 (h \circ \LieDstr{2}))^{1-(12)}
                \big) \big)^{(123)}  \\
              &\quad+ \big( \LieDstr{2;1} \circ_2 \big( \dEnd h \circ \big(
                  - \LieDstr{2} \circ_1 (h \circ \LieDstr{2})
                  + (\LieDstr{2} \circ_2 (h \circ \LieDstr{2}))^{1-(12)}
                \big) \big) \big)^{(123)}  \\
              &- \big( \LieDstr{2;1} \circ_1 (h\circ\LieDstr{2})
                  + \LieDstr{2} \circ_2 (h\circ\LieDstr{2;1})
                  + (\LieDstr{2;1} \circ_2 (h\circ\LieDstr{2}))^{(12)}
                \big) \circ_1 \LieDstr{2}  \\
              &\quad+ \big( \LieDstr{2;1} \circ_1 (h\circ\LieDstr{2})
                  + \LieDstr{2} \circ_2 (h\circ\LieDstr{2;1})
                  + (\LieDstr{2;1} \circ_2 (h\circ\LieDstr{2}))^{(12)}
                \big) \circ_1 (\dEnd h \circ \LieDstr{2})  \\
              &+ \big( \big( \LieDstr{2;1} \circ_1 (h\circ\LieDstr{2})
                  + \LieDstr{2} \circ_2 (h\circ\LieDstr{2;1})
                  + (\LieDstr{2;1} \circ_2 (h\circ\LieDstr{2}))^{(12)}
                \big) \circ_2 \LieDstr{2} \big)^{1-(12)}  \\
              &\quad- \left( \left(\begin{aligned}
                  \LieDstr{2;1} \circ_1 (h\circ\LieDstr{2})
                  &+ \LieDstr{2} \circ_2 (h\circ\LieDstr{2;1})  \\
                  &+ (\LieDstr{2;1} \circ_2 (h\circ\LieDstr{2}))^{(12)}
                \end{aligned}\right) \circ_2 (\dEnd h \circ \LieDstr{2}) \right)^{1-(12)}  \\
              &- \big( - \LieDstr{2}\circ_1(h\circ\LieDstr{2})
                + (\LieDstr{2}\circ_2(h\circ\LieDstr{2}))^{1-(12)} \big) \circ_3 \LieDstr{2;1}  \\
              &\quad+ \big( - \LieDstr{2}\circ_1(h\circ\LieDstr{2})
                + (\LieDstr{2}\circ_2(h\circ\LieDstr{2}))^{1-(12)} \big) \circ_3 (\dEnd h \circ \LieDstr{2;1})  \\
              &+ \big( \big( \LieDstr{2;1} \circ_1 (h\circ\LieDstr{2})
                  + \LieDstr{2} \circ_2 (h\circ\LieDstr{2;1})
                  + (\LieDstr{2;1} \circ_2 (h\circ\LieDstr{2}))^{(12)}
                \big) \circ_3 \LieDstr{2} \big)^{(23)-(132)}  \\
              &\quad- \left( \left(\begin{aligned}
                  \LieDstr{2;1} \circ_1 (h\circ\LieDstr{2})
                  &+ \LieDstr{2} \circ_2 (h\circ\LieDstr{2;1})  \\
                  &+ (\LieDstr{2;1} \circ_2 (h\circ\LieDstr{2}))^{(12)}
                \end{aligned}\right) \circ_3 (\dEnd h \circ \LieDstr{2}) \right)^{(23)-(132)}
            \end{aligned}\right) \circ (i,i,i,i)  \\
        &= p \circ \left(\begin{aligned}
            &\LieDstr{2;1} \circ_1 ((\id-\dEnd h) \circ \LieDstr{3})  \\
              &\quad- \LieDstr{2;1} \circ_1 \big((\id-\dEnd h) \circ \big(
                  - \LieDstr{2}\circ_1(h\circ\LieDstr{2})
                  + (\LieDstr{2}\circ_2(h\circ\LieDstr{2}))^{1-(12)}
                \big) \big)  \\
            &+ ( \LieDstr{2} \circ_2 ((\id-\dEnd h) \circ \LieDstr{3;2}))^{1-(12)}  \\
              &\quad+ \left( \LieDstr{2} \circ_2 \left((\id-\dEnd h) \circ \left(\begin{aligned}
                  \LieDstr{2;1}\circ_1(h\circ\LieDstr{2})
                  &+ \LieDstr{2}\circ_2(h\circ\LieDstr{2;1})  \\
                  &+ (\LieDstr{2;1}\circ_2(h\circ\LieDstr{2}))^{(12)}
                \end{aligned}\right) \right) \right)^{1-(12)}  \\
            &+ ( \LieDstr{2;1} \circ_2 ((\id-\dEnd h) \circ \LieDstr{3}) )^{(123)}  \\
              &\quad- \big( \LieDstr{2;1} \circ_2 \big((\id-\dEnd h) \circ \big(
                  - \LieDstr{2}\circ_1(h\circ\LieDstr{2})
                  + (\LieDstr{2}\circ_2(h\circ\LieDstr{2}))^{1-(12)}
                \big) \big) \big)^{(123)}  \\
            &- \LieDstr{3;2} \circ_1 ((\id-\dEnd h) \circ \LieDstr{2})  \\
              &\quad- \left(\begin{aligned}
                  \LieDstr{2;1}\circ_1(h\circ\LieDstr{2})
                  &+ \LieDstr{2}\circ_2(h\circ\LieDstr{2;1})  \\
                  &+ (\LieDstr{2;1}\circ_2(h\circ\LieDstr{2}))^{(12)}
                \end{aligned}\right) \circ_1 ((\id-\dEnd h) \circ \LieDstr{2})  \\
            &+ (\LieDstr{3;2} \circ_2 ((\id-\dEnd h) \circ \LieDstr{2}))^{1-(12)}  \\
              &\quad+ \left( \left(\begin{aligned}
                  \LieDstr{2;1}\circ_1(h\circ\LieDstr{2})
                  &+ \LieDstr{2}\circ_2(h\circ\LieDstr{2;1})  \\
                  &+ (\LieDstr{2;1}\circ_2(h\circ\LieDstr{2}))^{(12)}
                \end{aligned}\right) \circ_2 ((\id-\dEnd h) \circ \LieDstr{2}) \right)^{1-(12)}  \\
            &+ \LieDstr{3} \circ_3 ((\id-\dEnd h) \circ \LieDstr{2;1})  \\
              &\quad- \big( - \LieDstr{2} \circ_1 (h\circ\LieDstr{2})
                + (\LieDstr{2} \circ_2 (h\circ\LieDstr{2}))^{1-(12)} \big) \circ_3
                ((\id-\dEnd h) \circ \LieDstr{2;1})  \\
            &+ (\LieDstr{3;2} \circ_3 ((\id-\dEnd h) \circ \LieDstr{2}))^{(23)-(132)}  \\
              &\quad+ \left( \left(\begin{aligned}
                  \LieDstr{2;1}\circ_1(h\circ\LieDstr{2})
                  &+ \LieDstr{2}\circ_2(h\circ\LieDstr{2;1})  \\
                  &+ (\LieDstr{2;1}\circ_2(h\circ\LieDstr{2}))^{(12)}
                \end{aligned}\right) \circ_3 ((\id-\dEnd h) \circ \LieDstr{2}) \right)^{(23)-(132)}
          \end{aligned}\right) \circ (i,i,i,i)  \\
        \begin{split}
          &= \LieDstr{2;1}[\prime]\circ_1\LieDstr{3}[\prime]
          + (\LieDstr{2}[\prime]\circ_2\LieDstr{3;2}[\prime])^{1-(12)}
          + (\LieDstr{2;1}[\prime]\circ_2\LieDstr{3}[\prime])^{(123)}
          - \LieDstr{3;2}[\prime]\circ_1\LieDstr{2}[\prime]  \\
          &\quad+ (\LieDstr{3;2}[\prime]\circ_2\LieDstr{2}[\prime])^{1-(12)}
          + \LieDstr{3}[\prime]\circ_3\LieDstr{2;1}[\prime]
          + (\LieDstr{3;2}[\prime]\circ_3\LieDstr{2}[\prime])^{(23)-(132)}  .
        \end{split}
      \end{align*}
      This proves Lemma \ref{lm:htt:str}.

    \subsection{Proof of Lemma \ref{df:SS:mor}} \label{A:comp:SSmor}
      Since the components of $\SSy{f}$ are clearly skew-symmetric, the proof boils down to checking
      \cref{eq:Leib3:f1,eq:Leib3:f2,eq:Leib3:f3,eq:Leib3:f4}.  The first three are easily verified:
      \begin{align*}
        \dEnd(\SSy{f}_1) &= \partial f_1 = 0 ,  \\
        \dEnd(\SSy{f}_2) &= \frac{1}{2} \dEnd\left(f_2 - f_2^{(12)}\right)  \\
          &= \frac{1}{2} \left(
              f_1\circ\LieDstr{2} - \LieDstr{2}[\prime]\circ(f_1,f_1)
              - f_1\circ\LieDstr{2}[(12)] + \LieDstr{2}[\prime]\circ(f_1,f_1)^{(12)}
            \right)  \\
          &= \SSy{f}_1\circ\LieKstr{2} - \LieKstr{2}[\prime]\circ(\SSy{f}_1,\SSy{f}_1) ,  \\
        \dEnd(\SSy{f}_3) &= \frac{1}{6} \sum_{\sigma\in\Sy_3} (-1)^{|\sigma|} \cdot \dEnd(f_3)^\sigma
              - \frac{1}{24} \sum_{\sigma\in\Sy_3} (-1)^{|\sigma|} \cdot
                \left(\begin{aligned}
                    \dEnd(f_{2;1}) \circ_1 \LieDstr{2}
                    - \dEnd(\LieDstr{2;1}[\prime]) \circ (f_2,f_1)
                    + \LieDstr{2;1}[\prime] \circ (\dEnd(f_2),f_1)  \\
                    {}+{} \dEnd(f_{2;1}) \circ_2 \LieDstr{2}
                    - \dEnd(\LieDstr{2;1}[\prime]) \circ (f_1,f_2)
                    + \LieDstr{2;1}[\prime] \circ (f_1,\dEnd(f_2))
                  \end{aligned}\right)^\sigma  \\
          &= \frac{1}{6} \sum_{\sigma\in\Sy_3} (-1)^{|\sigma|} \cdot \left(\begin{aligned}
                &f_1\circ\LieDstr{3}
                + f_2\circ_1\LieDstr{2} - f_2\circ_2\LieDstr{2} + (f_2\circ_2\LieDstr{2})^{(12)}  \\
                &+ \LieDstr{2}[\prime]\circ(f_2,f_1) - \LieDstr{2}[\prime]\circ(f_1,f_2)
                + \LieDstr{2}[\prime]\circ(f_1,f_2)^{(12)} - \LieDstr{3}[\prime]\circ(f_1,f_1,f_1)
              \end{aligned}\right)^\sigma  \\
              &\quad- \frac{1}{24} \sum_{\sigma\in\Sy_3} (-1)^{|\sigma|} \cdot
                \left(\begin{aligned}
                    &\phantom{{}+{}}(- f_2 - f_2^{(12)}
                        + f_1\circ\LieDstr{2;1} - \LieDstr{2;1}[\prime]\circ(f_1,f_1)
                      ) \circ_1 \LieDstr{2}  \\
                    &+ (- f_2 - f_2^{(12)} + f_1\circ\LieDstr{2;1} - \LieDstr{2;1}[\prime]\circ(f_1,f_1))
                      \circ_2 \LieDstr{2}  \\
                    &- (\LieDstr{2}[\prime] + \LieDstr{2}[\prime(12)]) \circ (f_2,f_1)
                    + \LieDstr{2;1}[\prime] \circ (f_1\circ\LieDstr{2} - \LieDstr{2}[\prime]\circ(f_1,f_1),f_1)  \\
                    &- (\LieDstr{2}[\prime] + \LieDstr{2}[\prime(12)]) \circ (f_1,f_2)
                    + \LieDstr{2;1}[\prime] \circ (f_1,f_1\circ\LieDstr{2} - \LieDstr{2}[\prime]\circ(f_1,f_1))
                  \end{aligned}\right)^\sigma  \\
          &= \SSy{f}_1\circ\LieKstr{3} - \LieKstr{3}[\prime]\circ(\SSy{f}_1,\SSy{f}_1,\SSy{f}_1)  \\
            &\quad+ \frac{1}{12} \sum_{\sigma\in\Sy_3} (-1)^{|\sigma|} \cdot \left(
                2 f_2\circ_1\LieDstr{2} - 4 f_2\circ_2\LieDstr{2}
              \right)^\sigma
              + \frac{1}{24} \sum_{\sigma\in\Sy_3} (-1)^{|\sigma|} \cdot
                \left(\begin{aligned}
                    \big( f_2 + f_2^{(12)} \big) \circ_1 \LieDstr{2}  \\
                    + \big( f_2 + f_2^{(12)} \big) \circ_2 \LieDstr{2}
                  \end{aligned}\right)^\sigma  \\
            &\quad+ \frac{1}{12} \sum_{\sigma\in\Sy_3} (-1)^{|\sigma|} \cdot \left(
                2 \LieDstr{2}[\prime]\circ(f_2,f_1) - 4 \LieDstr{2}[\prime]\circ(f_1,f_2)
              \right)^\sigma
              + \frac{1}{24} \sum_{\sigma\in\Sy_3} (-1)^{|\sigma|} \cdot
                \left(\begin{aligned}
                    \big(\LieDstr{2}[\prime] + \LieDstr{2}[\prime(12)]\big) \circ (f_2,f_1)  \\
                    + \big(\LieDstr{2}[\prime] + \LieDstr{2}[\prime(12)]\big) \circ (f_1,f_2)
                  \end{aligned}\right)^\sigma  \\
            \begin{split}
              &= \SSy{f}_1\circ\LieKstr{3} + \SSy{f}_2\circ_1\LieKstr{2}
                - \SSy{f}_2\circ_2\LieKstr{2} + (\SSy{f}_2\circ_2\LieKstr{2})^{(12)}  \\
                &\quad+ \LieKstr{2}[\prime]\circ(\SSy{f}_2,\SSy{f}_1)
                - \LieKstr{2}[\prime]\circ(\SSy{f}_1,\SSy{f}_2)
                + \LieKstr{2}[\prime]\circ(\SSy{f}_1,\SSy{f}_2)^{(12)}
                - \LieKstr{3}[\prime]\circ(\SSy{f}_1,\SSy{f}_1,\SSy{f}_1)  .
            \end{split}
      \end{align*}
      A trick to showing \cref{eq:Leib3:f4} is to define
      \begin{align*}
        \SSy{f}_4 &\defeq \frac{1}{24} \sum_{\sigma\in\Sy_4} (-1)^{|\sigma|} \cdot f_4^\sigma
            + \frac{1}{48} \sum_{\sigma\in\Sy_4} (-1)^{|\sigma|} \cdot
              \left(\begin{aligned}
                  &f_{2;1} \circ_1 \LieDstr{3} - f_{2;1} \circ_2 \LieDstr{3}
                  - \LieDstr{2;1}[\prime] \circ (f_3,f_1) + \LieDstr{2;1}[\prime] \circ (f_1,f_3)  \\
                  &- f_{3;1} \circ_1 \LieDstr{2} - f_{3;1} \circ_2 \LieDstr{2}
                  - \LieDstr{3;1}[\prime] \circ (f_2,f_1,f_1) - \LieDstr{3;1}[\prime] \circ (f_1,f_2,f_1)  \\
                  &+ f_{3;2} \circ_2 \LieDstr{2} + f_{3;2} \circ_3 \LieDstr{2}
                  + \LieDstr{3;2}[\prime] \circ (f_1,f_2,f_1) + \LieDstr{3;2}[\prime] \circ (f_1,f_1,f_2)
                \end{aligned}\right)^\sigma ,
      \end{align*}
      which for degree reasons must vanish, and verify
      \begin{align*}
        0 = \dEnd(\SSy{f}_4) &= \frac{1}{24} \sum_{\sigma\in\Sy_4} (-1)^{|\sigma|} \cdot \dEnd(f_4)^\sigma  \\
            &\quad+ \frac{1}{48} \sum_{\sigma\in\Sy_4} (-1)^{|\sigma|} \cdot
              \left(\begin{aligned}
                  &\phantom{{}-{}}\dEnd(f_{2;1})\circ_1\LieDstr{3} + f_{2;1}\circ_1\dEnd(\LieDstr{3})
                  - \dEnd(\LieDstr{2;1}[\prime])\circ(f_3,f_1) + \LieDstr{2;1}[\prime]\circ(\dEnd(f_3),f_1)  \\
                  &- \dEnd(f_{2;1})\circ_2\LieDstr{3} - f_{2;1}\circ_2\dEnd(\LieDstr{3})
                  + \dEnd(\LieDstr{2;1}[\prime])\circ(f_1,f_3) - \LieDstr{2;1}[\prime]\circ(f_1,\dEnd(f_3))  \\
                  &- \dEnd(f_{3;1})\circ_1\LieDstr{2}
                  - \dEnd(\LieDstr{3;1}[\prime])\circ(f_2,f_1,f_1)
                  - \LieDstr{3;1}[\prime]\circ(\dEnd(f_2),f_1,f_1) \\
                  &- \dEnd(f_{3;1})\circ_2\LieDstr{2}
                  - \dEnd(\LieDstr{3;1}[\prime])\circ(f_1,f_2,f_1)
                  - \LieDstr{3;1}[\prime]\circ(f_1,\dEnd(f_2),f_1) \\
                  &+ \dEnd(f_{3;2})\circ_2\LieDstr{2}
                  + \dEnd(\LieDstr{3;2}[\prime])\circ(f_1,f_2,f_1)
                  + \LieDstr{3;2}[\prime]\circ(f_1,\dEnd(f_2),f_1) \\
                  &+ \dEnd(f_{3;2})\circ_3\LieDstr{2}
                  + \dEnd(\LieDstr{3;2}[\prime])\circ(f_1,f_1,f_2)
                  + \LieDstr{3;2}[\prime]\circ(f_1,f_1,\dEnd(f_2))
                \end{aligned}\right)^\sigma  \\
        &= \frac{1}{24} \sum_{\sigma\in\Sy_4} (-1)^{|\sigma|} \cdot \left(
              f_1\circ\LieDstr{4} - \LieDstr{4}[\prime]\circ(f_1,f_1,f_1,f_1)
            \right)^\sigma  \\
          &\quad- \frac{1}{24} \sum_{\sigma\in\Sy_4} (-1)^{|\sigma|} \cdot \left(
              f_2\circ_1\LieDstr{3} + 3 f_2\circ_2\LieDstr{3}
            \right)^\sigma  \\
          &\qquad+ \frac{1}{48} \sum_{\sigma\in\Sy_4} (-1)^{|\sigma|} \cdot \left(\begin{aligned}
              \big(- f_2 - f_2^{(12)} + f_1\circ\LieDstr{2;1} - \LieDstr{2;1}[\prime]\circ(f_1,f_1)\big)
                \circ_1 \LieDstr{3}  \\
              - \big(- f_2 - f_2^{(12)} + f_1\circ\LieDstr{2;1} - \LieDstr{2;1}[\prime]\circ(f_1,f_1)\big)
                \circ_2 \LieDstr{3}
            \end{aligned}\right)^\sigma  \\
          &\quad+ \frac{1}{24} \sum_{\sigma\in\Sy_4} (-1)^{|\sigma|} \cdot \left(
              f_3\circ_1\LieDstr{2} - 2 f_3\circ_2\LieDstr{2} + 3 f_3\circ_3\LieDstr{2}
            \right)^\sigma  \\
          &\qquad- \frac{1}{48} \sum_{\sigma\in\Sy_4} (-1)^{|\sigma|} \cdot \left(\begin{aligned}
              \left(\begin{aligned}
                  - f_3 - f_3^{(12)} + f_1\circ\LieDstr{3;1} - \LieDstr{3;1}[\prime]\circ(f_1,f_1,f_1)  \\
                  - f_2\circ_1\LieDstr{2;1} - \LieDstr{2}[\prime]\circ(f_{2;1},f_1)
                \end{aligned}\right) \circ_1 \LieDstr{2}  \\
              - \left(\begin{aligned}
                  - f_3 - f_3^{(12)} + f_1\circ\LieDstr{3;1} - \LieDstr{3;1}[\prime]\circ(f_1,f_1,f_1)  \\
                  - f_2\circ_1\LieDstr{2;1} - \LieDstr{2}[\prime]\circ(f_{2;1},f_1)
                \end{aligned}\right) \circ_2 \LieDstr{2}
            \end{aligned}\right)^\sigma  \\
          &\qquad+ \frac{1}{48} \sum_{\sigma\in\Sy_4} (-1)^{|\sigma|} \cdot \left(\begin{aligned}
              \left(\begin{aligned}
                  &- f_3 - f_3^{(23)} + f_1\circ\LieDstr{3;2} - \LieDstr{3;2}[\prime]\circ(f_1,f_1,f_1)  \\
                  &- f_{2;1}\circ_1\LieDstr{2} + f_2\circ_2\LieDstr{2;1} - (f_{2;1}\circ_2\LieDstr{2})^{(12)}  \\
                  &+ \LieDstr{2;1}[\prime]\circ(f_2,f_1) + \LieDstr{2}[\prime]\circ(f_1,f_{2;1})
                    + \LieDstr{2;1}[\prime]\circ(f_1,f_2)^{(12)}
                \end{aligned}\right) \circ_2 \LieDstr{2}  \\
              - \left(\begin{aligned}
                  &- f_3 - f_3^{(23)} + f_1\circ\LieDstr{3;2} - \LieDstr{3;2}[\prime]\circ(f_1,f_1,f_1)  \\
                  &- f_{2;1}\circ_1\LieDstr{2} + f_2\circ_2\LieDstr{2;1} - (f_{2;1}\circ_2\LieDstr{2})^{(12)}  \\
                  &+ \LieDstr{2;1}[\prime]\circ(f_2,f_1) + \LieDstr{2}[\prime]\circ(f_1,f_{2;1})
                    + \LieDstr{2;1}[\prime]\circ(f_1,f_2)^{(12)}
                \end{aligned}\right) \circ_3 \LieDstr{2}
            \end{aligned}\right)^\sigma  \\
          &\quad- \frac{1}{24} \sum_{\sigma\in\Sy_4} (-1)^{|\sigma|} \cdot \left(
              \LieDstr{2}[\prime]\circ(f_3,f_1) + \LieDstr{2}[\prime]\circ(f_1,f_3)
            \right)^\sigma  \\
          &\qquad- \frac{1}{48} \sum_{\sigma\in\Sy_4} (-1)^{|\sigma|} \cdot \left(
              \big(\LieDstr{2}[\prime] + \LieDstr{2}[\prime(12)]\big)\circ(f_3,f_1)
              - \big(\LieDstr{2}[\prime] + \LieDstr{2}[\prime(12)]\big)\circ(f_1,f_3)
            \right)^\sigma  \\
          &\quad- \frac{1}{24} \sum_{\sigma\in\Sy_4} (-1)^{|\sigma|} \cdot \left(
              \LieDstr{3}[\prime]\circ(f_2,f_1,f_1) - 2 \LieDstr{3}[\prime]\circ(f_1,f_2,f_1)
              + 3 \LieDstr{3}[\prime]\circ(f_1,f_1,f_2)
            \right)^\sigma  \\
          &\qquad- \frac{1}{48} \sum_{\sigma\in\Sy_4} (-1)^{|\sigma|} \cdot \left(\begin{aligned}
              \big(\LieDstr{3}[\prime] + \LieDstr{3}[\prime(12)]
                  + \LieDstr{2}[\prime]\circ_1\LieDstr{2;1}[\prime]
                \big)\circ(f_2,f_1,f_1)  \\
              + \big(\LieDstr{3}[\prime] + \LieDstr{3}[\prime(12)]
                  + \LieDstr{2}[\prime]\circ_1\LieDstr{2;1}[\prime]
                \big)\circ(f_1,f_2,f_1)
            \end{aligned}\right)^\sigma  \\
          &\qquad+ \frac{1}{48} \sum_{\sigma\in\Sy_4} (-1)^{|\sigma|} \cdot \left(\begin{aligned}
              \big(\LieDstr{3}[\prime] + \LieDstr{3}[\prime(23)]
                  + \LieDstr{2;1}[\prime]\circ_1\LieDstr{2}[\prime]
                  - \LieDstr{2}[\prime]\circ_2\LieDstr{2;1}[\prime]
                  + (\LieDstr{2;1}[\prime]\circ_2\LieDstr{2}[\prime])^{(12)}
                \big)\circ(f_1,f_2,f_1)  \\
              + \big(\LieDstr{3}[\prime] + \LieDstr{3}[\prime(23)]
                  + \LieDstr{2;1}[\prime]\circ_1\LieDstr{2}[\prime]
                  - \LieDstr{2}[\prime]\circ_2\LieDstr{2;1}[\prime]
                  + (\LieDstr{2;1}[\prime]\circ_2\LieDstr{2}[\prime])^{(12)}
                \big)\circ(f_1,f_1,f_2)
            \end{aligned}\right)^\sigma  \\
          &\quad+ \frac{1}{24} \sum_{\sigma\in\Sy_4} (-1)^{|\sigma|} \cdot \left(
              3 \LieDstr{2}[\prime]\circ(f_2,f_2)
            \right)^\sigma  \\
          &\quad+ \frac{1}{48} \sum_{\sigma\in\Sy_4} (-1)^{|\sigma|} \cdot \left(\begin{aligned}
              f_{2;1}\circ_1\big(
                  - \LieDstr{2}[\prime]\circ_1\LieDstr{2}[\prime]
                  + 2 \LieDstr{2}[\prime]\circ_2\LieDstr{2}[\prime]
                \big)  \\
              - f_{2;1}\circ_2\big(
                  - \LieDstr{2}[\prime]\circ_1\LieDstr{2}[\prime]
                  + 2 \LieDstr{2}[\prime]\circ_2\LieDstr{2}[\prime]
                \big)
            \end{aligned}\right)^\sigma  \\
          &\quad+ \frac{1}{48} \sum_{\sigma\in\Sy_4} (-1)^{|\sigma|} \cdot \left(\begin{aligned}
              \LieDstr{2;1}[\prime]\circ\left(
                  \left(\begin{aligned}
                    &f_1\circ\LieDstr{3} + f_2\circ_1\LieDstr{2} - 2 f_2\circ_2\LieDstr{2}  \\
                    &+ \LieDstr{2}[\prime]\circ(f_2,f_1) - 2 \LieDstr{2}[\prime]\circ(f_1,f_2)
                    - \LieDstr{3}[\prime]\circ(f_1,f_1,f_1)
                  \end{aligned}\right)
                  , f_1
                \right)  \\
              - \LieDstr{2;1}[\prime]\circ\left(
                  f_1,
                  \left(\begin{aligned}
                    &f_1\circ\LieDstr{3} + f_2\circ_1\LieDstr{2} - 2 f_2\circ_2\LieDstr{2}  \\
                    &+ \LieDstr{2}[\prime]\circ(f_2,f_1) - 2 \LieDstr{2}[\prime]\circ(f_1,f_2)
                    - \LieDstr{3}[\prime]\circ(f_1,f_1,f_1)
                  \end{aligned}\right)
                \right)
            \end{aligned}\right)^\sigma  \\
          &\quad- \frac{1}{48} \sum_{\sigma\in\Sy_4} (-1)^{|\sigma|} \cdot \left(\begin{aligned}
              \LieDstr{3;1}[\prime]\circ\big(
                  f_1\circ\LieDstr{2} - \LieDstr{2}[\prime]\circ(f_1,f_1), f_1, f_1
                \big)
              + \LieDstr{3;1}[\prime]\circ\big(
                  f_1, f_1\circ\LieDstr{2} - \LieDstr{2}[\prime]\circ(f_1,f_1), f_1
                \big)
            \end{aligned}\right)^\sigma  \\
          &\quad+ \frac{1}{48} \sum_{\sigma\in\Sy_4} (-1)^{|\sigma|} \cdot \left(\begin{aligned}
              \LieDstr{3;2}[\prime]\circ\big(
                  f_1, f_1\circ\LieDstr{2} - \LieDstr{2}[\prime]\circ(f_1,f_1), f_1
                \big)
              + \LieDstr{3;2}[\prime]\circ\big(
                  f_1, f_1, f_1\circ\LieDstr{2} - \LieDstr{2}[\prime]\circ(f_1,f_1)
                \big)
            \end{aligned}\right)^\sigma  \\
        \begin{split}
          &= \SSy{f}_1\circ\LieKstr{4}
            - \SSy{f}_2\circ_1\LieKstr{3} - \SSy{f}_2\circ_2\LieKstr{3} + (\SSy{f}_2\circ_2\LieKstr{3})^{(12)}
            - (\SSy{f}_2\circ_2\LieKstr{3})^{(123)} + \SSy{f}_3\circ_1\LieKstr{2} - \SSy{f}_3\circ_2\LieKstr{2}
            + (\SSy{f}_3\circ_2\LieKstr{2})^{(12)}  \\
            &\quad+ \SSy{f}_3\circ_3\LieKstr{2} - (\SSy{f}_3\circ_3\LieKstr{2})^{(23)}
            + (\SSy{f}_3\circ_3\LieKstr{2})^{(132)} - \LieKstr{2}[\prime]\circ(\SSy{f}_3,\SSy{f}_1)
            - \LieKstr{2}[\prime]\circ(\SSy{f}_1,\SSy{f}_3)
            + \LieKstr{2}[\prime]\circ(\SSy{f}_1,\SSy{f}_3)^{(12)}  \\
            &\quad- \LieKstr{2}[\prime]\circ(\SSy{f}_1,\SSy{f}_3)^{(123)}
            + \LieKstr{2}[\prime]\circ(\SSy{f}_2,\SSy{f}_2)
            - \LieKstr{2}[\prime]\circ(\SSy{f}_2,\SSy{f}_2)^{(23)}
            + \LieKstr{2}[\prime]\circ(\SSy{f}_2,\SSy{f}_2)^{(132)}  \\
            &\quad- \LieKstr{3}[\prime]\circ(\SSy{f}_2,\SSy{f}_1,\SSy{f}_1)
            + \LieKstr{3}[\prime]\circ(\SSy{f}_1,\SSy{f}_2,\SSy{f}_1)
            - \LieKstr{3}[\prime]\circ(\SSy{f}_1,\SSy{f}_2,\SSy{f}_1)^{(12)}
            - \LieKstr{3}[\prime]\circ(\SSy{f}_1,\SSy{f}_1,\SSy{f}_2)  \\
            &\quad+ \LieKstr{3}[\prime]\circ(\SSy{f}_1,\SSy{f}_1,\SSy{f}_2)^{(23)}
            - \LieKstr{3}[\prime]\circ(\SSy{f}_1,\SSy{f}_1,\SSy{f}_2)^{(132)}
            - \LieKstr{4}[\prime]\circ(\SSy{f}_1,\SSy{f}_1,\SSy{f}_1,\SSy{f}_1) .
        \end{split}
      \end{align*}
      This concludes the proof of \Cref{df:SS:mor}.


  \printbibliography



\end{document}
